\documentclass[letterpaper,11pt,reqno]{amsart}
\usepackage[foot]{amsaddr}
\RequirePackage[utf8]{inputenc}
\usepackage[portrait,margin=3cm]{geometry}
\usepackage{latexsym,amssymb,amsmath,amsfonts,amsthm,mathrsfs,xcolor}
\usepackage{bm}
\makeatletter

\renewcommand\subsubsection{\@secnumfont}{\bfseries}%
\renewcommand\subsubsection{\@startsection{subsubsection}{3}
  \z@{.5\linespacing\@plus.7\linespacing}{-.5em}%
  {\normalfont\bfseries}}
  
  \makeatother
  
  \makeatletter
\@namedef{subjclassname@2020}{%
  \textup{2020} Mathematics Subject Classification}
\makeatother

\newtheorem{theorem}{Theorem}

\newtheorem{definition}{Definition}
\newtheorem{lemma}{Lemma}[section]
\newtheorem{remark}{Remark}
\newtheorem{prop}{Proposition}
\newtheorem{corollary}{Corollary}
\newtheorem*{conj}{Conjecture}

\def\beq{ \begin{equation} }
\def\eeq{ \end{equation} }
\def\mn{\medskip\noindent}

\def\ep{\epsilon}

\def\square{\vcenter{\vbox{\hrule height .4pt
  \hbox{\vrule width .4pt height 5pt \kern 5pt
        \vrule width .4pt} \hrule height .4pt}}}

\def\RR{\mathbb{R}}

\def\ZZ{\mathbb{Z}}
\def\FF{\mathcal{F}}
\def\GG{\mathcal{G}}
\def\AA{\mathcal{A}}
\def\EE{\mathcal{E}}
\def\CC{\mathcal{C}}
\def\tildeK{\mathcal{K}}
\def\LL{L}
\def\NN{\mathcal{N}}

\def\WW{\mathcal{W}}

\def\KK{\mathcal{K}}

\def\ep{\varepsilon}
\def\Cay{\mathrm{Cay}}

\def\dist{\mathrm{dist}}
\definecolor{darkblue}{rgb}{0,0,0.6}

\def\Cay{\text{Cay}}
\def\Diam{\mathrm{Diam}}
\def\P{\mathbb{P}}
\def\E{\mathbb{E}}

\def\HH{\mathcal{H}}
\def\id{\mathrm{id}}
\def\1{\mathbf{1}}
\def\0{\textbf{0}}
\def\Unif{\mathrm{Unif}}
\def\typ{\mathrm{\mathbf{typ}}}

\def\JJ{\mathcal{J}}
\def\gcd{\mathrm{gcd}}

\newcommand{\abs}[1]{|#1|}
\newcommand{\Quad}[1]{\quad\text{#1}\quad}

\newcommand{\floor}[1]{\lfloor #1 \rfloor}
\newcommand{\ceil}[1]{\lceil #1 \rceil}

\newcommand{\inprod}[1]{\langle #1 \rangle}

\usepackage{graphicx}
\graphicspath{%
    {converted_graphics/}
    {/}
}
\begin{document}

\title{Cutoff for random Cayley graphs of nilpotent groups}
\date{}
\author[Hermon]{Jonathan Hermon$^1$}
\email{jonathan.hermon@gmail.com}
\address{$^1$Department of Mathematics, University of British Columbia, BC, Canada}
\author[Huang]{Xiangying Huang$^2$}
\email{zoehuang@unc.edu}
\address{$^2$Department of Statistics and Operations Research, 304 Hanes Hall, University of North Carolina at Chapel Hill, US}

\begin{abstract}
 We consider the random Cayley graphs of a sequence of finite nilpotent groups of diverging sizes $G=G(n)$, whose ranks and nilpotency classes are uniformly bounded. For some $k=k(n)$ such that $1\ll\log k \ll \log |G|$, we pick a random set of generators $S=S(n)$ by sampling $k$ elements $Z_1,\ldots,Z_k$ from $G$ uniformly at random with replacement, and set $S:=\{Z_j^{\pm 1}:1 \le j\le k \}$. We show that the simple random walk on Cay$(G,S)$ exhibits cutoff with high probability. 

 Some of our results apply to a general set of generators. Namely, we show that there is a constant $c>0$, depending only on the rank and the nilpotency class of $G$, such that for all symmetric sets of generators $S$ of size at most $ \frac{c\log |G|}{\log \log |G|}$, the spectral gap and the $\varepsilon$-mixing time of the simple random walk $X=(X_t)_{t\geq 0}$ on Cay$(G,S)$ are asymptotically the same as those of the projection of $X$ to the abelianization of $G$, given by $[G,G]X_t$. In particular, $X$ exhibits cutoff if and only if its projection does.
\end{abstract}
\maketitle

\section{Introduction}
\subsection{Motivation and Objectives of Paper}

We examine the random walk (RW) $X=(X_t)_{t \ge 0}$ on a \textit{Cayley graph} $\Gamma:=\mathrm{Cay}(G,S)$ of a finite \textit{nilpotent} group $G$ w.r.t.\ a symmetric set of generators $S$. We are interested in the asymptotic behavior of the mixing time and the spectral gap of the walk as $|G| \to \infty$, while the \textit{step} (also called \textit{nilpotency class}) of $G$ and its \textit{rank} (minimal size of a set of generators), denoted respectively by $L=L(G)$ and $r=r(G)$, remain bounded. 

We also investigate the occurrence of the cutoff phenomenon for the walk, which in the above setup can only occur when $|S| \gg 1$ (i.e., $|S|$ diverges as $|G| \to \infty$). In particular, we prove that the walk exhibits cutoff with high probability when $S$ is obtained by picking $k$ elements of $G$, $Z_1,\ldots,Z_k$, uniformly at random, with replacement, and then setting $S:= \{ Z_i^{\pm 1} : 1 \le i \le k  \}$ under the necessary condition that $1 \ll \log k \ll \log |G|$. 

The overarching theme of this work is that in a certain strong quantitative sense the mixing behavior of the walk is governed by that of the projection of the walk to the \emph{abelianization} of $G$, denoted by $G_{\mathrm{ab}}:= G/[G,G]$, where $[G,G]$ is the \textit{commutator subgroup} of $G$. The projected walk $Y_t:=[G,G]X_t$ is a random walk on the projected Cayley graph $\Gamma_{\mathrm{ab}}:=\mathrm{Cay}(G_{\mathrm{ab}},S_{[G,G]})$, where $S_{[G,G]}:=\{[G,G]s : s \in S \}$ is the projection of $S$ to $G_{\mathrm{ab}}$. 

As will be stated precisely in Theorem \ref{reduction_ab} and Corollary \ref{t_rel}, for arbitrary symmetric $S$ such that $|S| \le \frac{c \log |G|}{\log \log |G|}$ for some constant $c=c(r,L)>0$ depending only on the rank and step of $G$,
the relaxation times (inverses of the spectral gaps) of $\Gamma=\mathrm{Cay}(G,S)$ and $\Gamma_{\mathrm{ab}}=\mathrm{Cay}(G_{\mathrm{ab}},S_{[G,G]})$ are equal:
$$t^G_{\mathrm{rel}}=t^{G^{\mathrm{ab}}}_{\mathrm{rel}},$$
and the total variation mixing times of the random walks on $\Gamma$ and $\Gamma_{\mathrm{ab}}$ satisfy
$$t_{\mathrm{mix}}^{G_{\mathrm{ab}}}(\ep)\leq t_{\mathrm{mix}}^G(\ep) \leq t_{\mathrm{mix}}^{G_{\mathrm{ab}}}(\ep-\delta)$$
for $\ep\in(0,1)$ and $\delta=\delta(|G|):=|G|\exp \left(- \left( \log |G| \right)^L \right)$.
In other words, in order to determine the mixing time or the occurrence of cutoff for the walk on $\Gamma$, it suffices to do so for $\Gamma_{\mathrm{ab}}$. 

In the case that $S= \{ Z_i^{\pm 1} : 1 \le i \le k  \}$, where $Z_1,\ldots,Z_k$ are i.i.d.\ and uniformly distributed over $G$, we extend the analysis to the regime $1 \ll \log k \ll \log |G|$ and prove that  the random walk exhibits cutoff with high probability around time $\max \{t_0(|G_{\mathrm{ab}}|,k),\log_{k}|G| \}$, where $t_0(n,k)$ is defined as the time at which the entropy of the rate 1 continuous time simple random walk on $\mathbb{Z}^k$ is $n$.

\subsubsection{Motivations}\label{motivation}
To motivate our investigation, let us consider for now the scenario where $G$ is a finite group, $k$ is an integer (allowed to depend on $G$) and $\GG_k$ denotes the Cayley graph of $G$ with respect to $k$ independently and uniformly chosen random generators. These are elements of $G$ that will generate $G$ with high probability when $k$ is sufficiently large and hence, with slight abuse of language, will be referred to as generators of $G$. We consider values of $k$ with $1 \ll \log k \ll \log |G|$ for which $\GG_k$ is connected with high probability, that is, with probability tending to 1 as $|G|\to \infty$.  

\mn
\textit{Universality of cutoff.}

Aldous and Diaconis \cite{aldous1986shuffling} introduced the term ``cutoff phenomenon", which describes the phenomenon where the total variation distance (TV) between the distribution of a random walk and its equilibrium distribution sharply decreases from nearly 1 to nearly 0 within a time interval  of smaller order than the mixing time. The material in this paper is motivated by their conjecture on the ``universality of cutoff" for the RW on the random Cayley graph $\GG_k$ given in \cite{aldous1986shuffling}.

\begin{conj}[Aldous and Diaconis, \cite{aldous1986shuffling}] 
For any group $G$, if $k \gg \log |G|$ and $\log k \ll \log |G|$,
then the random walk on $\GG_k$ exhibits cutoff with high probability. 
\end{conj}

Additionally, a secondary aspect of the conjecture suggests that the cutoff time does not rely on the algebraic structure of the group, but rather it can be expressed solely as a function of $k$ and $|G|$.

This conjecture has sparked a substantial body of research, see e.g., \cite{ dou1996enumeration, dou1992studies,hildebrand1994random, hildebrand2005survey, hough2017mixing, roichman1996random,wilson1997random}. The curious reader is referred to Section 1.3.1 in \cite{hermon2102cutoff} for a detailed exposition on the literature regarding the progress on the Aldous-Diaconis university conjecture.  In this context, we provide a more condensed overview of literature related to this conjecture, which serves as motivations for our work.

In  \cite{dou1996enumeration,dou1992studies}, Dou and Hildebrand confirmed the Aldous-Diaconis universality conjecture for all abelian groups. Additionally, their upper bound on the mixing time holds true for all groups. Furthermore, when $\log k \gg \log \log |G|$, this upper bound matches the trivial diameter lower bound of $\log_k |G|$, confirming the aforementioned secondary aspect of the conjecture.  Hermon and and Olesker-Taylor \cite{hermon2102cutoff,hermon2018supplementary,hermon2019further,hermon2021geometry} extend this conjecture to the regime $1\ll k\lesssim \log|G|$ for abelian groups, establishing cutoff under the condition $k- r(G) \gg 1$, where $r(G)$ is the minimal size of a generating subset of $G$. Moreover, when $k-r(G)\asymp k\gg 1$, the cutoff time is given by
\beq\label{cutoff_abelian}
\max\{ \log_k|G|, t_0(k,|G|)\},
\eeq
where $t_0(k,|G|)$ is the time at which the entropy of the rate 1 random walk $W$ on $\ZZ^k$ is $\log |G|$. Due to this definition $t_0(k,|G|)$ is also referred to as the \emph{entropic time} (see Definition \ref{entropic_time}). Their work confirms that for abelian groups the cutoff time of RW only depends on $|G|$ and $k$.

Building upon this point, the next natural step is to explore the mixing behavior of random walks on nilpotent groups (see Section \ref{RW_unipotent} for a brief overview of the literature concerning this topic). Hermon and Olesker-Taylor \cite{hermon2019cutoff} study two canonical families of nilpotent groups: the $d\times d$ unit-upper triangular matrices $U_{m,d}$ with entries in $\ZZ_m$ and the $d$-dimensional Heisenberg group $H_{m,d}$ over $\ZZ_m$ where $m\in \mathbb{N}$ (the results hold under certain assumptions on $m$ depending on the regimes of $k$). Let $G$ be either $U_{m,d}$ or $H_{m,d}$ and $G_{\mathrm{ab}}:=G/[G,G]$ its abelianization. They prove that for $1\ll \log k \ll \log |G|$ the random walk on the Cayley graph $\GG_k$ exhibits cutoff with high probability at time 
\beq\label{cutoff_characterization}
\max\{ \log_k|G|, t_0(k,|G_{\mathrm{ab}}|)\},
\eeq 
where $t_0(k,|G_{\mathrm{ab}}|)$ is the time at which the entropy of the rate 1 random walk $W$ on $\ZZ^k$ is $\log |G_{\mathrm{ab}}|$. We compare \eqref{cutoff_characterization} with \eqref{cutoff_abelian}, the latter of which gives the characterization of cutoff time for abelian groups. Indeed, $t_0(k,|G_{\mathrm{ab}}|)$ can be interpreted as the time at which the projection of the walk on $G$ onto the abelianization $G_{\mathrm{ab}}$ exhibits cutoff. This characterization of the cutoff time for random walks on $U_{m,d}$ and $H_{m,d}$ raises a natural question: does \eqref{cutoff_characterization} characterize the cutoff time for random walks on nilpotent groups in general? This question will be explored in our investigation.

\bigskip

Another natural extension of the current research on random walks on groups involves extending the choice of generators. Rather than requiring the $k$ generators to be chosen independently and uniformly at random from the group $G$, the aim is to advance our understanding to encompass scenarios involving arbitrary choices of generators.

More often than not, the analysis of the mixing of the random walk heavily replies on the specific selection of generators. For example, there is a line of research focusing on understanding the mixing properties of random walks on the unit upper triangular matrix group $U_{p,d}$ with $p$ prime, wherein the set of generators is either $\{I \pm E_{i,i+1}: 1 \leq i \leq d-1\}$ or $\{I +a E_{i,i+1}: 1 \leq i \leq d-1, a\in \ZZ_p\}$, where $E_{i,j}$ represents the $d\times d$ matrix with 1 at the entry $(i, j)$ and 0 elsewhere. See Section \ref{RW_unipotent} for an overview of existing research. The current analysis critically hinges on the fact that an operation $I+aE_{i,i+1}$ corresponds to a row addition/subtraction, allowing the decomposition of the walk's mixing behavior into the first row and the remaining part of the matrix, the latter of which can be regarded as a $(d-1)\times (d-1)$ matrix. However, such methodologies become inapplicable when dealing with arbitrary generators. It is hence of interest to develop techniques that enable the study of mixing properties of random walks on Cayley graphs using a wider range of prescribed generators.

\mn
\textit{Leading role of the abelianization in random walk mixing.}

When $G$ is a nilpotent group and $S$ is a symmetric set of generators consisting of i.i.d. uniform random elements of $G$, we shall see in Theorem \ref{cutoff_iid} that the mixing time of the RW is determined by the abelianization $G_{\mathrm{ab}}$ of $G$, given by the expression \eqref{cutoff_characterization}.

When the generator set is predetermined, in many instances it has also been demonstrated that the mixing time of random walks on groups is primarily determined by the abelianization of the group. In a recent and notable study, Diaconis and Hough \cite{diaconis2021random} introduced a novel approach to establishing a central limit theorem for random walks on unipotent matrix groups driven by a probability measure $\mu$, under certain general constraints. It is worth noting that their methodology applies to various choices of generators. For unit-upper triangular matrices $U_{p,d}$, it has been established that an individual coordinate on the $k$-th diagonal mixes in order $p^{2/k}$ steps, implying the leading role of abelianization (which corresponds to the first diagonal) in the mixing process of this random walk.

Nestoridi and Sly \cite{nestoridi2020random} studied the mixing behavior of the rate 1 RW on $U_{m,d}$ under the canonical set of generators $\{ I\pm E_{i,i+1}: i\in [d-1]\}$. In their analysis, it is proved that the mixing time of the RW on $U_{m,d}$ is bounded by $O(m^2 d \log d + d^2 m^{o(1)})$, where the former term vaguely characterizes the mixing behavior of the RW on the abelianization. This observation becomes clearer when we consider the projected walk on the abelianization, which, under the canonical set of generators, can be viewed as a product chain on $\ZZ_m^{d-1}$ and thus has mixing time of order $m^2d\log d$. Essentially, when $m$ is considerably larger than $d$, this upper bound is predominantly dictated by the mixing on the abelianization.

One might naturally inquire about the extent to which the abelianization dictates the mixing time of the RW on a general group. In addition, there is interest in explicitly identifying the dependence of the mixing time on the abelianization, which leads to our next motivation.

\mn
\textit{``Entropic time paradigm".}

As previously discussed, although the entropic time is the mixing time for ``most" choice
of generators  (when $1\ll k\lesssim \log|G|$) for abelian groups and nilpotent groups, finding an explicit choice of generators which gives rise to cutoff at the entropic time
is still open --- even for the cyclic group of prime order. Part of our motivation is to understand the extent to which this paradigm applies for a given set of generators.

It is worth pointing out that for general non-random choice of generators, the cutoff time is not necessarily given by the entropic time. For instance, Hough \cite[Theorem 1.11]{hough2017mixing} shows that for the cyclic group $\ZZ_p$ of prime order the choice of generators $S:=\{0\}\cup \{ \pm 2^i: 0\leq i\leq \floor{\log_2 p}-1\}$, which he describes as ``an approximate embedding of the classical hypercube walk into the cycle", gives rise to a random walk on $\ZZ_p$ that ehxibits cutoff, where the cutoff time is not the entropic time.

\mn
\textit{Mixing under minimal sets of generators.}

A minimal set of generators is a set of elements that generates the group which is minimal in terms of size. For $p$-groups, it is known that the minimal sets of generators can be described by the Frattini subgroup $\Phi=\Phi(G)$ in the sense that any set $\{x_1,\dots,x_r\}\subseteq G$ such that the cosets $\{\Phi x_i: 1\leq i\leq r\}$ form a basis of $G/\Phi$ gives a generating set of $G$. See, e.g.,  Diaconis and Saloff-Coste \cite[Section 5.C]{diaconis1994moderate}. A random walk supported on the minimal set of generators is thus referred to as a Frattini walk. Examples of such walks are discussed in Section 5.C of \cite{diaconis1994moderate}. For the Heisenberg group $H_{p,3}$ with prime $p$, it can be shown that all minimal sets of generators are equivalent from a group theory approach. 

Diaconis and Saloff-Coste additionally remarked that based on their experience with the circle and symmetric group, if the number of generators is fixed, most sets of generators should lead to the same convergence rate for the random walk.  Motivated by these examples, they ask the following open question (see Remark 2 on Page 23 of \cite{diaconis1994moderate}): to what extent does the choice of generators effect the mixing behavior?

We give a neat partial answer to this question in Theorem \ref{maximal_class_group}. 
Suppose that $G$ is a $p$-group with $G_{\mathrm{ab}} \cong \mathbb{Z}_{p^{\alpha_1}} \oplus \cdots \oplus \mathbb{Z}_{p^{\alpha_r}}$ or that $G_{\mathrm{ab}} \cong \mathbb{Z}_m^r$ for some $m \in \mathbb{N}$. Under very mild assumptions on the rank and step of $G$, for all minimal (symmetric) sets of generators, the corresponding mixing times on $G$ are the same up to smaller order terms and the corresponding relaxation times are the same.

\subsubsection{Objectives}\label{objectives}

Motivated by the questions discussed in the preceding section, our primary focus in this paper is as follows.

\mn
\textbf{(i) Study the random walk on $\GG_k$ for general nilpotent groups.} Expanding upon the current understanding of random walks on groups, our goal is to establish cutoff for random walks on $\GG_k$ when $G$ is a nilpotent group and $1 \ll \log k \ll \log |G|$. In particular, we are interested in a general characterization of the cutoff time. An important  implication of the findings in \cite{hermon2019cutoff} is that for certain regimes of $k$, the cutoff time for $G=U_{m,d}$ (or $G=H_{m,d}$) does not depend only on $k$ and $|G|$. Nevertheless, the only additional information required to determine the cutoff times for these two examples is the size of the abelianization, as indicated by \eqref{cutoff_characterization}. We hope to generalize the characterization of the cutoff time in \eqref{cutoff_characterization} to \emph{general} nilpotent groups.

We thank P\'eter Varj\'u for suggesting us the problem of extending the analysis from \cite{hermon2019cutoff} to other\footnote{Namely, the case that $G$ is step 2 and $G_{\mathrm{ab}} \cong \mathbb{Z}_p^{r}$, i.e., $G_{\mathrm{ab}}$ is elementary abelian. We also wish to thank him for providing invaluable insights regarding how certain components of the argument from \cite{hermon2019cutoff} could be interpreted in terms of the general theory of nilpotent groups.} nilpotent groups.

\mn
\textbf{(ii) Develop techniques applicable when the generators are chosen arbitrarily.} 
As indicated by previous discussions, the mixing time of random walks on a group under various choices of generators is largely determined by the abelianization of the group. We aim to explore the extent to which this leading role of the abelianization holds in a broader context.

In essence, our objective is to develop techniques for studying the mixing properties of random walks, applicable not only under arbitrary choices of generators but also for general groups, without dependence on specific group structures.

\subsection{Definitions and Notation}\label{setting&notation}
We give the precise definitions of the Cayley graph on $G$ and the random walk on Cayley graphs.

Let $G$ be a nilpotent group with lower central series
$$G = G_1 \trianglerighteq G_2 \trianglerighteq \cdots  \trianglerighteq G_{\LL}\trianglerighteq G_{\LL+1}=\{\id\}$$ 
where $G_{i+1} :=[G_i,G]=\inprod{\{ [g,g']: g\in G_i, g'\in G\}}$. In particular, $G_2=[G,G]$ denotes the commutator subgroup of $G$. We also denote by 
$G_{\mathrm{ab}}:= G/G_2$ the \emph{abelianization} of $G$. The \emph{rank} of a nilpotent group $G$, denoted by $r=r(G)$, is the smallest integer $r$ such that $G$ can be generated by a set containing $r$ elements of $G$ and their inverses.  The number $L=L(G)$ is called the \emph{step} (or the \emph{nilpotency class}) of $G$, i.e., $|G_L| > 1$ and $|G_{L+1}| = 1$.

For a finite group $G$, let $S\subseteq G$ be a symmetric subset, i.e., $s\in S$ if and only if $s^{-1}\in S$.  We will refer to $S$ as the set of generators when $S$ generates $G$. The \textit{undirected} Cayley graph of $G$ generated by $S$ is defined as follows.

\begin{definition}[Cayley multi-graph generated by a set of generators]
Fix a symmetric set $S:=\{ s_i^{\pm 1}: i\in[k]\}\subseteq G$ of generators.  Let $\Cay(G,S)$ denote the (right) Cayley multi-graph generated by $G$ with respect to $S$, where the vertex set $\mathbb{V}:=\{ g: g\in G\}$ and the edge set  $\mathbb{E}:=\{ \{g,gs\}: g\in G, s\in S\}$. We allow parallel edges and self loops (if $\id\in S$) so that the Cayley graph $\Cay(G,S)$ is regular with degree $2k$. 
\end{definition}

\mn
\textbf{Random walk on Cayley graphs.} We will consider the \textit{undirected} random walk $X_t$ on the Cayley graph $\Cay(G,S)$ which jumps at rate 1, where $S:=\{s_i^{\pm 1}: i\in [k]\}$. Let $\{\sigma_i\}_{i\in\mathbb{N}}$ be an i.i.d. sequence of indices uniformly sampled from $[k]$, and let $\{\eta_i\}_{i\in\mathbb{N}}$ be an i.i.d. sequence of signs uniformly sampled from $\{\pm 1\}$. At the $i$-th jump, the generator $s_{\sigma_i}^{\eta_i}$ is applied to the walk $X$ in the sense that we multiply $s_{\sigma_i}^{\eta_i}$ to the right of the current location of $X$. That is, the random walk $X$ can be written as a sequence
$$X=\prod_{i=1}^N s^{\eta_i}_{\sigma_i}=s_{\sigma_1}^{\eta_1}s_{\sigma_2}^{\eta_2}\cdots s_{\sigma_N}^{\eta_N},$$
where $N:=N(t)$ is the number of steps taken by $X$ by time $t$ and $s^{\eta_i}_{\sigma_i}$ denotes the $i$-th step taken by the random walk with $\sigma_i\in [k],\eta_i\in\{\pm 1\}$.

\mn
\textbf{Notation.} Throughout the paper, we use standard asymptotic notation: ``$\ll$" or ``$o(\cdot)$" means ``of smaller order"; ``$\lesssim$" or ``$\mathcal{O}(\cdot)$" means ``of order at most"; ``$\asymp$" means ``of the same order"; ``$\eqsim$" means ``asymptotically equivalent". We will abbreviate ``with high probability"  by \textit{whp}.

\mn
\textbf{Assumptions.} Throughout the paper, we will let $G$ be a finite nilpotent group of step $\LL\geq 2$ and rank $r$ where $r, \LL\asymp 1$.

\subsection{Overview of Main Results}\label{main_results}

We focus on the mixing behavior of the random walk on a Cayley graph $\Cay(G,S)$ of a finite nilpotent group $G$ with a symmetric generator set $S=\{s_i^{\pm 1}: i\in [k]\}$. We consider the
limit as $|G| \to \infty$ under the assumption that $1 \ll \log k \ll \log |G|$. The condition
$1 \ll \log k \ll \log |G|$ is necessary for the random walk to exhibit cutoff on $\Cay(G,S)$ for all nilpotent $G$, see the remark below.

\begin{remark}\label{condition_for_cutoff}
For any choice of generators, it was established by Diaconis and Saloff-Coste \cite{diaconis1994moderate} that there is no cutoff when $k\asymp 1$ for all nilpotent groups, which is a class of groups that satisfies their concept of \textit{moderate growth}. The interested reader can find a short exposition of their argument in \cite[\S 4]{hermon2019further}. When $\log k \asymp \log |G|$ and with $k$ i.i.d. uniform generators, there is no cutoff for all groups, see \cite[\S 7.2]{hermon2102cutoff}. Dou \cite[Theorems 3.3.1 and 3.4.7]{dou1992studies} establishes a more general result for $\log k \asymp \log |G|$.
\end{remark}

\subsubsection{Cutoff for Random Walks on Nilpotent Groups}

We use standard notation and definitions for mixing and cutoff, see e.g. \cite[\S 4 and \S 18]{wilmer2009markov}.
\begin{definition}
A sequence $(X_N)_{N\in\mathbb{N}}$ of Markov chains is said to exhibit cutoff if there exists a sequence of times $(t_N)_{N\in\mathbb{N}}$ with 
$$\limsup_{N\to\infty} d_N((1-\ep)t_N)=1 \quad \text{and}\quad \limsup_{N\to\infty} d_N((1+\ep)t_N)=0 \quad \text{ for all }\ep\in (0,1),$$
where $d_N (\cdot)$ is the TV distance of $X_N(\cdot)$ from its equilibrium distribution for each $N\in\mathbb{N}$.

 We say that a RW on a sequence of random graphs $(H_N)_{N\in\mathbb{N}}$ exhibits cutoff around time $(t_N)_{N\in\mathbb{N}}$ whp if, for all fixed $\ep$, in the limit $N \to \infty$, the TV distance at time $(1+\ep)t_N$ converges in distribution to 0 and at time $(1-\ep)t_N$ to 1, where the randomness is over $H_N$.

\end{definition}
In other words, $(X_N)_{N\in\mathbb{N}}$ is said to exhibits cutoff when the TV distance of the distribution of the chain from
equilibrium drops from close to 1 to close to 0 in a short time interval of smaller order than the mixing time.

\medskip
As briefly discussed in Section \ref{motivation}, there has been considerable interest in studying the cutoff behavior of random walks on groups. Our goal is to generalize the characterization of cutoff time as $\max\{ \log_k|G|, t_0(k,|G_{\mathrm{ab}}|)\}$ to general nilpotent groups (for random i.i.d. generators). 

We now give the formal definition of the entropic time $t_0:=t_0(k,|G_{\mathrm{ab}})$ and the proposed mixing time. 
\begin{definition}\label{entropic_time}
(i) Let $t_0(k,N)$ be the time at which the entropy of the rate 1 random walk $W$ on $\ZZ^k$ is $\log N$. We refer to $t_0(k, |G_{\mathrm{ab}}|)$ as the entropic time.\\
(ii) Define $t_*(k,G):=\max\{ t_0(k, |G_{\mathrm{ab}}|), \log_k|G|\}$. We refer to $t_*(k,G)$ as the cutoff time or the mixing time.
\end{definition}
The entropic time $t_0(k, |G_{\mathrm{ab}}|)$ is identified as the cutoff time for the projected random walk $Y_t:=G_2 X_t$ on $G_{\mathrm{ab}}$, see \cite{hermon2102cutoff}, which is naturally a lower bound on the mixing time of the RW $X_t$ on $G$. To offer insight into the definition of the cutoff time, note that we need to run the RW sufficiently long to ensure that all elements of the group can be reached with reasonable probability, which leads to a lower bound of $\log_k|G|$. 

Our first result establishes cutoff around time $t_*(k,G)$ for the random walk $X$ on $\Cay(G,S)$ where $S$ consists of i.i.d. uniform generators.
 \begin{theorem}\label{cutoff_iid}
 Let $G$ be a finite nilpotent group with $r(G), L(G)\asymp 1$. Let $S=\{ Z_i^{\pm 1}: i\in [k]\}$ with $Z_1,\dots,Z_k\overset{iid}{\sim} \Unif(G)$. Assume $1\ll \log k \ll \log|G|$. As $|G|\to \infty$, the  random walk on $\Cay(G,S)$ exhibits cutoff with high probability at time $t_*(k,G)$, which is the cutoff time defined in Definition \ref{entropic_time}.
 \end{theorem}

\subsubsection{Random Walk on Non-random Cayley Graphs: Reduction to Abelianization}

For a nilpotent group $G$ and any symmetric set of generators $S\subseteq G$ whose size satisfies an upper bound, we show that the mixing time of the random walk on $G$ is completely determined (up to smaller order terms) by the mixing time of the projected walk on $G_{\mathrm{ab}}$.

\begin{theorem}\label{reduction_ab}
Let $G$ be a finite nilpotent group such that $r(G),L(G)\asymp 1$ and $S\subseteq G$ be a symmetric set of generators. Suppose $ |S| \leq \frac{\log |G|}{ 8\LL r^\LL\log\log|G|}$. For any fixed $\ep\in (0,1)$ and $\delta\in ( 0,\ep)$ we have 
$$t_{\mathrm{mix}}^{G_{\mathrm{ab}}}(\ep)\leq t_{\mathrm{mix}}^G(\ep) \leq t_{\mathrm{mix}}^{G_{\mathrm{ab}}}(\ep-\delta)$$
when $|G|$ is sufficiently large (more precisely, when $|G|\exp(-(\log|G|)^L)\leq \delta$).
\end{theorem}

\begin{remark}
The assumption $|S| \leq \frac{\log |G|}{ 8\LL r^\LL\log\log|G|}$ is to guarantee that $\Diam_S(G_2)$ is of smaller order than $\Diam_S(G_{\mathrm{ab}})$ so that the mixing of $X_t$ is governed by its projected walk onto $G_{\mathrm{ab}}$. With more specific knowledge on the structure of $G$, one can expect to obtain a much less stringent constraint on $S$. Also see Remark \ref{dependence_rank_step}.
\end{remark}

As a direct consequence of the proof of Theorem \ref{reduction_ab}, we establish that under the same conditions,  the spectral gap of the random walk on $G$ is likewise determined by the spectral gap of its projection onto $G_{\mathrm{ab}}$.
\begin{corollary}\label{t_rel}
Let $t^G_{\mathrm{rel}}$ and $t^{G^{\mathrm{ab}}}_{\mathrm{rel}}$ be the relaxation time of the walk $X_t$ and $Y_t=G_2 X_t$ respectively. Then
$$t^{G^{\mathrm{ab}}}_{\mathrm{rel}}\leq t^G_{\mathrm{rel}} \leq \max\{ t^{G^{\mathrm{ab}}}_{\mathrm{rel}}, |S| \cdot \Diam_S(G_2)^{2}\}.$$
In particular, when $ |S| \leq \frac{\log |G|}{ 8\LL r^\LL\log\log|G|}$ we have $t^G_{\mathrm{rel}}=t^{G^{\mathrm{ab}}}_{\mathrm{rel}}$.
\end{corollary}

As a consequence of the above results, we can see that for a class of nilpotent groups $G$ whose abelianization has a unique representation, with a symmetric set of generators $S$ of minimal size, the mixing time and the relaxation time (inverse of the spectral gap) of the random walk do not depend on the choice of $S$. In this case, the choice of generators do not effect the mixing behavior. This provides a partial answer to the open question posed in Section \ref{motivation}.

\begin{theorem}\label{maximal_class_group}
Suppose $G$ is a nilpotent group with rank $r$ and step $L$ such that either (i) $G_{ab}\cong\ZZ^r_m$ where $m\in \mathbb{N}$ or (ii) $G$ is a $p$-group. Suppose the rank and step satisfy $Lr^{L+1}\leq \frac{\log|G|}{16 \log\log|G|}$. For any symmetric set of generators $S\subseteq G$ of minimal size and any given $\ep\in(0,1)$, the mixing time $t^{G,S}_{mix}(\ep)$ is the same up to smaller order terms, and the relaxation time $t_{\mathrm{rel}}^{G,S}$ is the same. 
\end{theorem} 

\medskip

The mixing property of the random walk $X_t$ on the Cayley graph of $G$ is closely related to that of the projected random walk on the Cayley graph of $G_{\mathrm{ab}}$. More precisely, denoting by $Y_t:=G_2X_t$ the projected RW on $G_{\mathrm{ab}}$ and starting with the walk $X_t$ being uniform over $G_2$, one can observe (see Lemma \ref{reduction_TV}) that
$$\| \P_{\pi_{G_2}}(X_t=\cdot)-\pi_G\|_{\mathrm{TV}}=\| \P_{G_2}(Y_t=\cdot)-\pi_{G_{\mathrm{ab}}}\|_{\mathrm{TV}}.$$
As suggested by the following triangle inequality
$$\| \P_{\id}(X_t=\cdot)-\pi_G\|_{\mathrm{TV}}\leq \| \P_{\pi_{G_2}}(X_t=\cdot)-\pi_G\|_{\mathrm{TV}}+ \| \P_{\id}(X_t=\cdot)- \P_{\pi_{G_2}}(X_t=\cdot)\|_{\mathrm{TV}},
$$
if the total variation distance between $\P_{\id}(X_t=\cdot)$ and $\P_{\pi_{G_2}}(X_t=\cdot)$ can be nicely controlled then the mixing property of $X_t$ is primarily characterized by the mixing of its projection on the abelianization, which we refer to as \textit{the reduction to abelianization}. 

We will prove in Lemma \ref{commutator_mixing} that indeed $ \| \P_{\id}(X_t=\cdot)- \P_{\pi_{G_2}}(X_t=\cdot)\|_{\mathrm{TV}}$ decays exponentially fast in time with rate at least $(|S| \cdot\Diam_S(G_2)^2)^{-1}$, where $\Diam_S(G_2)$ is the diameter of $G_2$ in $\Cay(G,S)$. This provides a quantitive criterion to determine when the mixing of the walk $X_t$ is governed by its projection onto the abelianization. In particular, if the mixing of the projected walk $Y_t$ occurs after $\| \P_{\id}(X_t=\cdot)- \P_{\pi_{G_2}}(X_t=\cdot)\|_{\mathrm{TV}}$ had become vanishingly small then the mixing time of $X_t$ is roughly that of $Y_t$.

Due to the well known connection between the mixing time and the diameter of the graph, see, e.g., \cite[Proposition 13.7]{lyons2017probability}, for our purpose it is sufficient to prove $\Diam_S(G_2)$ is small enough compared to $\Diam_S(G_{\mathrm{ab}})$. Section \ref{Geo_Cayley} is devoted to proving an upper bound on $\Diam_S(G_2)$ where the roles of $L$, $|S|$ and $\Diam_S(G_{\mathrm{ab}})$ are made explicit.

\begin{theorem}\label{diam_bound}
Let $S\subseteq G$ be a symmetric set of generators and let $R \subseteq S$ be such that $|\{s,s^{-1}\} \cap R| = 1$ for all $s \in S$. For $2\leq i\leq L$, we have
\begin{align}
\Diam_S(G_2)&\le
		\sum_{i=2}^\LL
	\nonumber \Diam_S(G_i / G_{i+1})\\
	&\le
		\sum_{i=2}^\LL
		2^{5i+7}|R|^i 
	\left( 2^{2i}+ L\cdot  \ceil{ \Diam_S(G_{\mathrm{ab}})/ |R|}^{1/i}   \right).
	\label{e:DiamGcom}
\end{align}
\end{theorem}

 As a consequence, for any set of generators satisfying $|S| \leq \frac{\log |G|}{ 8\LL r^\LL\log\log|G|}$, one has $\Diam_S(G_2) \lesssim \Diam_S(G_{\mathrm{ab}})^{3/4}$ and hence the mixing of $X_t$ can be reduced to the mixing of its projection onto the abelianization.

\subsubsection{Our Methodology}\label{our_contribution}

We describe our methodology in relation to the objectives described in Section \ref{objectives}. 

\mn
\textbf{(i) Representation of random walk.} A substantial body of work has been devoted to the study of random walks on unipotent matrix groups, see Section \ref{RW_unipotent}. The analysis in many existing work heavily depends on the favorable matrix structure specific to unipotent matrix groups, a feature not necessarily present in general nilpotent groups. 
 
There has been some progress made towards treating general nilpotent groups.  
In  \cite[\S 6]{hermon2102cutoff}, partial results were obtained using a comparison between the mixing time of a general nilpotent group $G$ with a ``corresponding" abelian group $\bar G:=\oplus_{\ell=1}^L G_\ell/G_{\ell+1}$ in \cite[\S 6]{hermon2102cutoff}. See Section \ref{setting&notation} for the definition of $\{G_\ell\}_{\ell\in [L]}$. More specifically, denoting by $\GG_k$ and $\bar \GG_k$ respectively the random Cayley graphs generated by $k$ i.i.d. uniform generators in $G$ and $\bar G$, it is shown that  $t_{mix}(\GG_k)/t_{mix}((\bar \GG)_k) \leq 1 + o(1)$ with high probability, thereby offering an upper bound on the mixing time on $\GG_k$. 

This comparison leads to a tight upper bound and thus establishes cutoff when $G$ is a nilpotent group when $G$ has a relatively small commutator subgroup $[G,G]$. Examples of such groups include $p$-groups with ``small" commutators and Heisenberg groups of diverging dimension, see \cite[Corollary D.1 and D.2]{hermon2102cutoff}. However, for general nilpotent groups this comparison is not sharp. 

While the comparison technique discussed in \cite{hermon2102cutoff} may not ensure a sharp upper bound on the mixing time for general nilpotent groups, it underscores the approach of examining the mixing behavior in relation to each quotient group $\{G_\ell/G_{\ell+1}\}_{\ell\in [L]}$. To obtain the tight upper bound and establish cutoff, we give an accurate representation of the random walk dynamics through the lens of quotient groups. 

To give a bit of intuition, let us consider the free nilpotent group of step 2 (i.e., $G_3=\{\id\}$). Let $S=\{ Z_i^{\pm 1}: i\in [k]\}$ be a set of i.i.d. uniform generators. Let $W:=W(t)=(W_1(t),\dots,W_k(t))$ be an auxiliary process defined based on the random walk $X_t$ where $W_i(t)$ is the number of times generator $s_i$ has been applied minus the number of times $s_i^{-1}$ has been applied in the random walk $X:=X_t$. Through rearranging, we can express any word in the form
\beq\label{step2_X_rep}
X=Z_1^{W_1}\cdots Z_k^{W_k} \prod_{a,b\in[k]: a<b}[Z_a,Z_b]^{m_{ba}},
\eeq
where $(m_{ba})_{a,b\in[k],a<b}$ results from the rearrangement of generators, see \eqref{seq_S} for more details. Roughly speaking, $Z_1^{W_1}\cdots Z_k^{W_k}$ keeps track of the walk on $G_{\mathrm{ab}}=G/G_2$ whereas the term $ \prod_{a,b\in[k]: a<b}[Z_a,Z_b]^{m_{ba}}$, which belongs to $G_2$, corresponds to the mixing on the quotient group $G_2/G_3$. 

We demonstrate in Section \ref{representation_X} that this line of reasoning applies to general nilpotent groups of step $L\geq 2$, see \eqref{rearranged_S}. Although the rearranging of generators leads to the presence of multi-fold commutators such as $[[Z_1,Z_2],Z_3]$ when $L\geq 3$, we will argue, through a further careful simplification, that the presence of multi-fold commutators does not add to the complexity of the analysis, and one only needs to control the distribution of $ \prod_{a,b\in[k]: a<b}[Z_a,Z_b]^{m_{ba}}$ as with the case where $L=2$.

\mn
\textbf{(ii.a) Comparison argument: reduction to abelianization.} We develop a nice argument of comparison that addresses the mixing of random walk on general groups with an arbitrary  generator set $S$. Under mild assumptions on the size of $S$, the mixing time on $G$ is the same as the mixing time of the projected walk on $G_{\mathrm{ab}}$ (up to smaller order terms), see the precise statement in Theorem \ref{reduction_ab}. That is, within the scope of Theorem \ref{reduction_ab}, the mixing time on $G$ is completely determined by that on the abelianization $G_{\mathrm{ab}}$.

Theorem \ref{reduction_ab} further implies that for a certain class of nilpotent groups with specific structures in their abelianization, the mixing time remains the same (up to smaller order terms) regardless of the choice of a minimal-sized symmetric set of generators. See Theorem \ref{maximal_class_group} for the precise statement. 

 \mn
 \textbf{(ii.b) Geometry of the Cayley graph on nilpotent groups.} We derive a quantitative upper bound on the diameter of the commutator subgroup $G_2$ in terms of the diameter of the abelianization $G_{\mathrm{ab}}$, with explicit dependence on the rank $r=r(G)$ and step $L=L(G)$ of the group $G$, as detailed in Theorem \ref{diam_bound}. This, combined with the aforementioned comparison argument, allows us to provide sufficient conditions under which the mixing behavior of the random walk on $G$ is governed by that of the projected walk on $G_{\mathrm{ab}}$.

\subsection{Historic Overview}

\subsubsection{Random Walks on Unipotent Matrix Groups}\label{RW_unipotent}

Consider the group ${\displaystyle \mathbb {U} _{n}}$ of upper-triangular matrices with $1$'s along the diagonal, so they are the group of matrices
$$
 \mathbb {U} _{n}=\left\{{\begin{pmatrix}1&*&\cdots &*&*\\0&1&\cdots &*&*\\\vdots &\vdots &&\vdots &\vdots \\0&0&\cdots &1&*\\0&0&\cdots &0&1\end{pmatrix}}\right\}.
 $$
Then, a unipotent group can be defined as a subgroup of some ${\displaystyle \mathbb {U} _{n}}$. This includes the two families of nilpotent groups discussed earlier: the $d\times d$ unit-upper triangular matrices $U_{m,d}$ with entries in $\ZZ_m$ and the $d$-dimensional Heisenberg group $H_{m,d}$ over $\ZZ_m$ where $m\in \mathbb{N}$.

The exploration of random walks on unit upper triangular matrices has led to a substantial body of research. One avenue of investigation involves a simple walk on $U_{m,d}$, the $d\times d$ unit upper triangular matrix group with entries over $\ZZ_m$ for some $m\in\mathbb{N}$: a row is chosen uniformly and added to or subtracted from the row above. Ellenberg \cite{ellenberg1993sharp} studied the diameter of the associated Cayley graph, with $d$ growing, and subsequently improved this in Ellenberg and Tymoczko \cite{ellenberg2010sharp}. 
Stong \cite{stong1995random} gave mixing bounds via analysis of eigenvalues. Coppersmith and Pak \cite{coppersmith2000random, pak2000two} look directly at mixing. Further work along this line includes Peres and Sly \cite{peres2013mixing}, Nestoridi \cite{nestoridi2019super} and Nestoridi and Sly \cite{nestoridi2020random}. Notably, Nestoridi and Sly \cite{nestoridi2020random} are the first to optimize bounds for $m$ and $d$ simultaneously.
Diaconis and Hough \cite{diaconis2021random} introduced a new method for proving a central limit theorem for random walks on unipotent matrix groups.   

In the context of i.i.d. uniformly chosen generators, Hermon and Olesker-Taylor \cite{hermon2019cutoff} prove the characterization of the cutoff time as the entropic time of the projected walk onto the abelianization for the two families of nilpotent groups: the $d\times d$ unit-upper triangular matrices $U_{m,d}$ with entries in $\ZZ_m$ and the $d$-dimensional Heisenberg group $H_{m,d}$ over $\ZZ_m$ where $m\in \mathbb{N}$.

\subsubsection{The Entropic Methodology}\label{entropic_method}
A common theme in the study of mixing times is that ``generic" instances often exhibit the cutoff
phenomenon. Moreover, this can often be handled via the entropic method, see, e.g., \cite{berestycki2018random,bordenave2019cutoff,bordenave2022cutoff}. A more detailed exposition of the known literature can be found in a previous article of one of the authors, see \cite[\S 1.3.5]{hermon2102cutoff}. Additionally, the entropic method has been applied within the context of random walks on groups, as discussed in \cite{hermon2102cutoff, hermon2019cutoff}, which we now explain in a little more depth. 

The main idea is to relate the mixing of the random walk $X=X_t$ on $\Cay(G,S)$ to that of an auxiliary process $W_t$ and study the entropy of $W_t$. Suppose $S=\{ s_i^{\pm 1}: i\in [k]\}$ is given. The auxiliary process $W=W_t:=(W_1(t),\dots,W_k(t))$ is defined based on $X_t$ where $W_i(t)$ is the number of times generator $s_i$ has been applied minus the number of times $s_i^{-1}$ has been applied in the random walk $X_t$. The observation that $W$ is a rate 1 random walk on $\ZZ^k$ (whose entropy reveals information regarding the mixing of the walk $X$) leads naturally to the definition of the entropic times, see Definition \ref{entropic_time}. 
More specifically, the auxiliary process $W$ is related to the original random walk $X$ as follows. 
We sample two independent copies of the random walk and the auxiliary process, denoted by $(X,W)$ and $(X',W')$. By Cauchy-Schwarz inequality one has
$$4 \|\P_S(X_t=\cdot|W_t)-\pi_G\|^2_{\mathrm{TV}}\leq |G| \cdot \P_S(X_t=X'_t|W_t,W'_t)-1,$$
which relates the mixing of $X$ to the hitting probability of $X$ and $X'$, i.e., the probability that $X(X')^{-1}=\id$, where the index $t$ is suppressed as it is clear from the context. 

When the group $G$ is abelian, given the choice of generators $S$, the total variation distance is a function of $W_t$ alone, see \cite{hermon2102cutoff}. When the group is not abelian, this is not the case. When $G$ is nilpotent, the auxiliary process $W_t$ still provides useful (albeit partial) information on $X_t$. In this case, to get a full picture of the mixing of the RW, we will combine the knowledge on the auxiliary process $W_t$ with further information obtained through analyzing the mixing on the quotient groups $\{Q_\ell\}_{\ell\in [L]}$ separately. See Section \ref{proof_nonzeroV} and \ref{proof_mix_com} for the complete discussion.


\section{Geometry of Cayley Graphs}\label{Geo_Cayley}

The definition of Cayley graph of a group $G$ can be naturally extended to its quotient groups. For $H \trianglelefteq G$, the Cayley graph of $G/H$, denoted by $\Cay(G/H, \{Hs : s \in S\})$, consists of vertex set $G/H$ and edge set $\{ \{Hg,Hgs\}: g\in G, s\in S\}$.

Let $\dist_S (\cdot, \cdot)$ denote the graph distance on $\Cay(G, S)$. Define 
\beq\label{projected_generator}
S_H:=\{Hs : s \in S\}.
\eeq 
Similarly, let $\dist_{S_H}(\cdot,\cdot)$ denote the 
graph distance on $\Cay(G/H, S_H)$. For a subgroup $H$ of $G$, we define the diameter of $H$ with respect to the graph distance $\dist_S (\cdot, \cdot)$ on $\Cay(G, S)$ by
\begin{equation}\label{diam_def}
\Diam_S(H) := \max\{\dist_S(id,h) : h \in H\}.
\end{equation}
For $H \trianglelefteq H' \trianglelefteq G$ such that $H\trianglelefteq G$ (so that $G/H$ is a group), with slight abuse of notation, we can define the diameter of $H'/H$ with respect to the graph distance $\dist_{S_H}(\cdot,\cdot)$ on  $\Cay(G/H, S_H)$,
$$\Diam_S(H'/H):= \max\{\dist_{S_H}(H,Hh') : h' \in H'\}.$$
 whose definition is consistent with \eqref{diam_def} with  $G,H,S$ replaced respectively by $G/H, H'/H,S_H$.
 
 We have the following  triangle inequality in terms of the diameter of a group $H'$ and that of its subgroup $H$ and the quotient group $H'/H$.
 \begin{prop}\label{triangle_diam}
For all $H\trianglelefteq H' \trianglelefteq G$  such that $H\trianglelefteq G$ the following holds:
\begin{align}
\Diam_S(H')&\le\Diam_S(H'/H) +\Diam_S(H).
\label{e:diameterdecomposition2}
\end{align}
\end{prop}

\begin{proof}
Let $h' \in H'$. By the definition of $\operatorname{Diam}_S(H'/H)$, there exists $s_1,\dots,s_m\in S$ with $m\leq \Diam_S(H'/H)$ such that 
$Hh'=Hs_1\cdots s_m$, i.e., there exists $h\in H$ such that $h'=hs_1\cdots s_m$. Hence
$$
	\dist_S(id,h')=\dist_S(id,hs_1\cdots s_m)\leq \dist_S(id,h)+m\leq \Diam_S(H) + \Diam_S(H'/H),
$$
which concludes the proof of \eqref{e:diameterdecomposition2}. 
\end{proof}

 Applying the triangle inequality in \eqref{e:diameterdecomposition2} iteratively leads to a decomposition of $\Diam_S(G_2)$ as the sum of $\Diam_S(G_i/G_{i+1})$ over $i\in [L]$, i.e.,
$$ \Diam_S(G_2)\le\sum_{i=2}^\LL\Diam_S(G_i / G_{i+1}).$$
Breulliard and Tointon \cite[Lemma 4.11]{breuillard2016nilprogressions} showed the diameter of $G_2$ is at most $C_{S,L}(\Diam_S(G)^{1/2})$, where $C_{S,L}$ is a constant depending on the size of $S$ and $L:=L(G)$. In fact, they showed for all $i\in [L]$ that $\Diam_S(G_i/G_{i+1})$ is at most $C_{S,L}\Diam_S(G)^{1/i}$. El-Baz and Pagano \cite{el2021diameters} proved the same estimates using somewhat similar arguments. In addition, they observe that $\Diam_S(G)\leq \Diam_S(G_2)+\Diam_S(G_{\mathrm{ab}})$ and hence one can estimate $\Diam_S(G_i/G_{i+1})$ in terms of $\Diam_S(G_{\mathrm{ab}})$. 

As discussed in Remark \ref{condition_for_cutoff}, a necessary condition for the random walk on $\Cay(G,S)$ to exhibit cutoff when $L$ is bounded is for $|S|$ to diverge. Consequently, as opposed to  \cite{breuillard2016nilprogressions} and \cite{el2021diameters} which did not quantify the dependence of the constant $C_{S,L}$ on $|S|$ and $L$,  it is necessary for us to quantify this dependence. Our approach for upper bounding $\Diam_S(G_2)$ adheres to the framework in El-Baz and Pagano \cite{el2021diameters}, but with considerably more attention devoted to quantifying the influence of $|S|$ as well as $L$.

\newtheorem*{thm:diam_bound}{Theorem \ref{diam_bound}}
\begin{thm:diam_bound}
Let $S\subseteq G$ be a symmetric set of generators and let $R \subseteq S$ be such that $|\{s,s^{-1}\} \cap R| = 1$ for all $s \in S$. For $2\leq i\leq L$, we have
\begin{align*}
\Diam_S(G_2)&\le
		\sum_{i=2}^\LL
	\nonumber \Diam_S(G_i / G_{i+1})\\
	&\le
		\sum_{i=2}^\LL
		2^{5i+7}|R|^i 
	\left( 2^{2i}+ L\cdot  \ceil{ \Diam_S(G_{\mathrm{ab}})/ |R|}^{1/i}   \right).
	\label{e:DiamGcom}
\end{align*}
\end{thm:diam_bound}
The following comparison between $\Diam_S(G_2)$ and $ \Diam_S(G_{\mathrm{ab}})$ is what we will use in the proof of Theorem \ref{reduction_ab}.

%
%

\begin{corollary}\label{diam_comparison}
For any fixed $L\in \mathbb{N}$, we have
$$\Diam_S(G_2) \lesssim \Diam_S(G_{\mathrm{ab}})^{3/4}$$
when $\Diam_S(G_{\mathrm{ab}})\gg |R|^{4L}$. In particular, this condition holds when $|R| \leq \frac{\log |G|}{8\LL r^\LL \log\log |G|}$.
\end{corollary}

\begin{remark}
The statement above is a special case of the following more general claim: For all $\ep>0$, if $\Diam_S(G_{\mathrm{ab}})\gg |R|^{L/\ep}$ then $\Diam_S(G_2) \lesssim \Diam_S(G_{\mathrm{ab}})^{1/2+\ep}$, which holds when $|R|\leq \frac{\ep\log|G|}{2\LL r^\LL \log\log |G|}$.
\end{remark}

\begin{proof}

Knowing that $\Diam_S(G_{\mathrm{ab}})\gg |R|^{4L}$, it is an easy consequence of Theorem \ref{diam_bound} that 
\begin{align*}
\Diam_S(G_2)&\le\sum_{i=2}^\LL2^{5i+7}|R|^i \left( 2^{2i} + L\ceil{ \Diam_S(G_{\mathrm{ab}})/|R| }^{1/i} \right)\\
&\lesssim |R|^{\LL}\cdot \Diam_S(G_{\mathrm{ab}})^{1/2} \lesssim \Diam_S(G_{\mathrm{ab}})^{3/4}. 
\end{align*}

It remains to prove that $\Diam_S(G_{\mathrm{ab}})\gg |R|^{4L}$ for the given range of $|R|$. Using the fact $|G_{\mathrm{ab}}|\geq |G|^{1/2r^\LL}$ from Corollary \ref{ab_size} we can observe that
\beq\label{lb_diamab}
|G_{\mathrm{ab}}|^{1/|R|} \geq |G|^{\frac{1}{2r^L|R|}} \geq (\log|G|)^{4L} \gg |R|^{4L}
\eeq 
for $|R| \leq \frac{\log |G|}{8\LL r^\LL \log\log |G|}$. Based on \eqref{lb_diamab}, it suffices to show $\Diam_S(G_{\mathrm{ab}})\gtrsim  |G_{\mathrm{ab}}|^{1/|R|}$ for the given range of $|R|$.

To prove $\Diam_S(G_{\mathrm{ab}})\gtrsim  |G_{\mathrm{ab}}|^{1/|R|}$ the key is to notice for the Cayley graph $\Cay(G_{\mathrm{ab}},S_{G_2})$, setting $k:=|R|$, trivially we have $|B_{G_{\mathrm{ab}}}(\ell)| \leq |B_k(\ell)|$, where $B_{G_{\mathrm{ab}}}(\ell):=\{ g\in G_{\mathrm{ab}}: \dist_{S_{G_2}}(G_2, g)\leq \ell\}$ is the ball of radius $\ell$ in $\Cay(G_{\mathrm{ab}},S_{G_2})$ and $B_{k}(\ell):=\{ \bm z\in \ZZ^k: \|\bm z\|_1\leq \ell\}$ is the $k$-dimensional lattice ball of radius $\ell$. Thus $\Diam_S(G_{\mathrm{ab}})\geq \min\{ \ell: |B_k(\ell)|\geq |G_{\mathrm{ab}}|\}.$ It follows from Lemma E.2a in  \cite{hermon2018supplementary} that $|B_k(\ell)|\leq 2^{k\wedge \ell} { \ell+k \choose k}\leq (4\ell)^k$ for $\ell\geq k$, which implies $ (4\Diam_S(G_{\mathrm{ab}}))^k\geq |G_{\mathrm{ab}}|$. Hence we have $\Diam_S(G_{\mathrm{ab}})\gtrsim  |G_{\mathrm{ab}}|^{1/|R|}$.

\end{proof}

Before we turn to the proof of Theorem \ref{diam_bound}, some preliminary results that will be useful are presented in the next section.

\subsection{Preliminaries}

We begin by recalling some standard notation and stating several properties of commutators. For $x,y\in G$ we write $[x,y] := x^{-1}y^{-1}xy = [y,x]^{-1}$ and $x^y := y^{-1}xy = x[x,y] = [y,x]x^{-1}$. Further observe that for $x,y,z\in G$, $[x,yz]=[x,z][x,y]^{z}=[x,z][x,y][[x,y],z]$. Define
	$\rho(x,y) := [x,y]$
	for $x,y \in G$ as the two-fold commutator, and inductively
\beq\label{multi_commutator}
	\rho(x_1, ..., x_i) := [\rho(x_1, ..., x_{i-1}), x_i]
\Quad{for}
	i \ge 3
\Quad{and}
	x_1, ..., x_i \in G.
\eeq

Some standard properties of commutators are collected into the following propositions whose proofs can be easily found in literature, see e.g. \cite{dummit2004abstract},  and thus are omitted. The following is a fairly well known result following from an induction argument using the three subgroup lemma.

\begin{prop}
\label{strongly_lower_central}
The lower central series of a nilpotent group $G$ is a strongly central series, i.e., $[G_i,G_j]$ is a subgroup of $G_{i+j}$  for all $i,j\geq 1$.
\end{prop}

\begin{prop}
\label{p:bilinearity} 
For $i\geq 0$, the map $\phi : G \times G_i  \to G_{i+1}/G_{i+2}$ given by $\phi(g,h) := G_{i+2}[g,h]$ is anti-symmetric and bi-linear. Namely, the following hold for all $x\in G$ and $y,z\in G_i$:
\begin{align*}
G_{i+2}[x,y]&=G_{i+2}[y,x]^{-1}\\
G_{i+2}[x,yz] &= G_{i+2}[x,y][x,z]\\
G_{i+2}[yz,x]&=G_{i+2}[y,x][z,x]\\
G_{i+2}[x^\ell,y^j] &= G_{i+2}[x,y]^{\ell j}
\Quad{for all}
	\ell,j \in \mathbb Z.
\end{align*}
Moreover,
for $i \ge 2$ and $j \le i$, if $x_1,\ldots x_i \in G$ and $y\in G$, then we have the following linearity in the $j$-th component, i.e.,
\begin{gather}
	G_{i+1} \rho(x_1, ..., x_{j-1}, x_j y,x_{j+1}, ..., x_i)
=
	G_{i+1} \rho(x_1, ..., x_i) \widehat x_{y,j}
=
	G_{i+1} \widehat x_{y,j} \rho(x_1, ..., x_i).
\label{e:linearityinfirstcoord}
\end{gather}
where $\widehat x_{y,j} := \rho(x_1, ..., x_{j-1}, y, x_{j+1}, ..., x_i)$, and so
\beq\label{e:bylinearity}
G_{i+1}\rho(a,x_2,\dots,x_i)=G_{i+1}\rho(b,x_2,\dots,x_i) \quad \text{ if } ab^{-1}\in G_2.
\eeq
\end{prop}

Let $R \subseteq S$ be such that $|\{s,s^{-1}\} \cap R| = 1$ for all $s \in S$. Now define inductively
\[
	S_1 := S
\Quad{and}
	S_i := \{ [s,s'] \mid s \in R, \: s' \in S_{i-1} \}
\Quad{for}
	i \ge 2.
\]
Write
	\(
		\widehat S_i := \{ G_{i+1}s : s \in S_i \}
	\)
	for $i \ge 1$.
The following proposition can be proved by induction on $i$ using Proposition \ref{p:bilinearity}.
We omit the details, and refer the reader to \cite{el2021diameters} for additional details.

\begin{prop}
\label{p:generatingG}
Assume that $\widehat S_1$ generates $G_{\mathrm{ab}}$.  Then $\widehat S_i$ generates the Abelian group $G_i/G_{i+1}$ for all $i\geq 1$. In particular, $S$ generates $G$ if and only if $\widehat S_1$ generates $G_\mathrm{ab}$.  
\end{prop}

\begin{corollary}\label{g_expression}
For $i\geq 1$ and any $g \in G_i$ we can write 
\begin{equation}
	G_{i+1}g
=
	G_{i+1}
	\prod_{(x_2, ..., x_i) \in R^{i-1}}
	\prod_{s \in S}
	\rho(s^{\ell_{(x_2, ..., x_i),g}(s)}, x_2, ..., x_i),
\label{e:bilinearity11}
\end{equation}
where $\{\ell_{(x_2, ..., x_i),g}(\cdot) : (x_2, ..., x_i) \in R^{i-1}\}$ are functions from $S$ to $\ZZ_+$ belonging to the set
$$
A:=	\{ \ell :S\to\ZZ_+  \text{ s.t. }\sum_{s \in S}
	|\ell(s)| \le \Diam_S(G_{\mathrm{ab}}) \text{ and } \ell(s)\cdot \ell(s^{-1}) = 0 \text{ for all } s \in S \text{ such that }s\neq s^{-1}\},
$$
where the second condition means for all $s\in S$ such that $s\neq s^{-1}$ we have either $\ell(s)=0$ or $\ell(s^{-1})=0$ for any $\ell\in A$.

	%
\end{corollary}
\begin{proof}
We know from Proposition \ref{p:generatingG} that $\widehat S_i$ generates $G_i/G_{i+1}$, i.e., for any $g\in G_i$ we can express $G_{i+1}g$ as a product of elements in $\widehat S_i=\{ G_{i+1}s: s\in S_i\}$. Observe that $\rho: S\times R^{i-1} \to S_i$ is surjective due to the definition of $S_i$. Hence we can express
$$G_{i+1}g=
	G_{i+1}
	\prod_{(x_2, ..., x_i) \in R^{i-1}}
	\prod_{s \in S}
	\rho(s, x_2, ..., x_i)^{ \tilde\ell(s,x_2,\dots,x_i)},
	$$
where $ \tilde\ell(s,x_2,\dots,x_i)\in \ZZ$ corresponds to the number of times $\rho(s, x_2, ..., x_i)$ appears. Let $h_{(x_2,\dots,x_i),g}:=\prod_{s\in S}  s^{ \tilde\ell(s,x_2,\dots,x_i)} \in G$ so that by \eqref{e:linearityinfirstcoord} 
\begin{align*}
G_{i+1} \prod_{s \in S}\rho(s, x_2, ..., x_i)^{ \tilde\ell(s,x_2,\dots,x_i)}&=G_{i+1}\rho( \prod_{s\in S}  s^{ \tilde\ell(s,x_2,\dots,x_i)}, x_2,\dots,x_i)=G_{i+1}\rho(h_{(x_2,\dots,x_i),g}, x_2,\dots,x_i).
\end{align*}
Then we can take $h^{ab}_{(x_2,\dots,x_i),g}=G_2h_{(x_2,\dots,x_i),g}\in G_{\mathrm{ab}} $ so that 
\beq\label{Gh}
G_{i+1} \rho( h_{(x_2,\dots,x_i),g}, x_2, ..., x_i)=G_{i+1}\rho( h^{ab}_{(x_2,\dots,x_i),g}, x_2, ..., x_i).
\eeq
The above expression contains a slight abuse of notation on the right hand side as $h^{ab}_{(x_2,\dots,x_i),g}$ is not an element of $G$ while $\rho$ was defined to have inputs from $G$. By \eqref{e:bylinearity} we see that for any $h'\in G$ such that $h' h_{(x_2,\dots,x_i),g}^{-1}\in G_2$,
$$G_{i+1} \rho( h_{(x_2,\dots,x_i),g}, x_2, ..., x_i)=G_{i+1}\rho(h', x_2, ..., x_i)$$
and hence what essentially determines the value of \eqref{Gh} is $h^{ab}_{(x_2,\dots,x_i),g}=G_2h_{(x_2,\dots,x_i),g}$, which clarifies the meaning of the right hand side of \eqref{Gh}. The point of doing so is that we can identify $G_{i+1} \prod_{s \in S}\rho(s, x_2, ..., x_i)^{ \tilde\ell(s,x_2,\dots,x_i)}$ with $G_{i+1} \rho( h^{ab}_{(x_2,\dots,x_i),g}, x_2, ..., x_i)$ for some $h^{ab}_{(x_2,\dots,x_i),g}\in G_{\mathrm{ab}}$. As $G_{\mathrm{ab}}$ can be generated by $\widehat S_1$, there exists some function $\hat\ell(\cdot ,x_2,\dots,x_i)$ that satisfies $\sum_{s\in S} |\hat\ell(s,x_2,\dots,x_i)|\leq \Diam_S(G_{\mathrm{ab}})$ such that 
$$G_2 \prod_{s\in S} s^{\hat\ell(s,x_2,\dots,x_i)}= h^{ab}_{(x_2,\dots,x_i),g},$$
i.e.,
$$G_{i+1} \prod_{s \in S}\rho(s, x_2, ..., x_i)^{ \tilde\ell(s,x_2,\dots,x_i)}=G_{i+1}\prod_{s \in S}\rho(s, x_2, ..., x_i)^{ \hat\ell(s,x_2,\dots,x_i)}.$$
This explains the first condition in the definition of $A$.

To explain the second condition in the definition of $A$, we observe that since $\rho(s^{-1},x_2,\dots,x_i)=\rho(s,x_2,\dots,x_i)^{-1}$ for $s\in S$, only one of $\{s, s^{-1}\}$ needs to appear in the expression above. Given $g\in G_i$ and $(x_2,\dots,x_i)$, for each $s\in S$, we choose $s_+\in \{s,s^{-1}\}$ such that simplifying the product
$$G_{i+1} \prod_{s'\in \{s,s^{-1}\}} \rho(s', x_2, ..., x_i)^{ \hat\ell(s',x_2,\dots,x_i)}=G_{i+1} \rho(s_+, x_2,\dots,x_{i})^{\ell_{(x_2,\dots,x_i),g}(s)}$$
leads to a non-negative power $\ell_{(x_2,\dots,x_i),g}(s)$. We can view $\ell_{(x_2,\dots,x_i),g}(\cdot)$ as function from $S$ to $\ZZ_+$ such that only one of $\{\ell_{(x_2,\dots,x_i),g}(s), \ell_{(x_2,\dots,x_i),g}(s^{-1})\}$ is nonzero for $s\in S$. It is straightforward to verify $\sum_{s\in S} |\ell_{(x_2,\dots,x_i),g}(s)|\leq \Diam_S(G_{\mathrm{ab}})$.

Finally, 
the proof is concluded by applying \eqref{e:linearityinfirstcoord} with the above choice of $\ell_{(x_2,\dots,x_i),g}(\cdot)$.

\end{proof}

\begin{corollary}\label{ab_size}
For $1\leq i\leq L$,
$$
	|G_i/G_{i+1}|\le \abs{G_\mathrm{ab}}^{r(G)^{i-1}} \quad \text{ and }\quad |G|\leq |G_{\mathrm{ab}}|^{2r(G)^\LL}.
$$
\end{corollary}
\begin{remark}\label{dependence_rank_step}
From a technical standpoint, the second inequality above is why the term $r^L$ is present in the condition $ |S| \leq \frac{\log |G|}{ 8\LL r^\LL\log\log|G|}$ of Theorem \ref{reduction_ab}. This inequality can be improved in various scenarios with extra knowledge on the group structure. 
\end{remark}
\begin{proof}
By \eqref{e:linearityinfirstcoord} and \eqref{e:bilinearity11}, for any $g\in G_i$, we can express $G_{i+1}g$ as 
$$G_{i+1}g=G_{i+1}\prod_{(x_2, ..., x_i) \in R^{i-1}}\rho(h_{(x_2,\dots,x_i),g}, x_2, ..., x_i)$$
for some $h_{(x_2,\dots,x_i),g}:=\prod_{s \in S}s^{\ell_{(x_2, ..., x_i),g}(s)}$ where $\ell_{(x_2, ..., x_i),g}(\cdot)\in A$ for all $(x_2,\dots,x_i)\in R^{i-1}$ where $A$ is as in Corollary \ref{g_expression}. By the same argument as in the proof of Corollary \ref{g_expression}, for any given $h_{(x_2,\dots,x_i),g}\in G$, we can take $h^{ab}_{(x_2,\dots,x_i),g}=G_2h_{(x_2,\dots,x_i),g}\in G_{\mathrm{ab}} $ so that 
$$
G_{i+1} \rho( h_{(x_2,\dots,x_i),g}, x_2, ..., x_i)=G_{i+1}\rho( h^{ab}_{(x_2,\dots,x_i),g}, x_2, ..., x_i).
$$
That is, for any $G_{i+1}g\in G_i/G_{i+1}$, we can define a function $\phi_g:R^{i-1}\to G_{\mathrm{ab}}$ by 
 $\phi_g(x_2,\dots,x_i)=h^{ab}_{(x_2,\dots,x_i),g}$, which implies $|G_i/G_{i+1}|$ is upper bounded by the number of functions from $R^{i-1}$ to $G_{\mathrm{ab}}$.
$$|G|=\prod_{i=1}^{\LL}|G_i/G_{i+1}|\leq \prod_{i=1}^{\LL} |G_{\mathrm{ab}}|^{|R^{i-1}|}\leq |G_{\mathrm{ab}}|^{2|R|^\LL}.$$
Taking $S$ such that $|R|=r(G)$ gives the desired inequality. 

\end{proof}

\subsection{Proof of Theorem~\ref{diam_bound}}
\label{sec-p3:nil:thm}
We begin with the following estimate that plays a key role in the proof of Theorem \ref{diam_bound}. 
\begin{lemma}\label{lemma_thm1}
Let $i\in [L]$ be fixed. For any $1\leq m\leq  \Diam_S(G_{\mathrm{ab}})$ and $(s,x_2, ..., x_i) \in S \times R^{i-1}$
$$
|G_{i+1}\rho(s^{m}, x_2, ..., x_i)|
\leq 2^{5i+6} \left(2^{2i}+L m^{1/i}\right).
$$
\end{lemma}
 
In what follows, we first present the proof of Theorem \ref{diam_bound} given Lemma \ref{lemma_thm1} and then complete the proof of Lemma  \ref{lemma_thm1}.

\medskip

\begin{proof}[Proof of Theorem \ref{diam_bound}]

To simplify notation, we abbreviate $D := \Diam_S(G_{\mathrm{ab}})$. Recall from \eqref{projected_generator} that $S_{G_{i+1}}=\{ G_{i+1}s: s\in S\}$ for $i\in [L]$.  Let $\dist_{S,i}(\cdot, \cdot)$ denote the graph distance on the Cayley graph $\Cay(G_i/G_{i+1}, S_{G_{i+1}})$ and  write $|G_{i+1}g| := \operatorname{dist}_{S,i}(\mathrm{id}, g)$ for $g\in G$. 

The first inequality $\Diam_S(G_2)\le\sum_{i=2}^\LL \Diam_S(G_i / G_{i+1})$  follows from inductively applying \eqref{e:diameterdecomposition2} with $H=G_{i+1}$ and $H'=G_i$ for $ 2\leq i\leq \LL$.

We turn to the second inequality. The goal is to prove that for every $2\leq i\leq \LL$, 
\beq\label{diam_quotient}
\Diam_S(G_i / G_{i+1})\leq2^{5i+7}|R|^i 
	\left( 2^{2i} + L\ceil{ D/|R| }^{1/i} \right).
\eeq
By Corollary \ref{g_expression}, in order to prove \eqref{diam_quotient} it suffices to show that for any $g\in G_i$,

\beq\label{diam_quotient2}
	\abs{G_{i+1}
			\prod_{(x_2, ..., x_i) \in R^{i-1}}
			\prod_{s \in S}
			\rho( s^{\ell_{(x_2, ..., x_i),g}(s)}, x_2, ..., x_i) 
		}
\le
	2^{5i+7}|R|^i 
	\left( 2^{2i} + L\ceil{ D/|R| }^{1/i} \right),
\eeq
where $\{\ell_{(x_2, ..., x_i),g}(\cdot) : (x_2, ..., x_i) \in R^{i-1}\}$ are functions defined in Corollary \ref{g_expression}, belonging to the set

$$
A:=	\{ \ell :S\to\ZZ_+  \text{ s.t. }\sum_{s \in S}
	|\ell(s)| \le D \text{ and } \ell(s)\cdot \ell(s^{-1}) = 0 \text{ for all } s \in S \text{ such that }s\neq s^{-1}\}.
$$

For any $\ell\in A$ and $s\in S$, we have $\ell(s)\leq D$.  Applying Lemma \ref{lemma_thm1} to $\ell(s)$ and using the triangle inequality, we can obtain 
\[
	\abs{G_{i+1} \prod_{s \in S} \rho(s^{\ell(s)}, x_2, ..., x_i) }
\le
	2^{5i+6} ( 2^{2i} |S| +L \max_{\ell \in A}\{ \sum_{s \in S}  \ell(s)^{1/i}  \} ).
\]
Given the constraint $\ell\in A$,  a simple application of Lagrange multipliers gives
\[
	\max_{\ell \in A}
	\{ \sum_{s \in S}  \ell(s)^{1/i} \}
\le
	|R| \cdot \ceil{ D/|R| }^{1/i}.
\]
Plugging this into the previous display gives 
$$\abs{G_{i+1} \prod_{s \in S} \rho(s^{\ell(s)}, x_2, ..., x_i) }
\le 2^{5i+6} ( 2^{2i} |S| +L|R| \cdot \ceil{ D/|R| }^{1/i}).$$
Summing over $(x_2, ..., x_i) \in R^{i-1}$ using the triangle inequality gives the required bound in \eqref{diam_quotient2} and thus completes the proof.

\end{proof}

\mn
\textit{Proof of Lemma \ref{lemma_thm1}.} The following simple estimate will play a major role in our proof. For $2\leq j\leq i$, $(x_1, ..., x_i) \in S \times R^{i-1}$ and $n \in \mathbb{N}$,
we have
\begin{equation}
	|\rho( x_1^{n},x_2^n, ..., x_j^n,x_{j+1}, ..., x_i)|
\le
	2^{i+2}n.
\label{e:bilinearity3}
\end{equation}
This follows from the fact that
$$
	|\rho(x_1, ..., x_i)|
\le
	2(|x_i|+|\rho(x_1, ..., x_{i-1})|),
$$
which is a simple consequence of the definition $\rho(x_1,\dots,x_i)=[\rho(x_1,\dots,x_{i-1}),x_i]$. 

Note that if $m=n^j$ for some $n\in \mathbb{N}$ and $j \ge 1$, then by \eqref{e:linearityinfirstcoord} we can simply write
\[
	G_{i+1} \rho(s^{m },x_2, ..., x_i)
=
	G_{i+1} \rho(s^n, x_2^n, ..., x_j^n, x_{j+1}, ..., x_i)
\]
and use \eqref{e:bilinearity3} to conclude the proof. Otherwise, we can still try to decompose $m$ as a sum of terms of the form $\{ n^j : j\in [i], n\in\mathbb{N}\}$, which helps improving the upper bound on $|G_{i+1} \rho(s^{m },x_2, ..., x_i)|$. In other words, our goal is to express $G_{i+1} \rho(s^m,x_2, ..., x_i)$ as the product of elements of the form
$$
	\{ G_{i+1}\rho(s^n, x_2^n, ..., x_j^n, x_{j+1}, ..., x_i) : j\in [i], n\in \mathbb{N}\}.
$$
 
To find the decomposition of $G_{i+1} \rho(s^m,x_2, ..., x_i)$ we can employ a greedy procedure to search for some set  $W(j)\subseteq \mathbb{N}$ for each $j\in [i]$ so that
$m=\sum_{j\in [i]}\sum_{n \in W(j)}n^{j}$. In what follows we first define $W(i)$ and then find $W(j)$ for $j=i-1,i-2,\dots,1$.
Setting $E_1:=m$ and $D_1 := \floor{ m^{1/i} }$, we will define $E_j$ and $D_j$ inductively for $j=i-1,i-2,\dots 2.$

For $a\geq 1$ such that $E_a\geq 4^{i^2}$, let 	
\[
		E_{a+1} := E_a - D_a^i,
	\quad
		D_{a+1} := \floor{ E_{a+1}^{1/i} }
	\Quad{and}
		y_{a} := \rho(s^{D_{a}}, x_2^{D_{a}}, ..., x_i^{D_{a}}).
	\]
We stop at the first time when $|E_a|<4^{i^2}$ and record $\ell_i:=\min\{ a: E_a < 4^{i^2}\}$. Set $W(i):=\{ D_a :1\leq a<\ell_i\}$.

For each $j=i-1,i-2,\dots,2$, in order to find $W(j)$ we proceed as follows: let 
\[
		E_{a+1} := E_a - D_a^j,
	\quad
		D_{a+1} := \floor{ E_{a+1}^{1/j} }
	\Quad{and}
		y_a := \rho(s^{D_a}, x_2^{D_a}, ..., x_j^{D_a}, x_{j+1}, ..., x_i).
	\]
We stop  at the first time when $E_a<4^{j^2}$ and record  $\ell_j:=\min\{ a: E_a < 4^{j^2}\}$. Set $W(j):=\{ D_a: \ell_{j+1}\leq a <\ell_{j}\}$. 

Finally, we set $y_{\ell_2} := \rho(s^{E_{\ell_2}},x_2, ..., x_i)$, $y := \prod_{a=1}^{\ell_2}y_a$ and $W(1):=\{ E_{\ell_2}\}$.

\medskip
By Propositon \ref{p:bilinearity} we have
\[
	G_{i+1} \rho(s^{m}, x_2, ..., x_i)
=
	G_{i+1} y.
\]
That is, it suffices to upper bound $|G_{i+1}y|$. For $1\leq a<\ell_2$, it follows from \eqref{e:bilinearity3} and the definition of $y_a$ that $|G_{i+1}y_a|\leq 2^{i+2}D_a$. It is easy to see that  $D_{a+1}\leq D_a$ for all $\ell_{j+1}\leq a<\ell_j$ and thus 
$D_a\leq D_{\ell_{j+1}}$ for $\ell_{j+1}\leq a<\ell_j$. Lastly, by definition, $|G_{i+1}y_{\ell_2}|\leq E_{\ell_2}\leq 4^{2^2}=2^8$. Combining these facts gives
\begin{align}\label{decomp_y}
\nonumber|G_{i+1}y|&\leq \sum_{a=1}^{\ell_2}|G_{i+1}y_a|=\sum_{a=1}^{\ell_i-1} |G_{i+1}y_a| +\sum_{j=2}^{i-1} \sum_{\ell_{j+1}\leq a<\ell_j}  |G_{i+1}y_a| +|G_{i+1} y_{\ell_2}|\\
&\leq \sum_{a=1}^{\ell_i-1} 2^{i+2}D_a+\sum_{j=2}^{i-1} (\ell_j-\ell_{j+1})2^{i+2}D_{\ell_{j+1}}+2^8.
\end{align}
We first upper bound the second term in \eqref{decomp_y}. By definition $E_{\ell_{j+1}}<4^{(j+1)^2}$ and thus $D_{\ell_{j+1}}\leq E_{\ell_{j+1}}^{1/j}\leq 4^{(j+1)^2/j}\leq 4^{j+3}$. To bound $\ell_{j}-\ell_{j+1}$ for $2\leq j \leq i-1$, note that for $\ell_{j+1}\leq a<\ell_j$, $E_a\geq 4^{j^2}$ and thus $D_a\geq 4^j$, which implies that
$$4^{j^2}\leq E_{\ell_j-1}\leq E_{\ell_j-2}-4^{j^2}\leq E_{\ell_{j+1}}-(\ell_j-1-\ell_{j+1})4^{j^2},$$
i.e.,
$$
(\ell_j-\ell_{j+1})\leq \frac{E_{\ell_{j+1}}}{4^{j^2}}<\frac{ 4^{(j+1)^2}}{4^{j^2}}=4^{2j+1}.
$$
It follows that the second term in \eqref{decomp_y} satisfies
\beq\label{second_term_decomp}
\sum_{j=2}^{i-1} (\ell_j-\ell_{j+1})2^{i+2}D_{\ell_{j+1}}\leq 2^{i+2} \sum_{j=2}^{i-1} 4^{2j+1} 4^{j+3}\leq 2^{7i+5}.
\eeq
Next, we estimate $\sum_{a=1}^{\ell_i-1}D_a$ in \eqref{decomp_y}. Observe that
\beq\label{E2}
E_2=m-\lfloor m^{1/i} \rfloor ^i\leq (\lfloor m^{1/i} \rfloor+1)^i-\lfloor m^{1/i} \rfloor^i \leq 2^i \lfloor m^{1/i} \rfloor^{(i-1)}\leq 2^i m^{\frac{i-1}{i}}.
\eeq
Repeating the same calculation for $2^i m^{\frac{i-1}{i}}$ yields that $E_3\leq 2^{i(1+\frac{i-1}{i})}m^{(\frac{i-1}{i})^2}$. More generally, for $1\leq a<\ell_i$, since $\sum_{h=0}^{\infty} (\frac{i-1}{i})^h=i$,
$$E_{a+1}\leq (2^i)^{\sum_{h=0}^{a-1} (\frac{i-1}{i})^h }m^{(\frac{i-1}{i})^a}\leq 2^{i^2}m^{(\frac{i-1}{i})^a}.$$
Since for $ 2\leq a<\ell_i$, $D_a=\floor{ E_{a}^{1/i} }\leq\floor{ E_{2}^{1/i} }\leq E_2^{1/i}$, by \eqref{E2} and the fact that $D_1 := \floor{ m^{1/i} }$ we have  
$$\sum_{a=1}^{\ell_i-1}D_a\leq \floor{m^{1/i}}+\ell_i\cdot 2m^{\frac{i-1}{i^2}}.$$

It remains to upper bound $\ell_i$. By definition, we have $\ell_i\leq \min\{ a: 2^{i^2}m^{(\frac{i-1}{i})^a}<4^{i^2}\}$. Simple calculation shows that for any $2\leq i\leq L$,
$$\ell_i\leq \left\lceil \frac{\log\log m-\log\log (2^{i^2})}{\log(\frac{i}{i-1})}\right\rceil \leq   \frac{\log\log m}{\log(\frac{L}{L-1})}\leq 2(L-1) \log\log m,$$
where the last inequality follows from the fact that $\log(1+x)\geq x/2$ for $x\in [0,1]$.
Therefore, for $2\leq i \leq \LL$ and $1\leq m\leq D$,
\begin{align*}
 \sum_{a=1}^{\ell_i-1}D_a&\leq \floor{m^{1/i}}+2(L-1)(\log\log m)\cdot 2m^{\frac{i-1}{i^2}}\leq 4L\cdot \max\{ m^{1/i},(\log\log m)m^{\frac{i-1}{i^2}}\}.
\end{align*}
Noting that $\max_{1\leq m\leq e^{i^2}}\{ m^{1/i},(\log\log m)m^{\frac{i-1}{i^2}}\}\leq (\log i)e^i$ and $\max_{m>e^{i^2}}\{ m^{1/i},(\log\log m)m^{\frac{i-1}{i^2}}\}\leq (\log i)m^{1/i}$, we have 
\beq\label{sum_Da}
 \sum_{a=1}^{\ell_i-1}D_a\leq 4L (\log i)e^i m^{1/i}\leq 2^{4i+2}L m^{1/i}.
\eeq

Finally, plugging the upper bounds in \eqref{second_term_decomp} and \eqref{sum_Da} into \eqref{decomp_y} yields, for $2\leq i\leq L$,
\begin{align*}
|G_{i+1}y|&\leq 2^{i+2} \cdot 2^{4i+2}Lm^{1/i} +2^{7i+5}+2^8\leq  2^{5i+4}L m^{1/i}+2^{7i+6}\\
&\leq 2^{5i+6} \left(2^{2i}+L m^{1/i}\right).
\end{align*}
which completes the proof of Lemma \ref{lemma_thm1}.
\qed

\section{Reduction to Abelianization}

Let $X_t$ be a rate 1 simple random walk on $\Cay(G,S)$ and let $Y_t:=G_2 X_t$ be the projected random walk of $X_t$ onto $G_{\mathrm{ab}}$, which is a rate 1 simple random walk on $\Cay(G_{\mathrm{ab}}, S_{G_2})$. Let $t_{\mathrm{mix}}^{G,S}(\ep)$ denote the $\ep$-mixing time of the random walk on $\Cay(G,S)$ and $t_{\mathrm{mix}}^{G_{\mathrm{ab}},S}(\ep)$ the $\ep$-mixing time for the projected random walk on $\Cay(G_{\mathrm{ab}}, S_{G_2})$. To simplify notation we will drop the $S$ in the superscript and write $t_{\mathrm{mix}}^{G}(\ep)$ instead when the choice of $S$ is clear from the context.

Since $Y_t$ is the projection of $X_t$ onto the abelianization $G_{\mathrm{ab}}$ we can observe 
\beq\label{tv_lower}
 \| \P_{G_2}(Y_t=\cdot)-\pi_{G_{\mathrm{ab}}}\|_{\mathrm{TV}}\leq \|\P_{\id}(X_t=\cdot)-\pi_G\|_{\mathrm{TV}},
 \eeq
which implies $t_{\mathrm{mix}}^{G_{\mathrm{ab}}}(\ep)\leq t_{\mathrm{mix}}^{G}(\ep)$. Naturally we are interested in the mixing behavior of $X_t$ in comparison to that of $Y_t$, that is, we hope to understand to what extent  the mixing behavior of the random walk on $G$ is governed by its projection on the abelianzation group $G_{\mathrm{ab}}$. It turns out that for a nilpotent group $G$ of bounded step and rank, when the generator set $S$ is not too large, the mixing of $X_t$ is completely governed by that of $Y_t$.

\newtheorem*{thm:reduction_ab}{Theorem \ref{reduction_ab}}
\begin{thm:reduction_ab}
Let $G$ be a finite nilpotent group such that $r(G),L(G)\asymp 1$ and $S\subseteq G$ be a symmetric set of generators. Suppose $ |S| \leq \frac{\log |G|}{ 8\LL r^\LL\log\log|G|}$. For any fixed $\ep\in (0,1)$ and $\delta\in ( 0,\ep)$ we have 
$$t_{\mathrm{mix}}^{G_{\mathrm{ab}}}(\ep)\leq t_{\mathrm{mix}}^G(\ep) \leq t_{\mathrm{mix}}^{G_{\mathrm{ab}}}(\ep-\delta)$$
when $|G|$ is sufficiently large (more precisely, when $|G|\exp(-(\log|G|)^L)\leq \delta$).
\end{thm:reduction_ab}

It is well known that the relaxation time is characterized by the exponential decay rate of the total variation distance between the talk and its equilibrium (see e.g., Corollary 12.7 in \cite{wilmer2009markov}). As a consequence of the proof of Theorem \ref{reduction_ab}, we obtain the following characterization of the relaxation time of $X_t$ in terms of its projection $Y_t$.

\newtheorem*{cor:t_rel}{Corollary \ref{t_rel}}
\begin{cor:t_rel}
Let $t^G_{\mathrm{rel}}$ and $t^{G^{\mathrm{ab}}}_{\mathrm{rel}}$ be the relaxation time of the walk $X_t$ and $Y_t$ respectively. Then
$$t^{G^{\mathrm{ab}}}_{\mathrm{rel}}\leq t^G_{\mathrm{rel}} \leq \max\{ t^{G^{\mathrm{ab}}}_{\mathrm{rel}}, |S| \cdot \Diam_S(G_2)^{2}\}.$$
In particular, when $ |S| \leq \frac{\log |G|}{ 8\LL r^\LL\log\log|G|}$ we have $t^G_{\mathrm{rel}}=t^{G^{\mathrm{ab}}}_{\mathrm{rel}}$.
\end{cor:t_rel}

Before moving on to proving Theorem \ref{reduction_ab}, we first explain why Corollary \ref{t_rel} is an easy consequence of the proof of Theorem \ref{reduction_ab} and delay its proof to the end of this section. It is useful to observe the following inequality
\begin{align}\label{tv_triangle_upper}
\|\P_{\id}(X_t=\cdot)-\pi_G\|_{\mathrm{TV}}&\leq  \| \P_{\id}(X_t=\cdot)-\P_{\pi_{G_2}}(X_t=\cdot)\|_{\mathrm{TV}}+\| \P_{\pi_{G_2}}(X_t=\cdot)-\pi_G\|_{\mathrm{TV}},
\end{align}
where $\pi_A$ denotes the uniform distribution over the set $A\subseteq G$. We will establish that in the given regime of $|S|$, the second term $\| \P_{\pi_{G_2}}(X_t=\cdot)-\pi_G\|_{\mathrm{TV}}$ in the above inequality is the leading order term that determines the time of mixing. Moreover, Lemma \ref{reduction_TV} shows this term is fully characterized by the projected random walk $Y_t$ on $G_{\mathrm{ab}}$. 

Using the interpretation that the relaxation time is the exponential decay rate of the total variation distance, by taking power $1/t$ and letting $t\to\infty$ on both sides of \eqref{tv_lower} and \eqref{tv_triangle_upper}, we shall see that $t^G_{\mathrm{rel}}=t^{G^{\mathrm{ab}}}_{\mathrm{rel}}$ if $|S| \leq \frac{\log |G|}{ 8\LL r^\LL\log\log|G|}$.

As a direct consequence of Theorem \ref{reduction_ab} and Corollary \ref{t_rel} we can now present the proof of Theorem \ref{maximal_class_group}.

\newtheorem*{thm:maximal_class_group}{Theorem \ref{maximal_class_group}}
\begin{thm:maximal_class_group}
Suppose $G$ is a nilpotent group with rank $r$ and step $L$ such that either (i) $G_{ab}\cong\ZZ^r_m$ where $m\in \mathbb{N}$ or (ii) $G$ is a $p$-group. Suppose the rank and step satisfy $Lr^{L+1}\leq \frac{\log|G|}{16 \log\log|G|}$. For any symmetric set of generators $S\subseteq G$ of minimal size and any given $\ep>0$, the mixing time $t^{G,S}_{mix}(\ep)$ is the same up to smaller order terms, and the relaxation time $t_{\mathrm{rel}}^{G,S}$ is the same. 
\end{thm:maximal_class_group}
\begin{proof}
Let $S$ be a minimal size symmetric set of generators for $G$ and let $S_{G_2}=\{ G_2 s: s\in S\}$. By Proposition \ref{p:generatingG} we can see that $S_{G_2}$ is a minimal size symmetric set of generators for $G_{\mathrm{ab}}$. As $G_{\mathrm{ab}}$ is an abelian group, it
can be expressed in the form 
\beq\label{ab_expression}
G_{\mathrm{ab}}\cong\ZZ_{m_1}\oplus\ZZ_{m_2}\oplus \cdots \oplus \ZZ_{m_r},
\eeq
where $m_1,\dots,m_r\in \ZZ$ and $r$ is the rank of $G$. In general the choice of $m_1,\dots,m_r$ is not unique (e.g., $\ZZ_{2}\oplus \ZZ_{15}\cong\ZZ_6 \oplus \ZZ_5$), but when $G$ satisfies the assumption in the statement the expression in \eqref{ab_expression} is unique: in case (i) $G_{ab}\cong\ZZ^r_m$; in case (ii) we can write $G_{\mathrm{ab}}\cong \ZZ_{p^{\alpha_1}}\oplus   \cdots \oplus  \ZZ_{p^{\alpha_r}}$ for some $\alpha_i\in\mathbb{\ZZ}$ with $\alpha_1\leq \dots\leq \alpha_r$. Hence, any minimal size symmetric set $S_{G_2}$ that generates $G_{\mathrm{ab}}$ uniquely corresponds to $\{\pm e_i :i\in [r]\}$, the collection of standard basis. Thus the mixing time $t^{G_{\mathrm{ab}},S}_{mix}(\ep)$ on $G_{\mathrm{ab}}$ for any such $S$ is equal to the mixing time of the walk on $\ZZ^r_m$ (or  $\ZZ_{p^{\alpha_1}}\oplus   \cdots \oplus  \ZZ_{p^{\alpha_r}}$) with generators $\{\pm e_i :i\in [r]\}$.

As our assumptions on $r$ and $L$ guarantees $|S|=2r\leq \frac{\log|G|}{8Lr^L\log\log|G|}$, we can apply Theorem \ref{reduction_ab} to show that the mixing time $t^{G,S}_{mix}(\ep)$ is the equal to (up to smaller order terms)  the mixing time of the walk on $\ZZ^r_m$ (or  $\ZZ_{p^{\alpha_1}}\oplus   \cdots \oplus  \ZZ_{p^{\alpha_r}}$) with generators $\{\pm e_i :i\in [r]\}$ for any minimal size symmetric set of generators $S$ of $G$. The result for relaxation time follows similarly from Corollary \ref{t_rel}.
\end{proof}

\subsection{Proofs}

\begin{lemma}\label{reduction_TV}
For $t\geq 0$,
$$\| \P_{\pi_{G_2}}(X_t=\cdot)-\pi_G\|_{\mathrm{TV}}=\| \P_{G_2}(Y_t=\cdot)-\pi_{G_{\mathrm{ab}}}\|_{\mathrm{TV}}.$$
\end{lemma}
\begin{proof}
Write $G_{\mathrm{ab}}=\{G_2g_i: 1\leq i\leq |G_{\mathrm{ab}}|\}$. One can easily check that starting from the initial distribution ${\pi_{G_2}}$, for any  $1\leq i\leq |G_{\mathrm{ab}}|$ and $h,h'\in G_2g_i$, $\P_{\pi_{G_2}}(X_t=h)=\P_{\pi_{G_2}}(X_t=h')$. 
Hence,
\begin{align*}
2\| \P_{\pi_{G_2}}(X_t=\cdot)-\pi_G\|_{\mathrm{TV}}&=\sum_{i=1}^{|G_{\mathrm{ab}}|}\sum_{x\in G_2g_i} \bigg| \P_{\pi_{G_2}}(X_t=x)-\frac{1}{|G|}\bigg|=\sum_{i=1}^{|G_{\mathrm{ab}}|} \bigg|\P_{\mathrm{id}}(X_t\in G_2g_i)-\frac{|G_2|}{|G|}\bigg|\\
&=\sum_{i=1}^{|G_{\mathrm{ab}}|} \bigg|\P_{G_2}(Y_t=G_2g_i)-\frac{1}{|G_{\mathrm{ab}}|}\bigg|\\
&=2\| \P_{G_2}(Y_t=\cdot)-\pi_{G_{\mathrm{ab}}}\|_{\mathrm{TV}}.
\end{align*}

\end{proof}

It remains to upper bound the difference $\| \P_{\id}(X_t=\cdot)-\P_{\pi_{G_2}}(X_t=\cdot)\|_{\mathrm{TV}}$. 
\begin{lemma}\label{commutator_mixing}
For $t\geq 0$, 
$$\| \P_{\mathrm{id}}(X_t=\cdot)-\P_{\pi_{G_2}}(X_t=\cdot)\|_{\mathrm{TV}}\leq  \frac{|G|}{2}\exp\left(- \frac{t}{|S|\cdot \Diam_S(G_2)^2}\right).$$
\end{lemma}

\begin{remark}
The conclusion in Lemma \ref{commutator_mixing}  holds if we replace $G_2$ by any  subgroup $H$ of $G$. 
\end{remark}

\begin{proof}
Let $P$ be the transition matrix of the simple random walk $X_t$ on $\Cay(G,S)$. For $t\geq 0$, we can define the continuous time kernel by $P_t:=\sum_{n=0}^\infty \frac{(tP)^n}{n!} e^{-t}$.

Consider the linear subspace of functions
$$\AA:=\{ f: G\to \RR \big| \sum_{x\in G_2g} f(x)=0 \text{ for all }g\in G\}.$$

We now show that $\AA$ is invariant under the transition matrix $P$, i.e., $Pf\in \AA$ for all $f\in \AA$. For any $g\in G$,
\begin{align*}
\sum_{x\in G_2g} Pf(x)&=\sum_{h\in G_2} Pf(hg)=\sum_{h\in G_2}\sum_{y\in G} P(hg,y)f(y)=\sum_{h\in G_2}\sum_{z\in G} P(hg,hz)f(hz)\\
&=\sum_{z\in G} P(g,z) \left(\sum_{h\in G_2} f(hz)\right)=0,
\end{align*}
where the second line uses the fact that $P$ is translation invariant, i.e., $P(hg,hz)=P(g,z)$ for any $h,g,z\in G$ and that $f\in \AA$.



Let $\tilde{P}$ denote the transition matrix of the SRW on $\Cay(G,\tilde S)$ with $\tilde S:=G_2$. We can also check that $\tilde P f=0$ for all $f\in \AA$, i.e.,
$$\tilde P f(x)=\sum_{y\in G} \tilde P(x,y)f(y)=\sum_{y\in G} \frac{\1\{ x^{-1}y\in G_2\} f(y)}{|G_2|}=0 \quad \text{ for all }x\in G.$$
Hence, for $f\in \AA$ we have that (below $X$ and $Y$ are independent)
\begin{align*}
\tilde\EE(f,f)&:=\frac{1}{2}\E_{X\sim \pi, Y\sim \Unif(\tilde S)}[(f(X)-f(XY))^2]=\|f\|_2^2-\E_\pi[f(\tilde P f)]=\|f\|_2^2,\\
\EE(f,f)&:=\frac{1}{2}\E_{X\sim \pi, Y\sim \Unif(S)}[(f(X)-f(XY))^2]=\|f\|_2^2-\E_\pi[f (Pf)]\geq \|f\|_2^2-\|f\|_2\|Pf\|_2,
\end{align*}
where $\|f\|_2:= \left(\E_\pi[f^2]\right)^{1/2}=\left( \sum_{g\in G} \frac{1}{|G|} f^2(g) \right)^{1/2}$ is the $\ell_2$ norm of $f$.
Applying Theorem 4.4 in \cite{berestycki2016mixing}, which gives a comparison of Dirichlet forms for two sets of generators, we see that for the two symmetric random walks on the finite group $G$ with transition matrices $P$ and $\tilde P$ defined as before, 
$$\tilde\EE(f,f)\leq  |S|\cdot \Diam_S(G_2)^2 \cdot \EE(f,f)$$
for all functions $f:G\to\RR$. That is,
$$1-\frac{\|Pf\|_2}{\|f\|_2}=\frac{\EE(f,f)}{\tilde \EE(f,f)}\geq \frac{1}{|S|\cdot \Diam_S(G_2)^2},$$
and it follows by induction that 
$$\|P^nf\|_2\leq \|f\|_2 \left(1- \frac{1}{|S| \cdot\Diam_S(G_2)^2}\right)^n.$$
Hence, for any $t\geq 0$, recalling that $P_t:=\sum_{n=0}^\infty \frac{(tP)^n }{n!} e^{-t}$, we have
\beq\label{loc_mix}
\|P_tf\|_2\leq  \|f\|_2 \exp\left(- \frac{t}{|S| \cdot\Diam_S(G_2)^2}\right).
\eeq

Define the function $\varphi: G \to\RR$ by $\varphi(g)=\1\{g=\id\}-\1\{g\in G_2\}/|G_2|$ so that
$$P_t \varphi(g)=\E_g[\varphi(X_t)]=\sum_{x\in G} P_t(g,x)\varphi(x)=\sum_{x\in G} P_t(x,g)\varphi(x).$$
By our choice of $\varphi$ it is easy to see that $P_t \varphi(g)= \P_{\id}(X_t=g)-\P_{\pi_{G_2}}(X_t=g)$ and $\varphi\in \AA$.  By the definition of $\varphi$ one can check that $\|\varphi\|_2\leq \|\varphi\|_\infty\leq 1$. 
Note that  it follows from Cauchy-Schwarz inequality that
$$4\| \P_{\id}(X_t=\cdot)-\P_{\pi_{G_2}}(X_t=\cdot)\|^2_{\mathrm{TV}}=\left(\sum_{g\in G} |P_t \varphi(g)|\right)^2 \leq|G|^2 \cdot \| P_t \varphi\|^2_2.$$
Therefore, applying $\varphi$ in \eqref{loc_mix} gives
\begin{align*}
 \| \P_{\id}(X_t=\cdot)-\P_{\pi_{G_2}}(X_t=\cdot)\|_{\mathrm{TV}}&\leq\frac{|G|}{2}\cdot \| P_t \varphi\|_2\\
 &\leq \frac{|G|}{2}\exp\left(- \frac{t}{|S| \cdot\Diam_S(G_2)^2}\right). 
 \end{align*}


\end{proof}

%
\noindent
\textit{Proof of Theorem \ref{reduction_ab}.}
The first inequality $t_{\mathrm{mix}}^{G_{\mathrm{ab}}}(\ep)\leq t_{\mathrm{mix}}^G(\ep)$ follows directly from the projection of $G$ onto $G_{\mathrm{ab}}$. To prove the second part of the inequality, observe that by the triangle inequality
$$\|P_{\id}(X_t=\cdot)-\pi_G\|_{\mathrm{TV}}\leq  \| \P_{\id}(X_t=\cdot)-\P_{\pi_{G_2}}(X_t=\cdot)\|_{\mathrm{TV}}+\| \P_{\pi_{G_2}}(X_t=\cdot)-\pi_G\|_{\mathrm{TV}}.$$
Lemma \ref{reduction_TV} shows that $\| \P_{\pi_{G_2}}(X_t=\cdot)-\pi_G\|_{\mathrm{TV}}\leq \ep-\delta$ for $t:=t_{\mathrm{mix}}^{G_{\mathrm{ab}}}(\ep-\delta)$. 

It remains to prove $ \| \P_{\id}(X_t=\cdot)-\P_{\pi_{G_2}}(X_t=\cdot)\|_{\mathrm{TV}}\leq \delta$. Recall from Corollary \ref{diam_comparison} that $\Diam_S(G_{\mathrm{ab}})\gg |S|^{4L}$ when $ |S| \leq \frac{\log |G|}{ 8\LL r^\LL\log\log|G|}$. Note that as a direct consequence of Proposition 13.7 in \cite{lyons2017probability}, which is a simple application of the Carne-Varopoulos inequality, one has  that for all $\ep\leq 1-|G_{\mathrm{ab}}|^{-1/4}$,
\beq\label{tmix_diam}
t_{\mathrm{mix}}^{G_{\mathrm{ab}}}(\ep)\gtrsim \frac{\Diam_S(G_{\mathrm{ab}})^2}{\log|G_{\mathrm{ab}}|}.
\eeq
It then follows from \eqref{tmix_diam} and Corollary \ref{diam_comparison} that
\begin{align*}
t^{G_{\mathrm{ab}}}_{\mathrm{mix}}(\ep-\delta)&\gtrsim \frac{\Diam_S(G_{\mathrm{ab}})^2}{\log|G_{\mathrm{ab}}|}\geq \frac{\Diam_S(G_{\mathrm{ab}})^{1/4}}{\log|G|}\cdot \Diam_S(G_{\mathrm{ab}})^{7/4} \gg (\log |G|)^{\LL}\cdot  |S| \cdot \Diam_S(G_2)^{2},
\end{align*}
where we recall that $L\geq 2$.
We can then apply Lemma \ref{commutator_mixing} to $t:=t_{\mathrm{mix}}^{G_{\mathrm{ab}}}(\ep-\delta)$ and get
$$\| \P_{\id}(X_t=\cdot)-\P_{\pi_{G_2}}(X_t=\cdot)\|_{\mathrm{TV}}\ll |G|\exp\left(-(\log|G|)^{L}\right)\leq \delta$$
when $|G|$ is large enough. The proof is then complete.

\qed

\mn
\textit{Proof of Corollary \ref{t_rel}.} It is straightforward to extend the conclusion in Corollary 12.7 of \cite{wilmer2009markov} to continuous time Markov chains to get 
$$\lim_{t\to\infty} \frac{\log(d(t))}{t}= -\frac{1}{t_{\mathrm{rel}}},$$
where $d(t)$ denotes the total variation distance to stationarity and $t_{\mathrm{rel}}$ denotes the corresponding relaxation time. Recall from \eqref{tv_lower} and \eqref{tv_triangle_upper} that 
\begin{align*}
 \| \P_{G_2}(Y_t=\cdot)-\pi_{G_{\mathrm{ab}}}\|_{\mathrm{TV}}&\leq \|\P_{\id}(X_t=\cdot)-\pi_G\|_{\mathrm{TV}}\\
 &\leq  \| \P_{\id}(X_t=\cdot)-\P_{\pi_{G_2}}(X_t=\cdot)\|_{\mathrm{TV}}+\| \P_{\pi_{G_2}}(X_t=\cdot)-\pi_G\|_{\mathrm{TV}},
\end{align*}

Define $f(\cdot)=\lim_{t\to\infty} \frac{\log(\cdot)}{t}$. We can apply the function $f$ to both sides in the above inequality and use Lemma  \ref{reduction_TV} and \ref{commutator_mixing} to obtain
$$t^{G^{\mathrm{ab}}}_{\mathrm{rel}}\leq t^G_{\mathrm{rel}}\leq \max\{ t^{G^{\mathrm{ab}}}_{\mathrm{rel}}, |S| \cdot \Diam_S(G_2)^{2}\}.$$
\qed

\section{On Cayley Graphs with Random i.i.d. Generators}

We consider the random walk $X(t)$ on the Cayley graph of a finite nilpotent group $G$ with respect to $k$ generators chosen
uniformly at random. The graph is denoted by $\Cay(G,S)$, where the generator set $S:=\{ Z_i^{\pm 1}: i\in [k]\}$ with $Z_1,\dots,Z_k\overset{iid}{\sim} \Unif(G)$. The regime of interest in this paper is $1\ll \log k \ll \log|G|$. In this section, our aim is to prove Theorem \ref{cutoff_iid}, which we have restated below for ease of reference.

\newtheorem*{thm:cutoff_iid}{Theorem \ref{cutoff_iid}}
\begin{thm:cutoff_iid}
 Let $G$ be a finite nilpotent group with $r(G), L(G)\asymp 1$. Let $S=\{ Z_i^{\pm 1}: i\in [k]\}$ with $Z_1,\dots,Z_k\overset{iid}{\sim} \Unif(G)$. Assume $1\ll \log k \ll \log|G|$. As $|G|\to \infty$, the  random walk on $\Cay(G,S)$ exhibits cutoff with high probability at time $t_*(k,G)$, which is the cutoff time defined in Definition \ref{entropic_time}.
\end{thm:cutoff_iid}

The structure of this section is as follows: in Section \ref{entropic_method} we will give an overview of the entropic method within the context of Theorem \ref{cutoff_iid} and how it is useful as a framework for proving bounds on the mixing time; in Section \ref{mixing_lower_bound}  we prove the lower bound on the mixing time.

Establishing the corresponding upper bound on the mixing time requires significant effort, which involves meticulous analysis of the distribution of the RW on each quotient group $Q_\ell=G_{\ell}/G_{\ell+1}$ for $\ell\in [L]$. To this end we give a representation of the random walk $X(t)$ on $\Cay(G,S)$ in Section \ref{representation_X}. Subsequently, we provide an outline of the proof for the upper bound on mixing time in Section \ref{mixing_upper_bound} and complete the proof in the remaining of the paper.

\subsection{Entropic Method and Entropic Times}\label{entropic_method}

Let $S:=\{ Z_i^{\pm}\in G : i\in [k]\}$ denote the symmetric set of generators of $G$. We define the 
 auxiliary process $W:=W(t)=(W_1(t),\dots,W_k(t))$ based on $X(t)$ where $W_i(t)$ is the number of times generator $Z_i$ has been applied minus the number of times $Z_i^{-1}$ has been applied in the random walk $X(t)$. It is easy to see $W(t)$ is a rate 1 random walk on $\ZZ^k$. To simplify notation, we sometimes drop the time index $t$ when it is clear from the context.

Let $(X, W)$ and $(X', W')$ be two independent copies of the random walk on $\Cay(G,S)$ starting at $\id$ and its auxiliary process. Denote by $\mathcal{W}:=\mathcal{W}(t)\subseteq\ZZ^k$ a subset of the state space of the walk $W$, where the index suggests that the precise choice of $\mathcal{W}$ (which is postponed to Definition \ref{typical_event}) depends on the time $t$. For now we will think of $\mathcal{W}$ as a set of ``typical" locations of $W$ such that $\P(W(t) \notin \mathcal{W})=o(1)$ for the relevant choice of $t$.

We use a ``modified $L^2$ calculation": first conditioning on $W$ being ``typical"; then using a standard
$L^2$ calculation on the conditioned law. The proof of the following lemma is quite straightforward and hence is omitted.

\begin{lemma}[Lemma 2.6 of \cite{hermon2102cutoff}]\label{TV_entropic}
For all $t\geq 0$ and all $\mathcal{W}\subseteq \ZZ^k$ the following inequalities hold:
\begin{align*}
d_{S}(t):=\| \P_S(X(t)\in \cdot)-\pi_G\|_{\mathrm{TV}}&\leq \| \P_S(X(t)\in \cdot |W(t)\in \mathcal{W})-\pi_G\|_{\mathrm{TV}}+\P(W(t)\notin \mathcal{W})\\
4  \|\P_S(X(t)\in \cdot|W(t)\in \mathcal{W})-\pi_G\|^2_{\mathrm{TV}}&\leq |G| \cdot \P_S(X(t)=X'(t)| W(t),W'(t)\in \mathcal{W})-1
\end{align*}
where $\P_S$ denotes the law of the random walk given the generator set $S$ starting at $X(0)=\id$. 
\end{lemma}
Note that when $S$ is a random set of generators, $d_{S}(t)$ is a random variable that is measurable with respect to $\sigma(S)$, the $\sigma$-field generated by the choice of $S$. In what comes later in our arguments, we will take the expectation over the choices of $S$ and work with
$$\E[\P_S(X(t)=X'(t)| W(t),W'(t)\in \mathcal{W})]=:\P(X(t)=X'(t)| W(t),W'(t)\in \mathcal{W}).$$

A good choice of the set of ``typical" locations $\mathcal{W}$ will greatly simplify the analysis. Hence, in the remaining of this section, we will discuss in detail the choice of $\mathcal{W}$, or in other words, the ``typical event" that we will condition on. As our goal is to obtain an upper bound on the total variation mixing time, we will look at a time $t$ that is slightly larger than the proposed mixing time, which is the entropic times that will be defined in Section \ref{entropic_times}. Based on this choice of time we then define the typical event in Section \ref{typ_event}.

\subsubsection{Asymptotics of Entropic Times}\label{entropic_times}
Recall from Definition \ref{entropic_time} that the entropic time $t_0:=t_0(k,|G_{\mathrm{ab}}|)$ is the time at which the entropy of the rate 1 random walk $W$ on $\ZZ^k$ is $\log |G_{\mathrm{ab}}|$.

For simplicity in simplicity of notation, we will write $t_1:=\log_k |G|$ and let the cutoff time be denoted by $t_*:=t_*(k,G)=\max\{ t_0,t_1\}$. 

The entropy of the random walk on $\ZZ^k$ has been well understood. The interested reader can find a detailed exposition of the entropic times in \cite{hermon2018supplementary}. 
Via direct calculation with the simple random walk and Poisson laws, one can obtain asymptotics of entropic times, see, e.g., \cite[Proposition A.2 and \S A.5]{hermon2018supplementary} for full details. Here we content ourselves with restating the result of such calculation, which can be found in Proposition 2.2 in \cite{hermon2019cutoff}, that yields the asymptotic values of $t_0$:

\begin{equation}
\label{flambda}
\begin{array}{ll}
    t_0 \eqsim k\cdot  |G_{\mathrm{ab}}|^{2/k}/(2\pi e) & \quad\text{when } k\ll \log|G_{\mathrm{ab}}|, \\
    t_0 \eqsim k\cdot  f(\lambda) &\quad \text{when } k\eqsim \lambda \log|G_{\mathrm{ab}}|, \\
    t_0 \eqsim k\cdot 1/(\kappa \log \kappa) & \quad \text{when } k\gg \log|G_{\mathrm{ab}}|,
\end{array}
\end{equation}
where $\kappa:=k/\log |G_{\mathrm{ab}}|$ and $f:(0,\infty)\to(0,\infty)$ is some continuous decreasing function whose exact value is unimportant for our analysis. Since we assume $r\asymp 1$ and $L\asymp 1$, it follows from Corollary \ref{ab_size} that $\log|G_{\mathrm{ab}}|\asymp \log|G|$. Consequently, it is possible to have $t_1>t_0$ only in the regime $k\gg \log|G_{\mathrm{ab}}|$. To be more specific, writing $\rho=\frac{\log k}{\log\log |G_{\mathrm{ab}}|}$, in the regime $k\gg \log|G_{\mathrm{ab}}|$ we have $\liminf_{|G|\to\infty}\rho\geq 1$ and
$$
\begin{cases}
t_0> t_1  & \quad \text{ when } \frac{\rho}{\rho-1}> \frac{\log |G|}{\log |G_{\mathrm{ab}}|},\\
t_0 \leq t_1 & \quad \text{ otherwise.} 
\end{cases}
$$

The following proposition gives the asymptotics of the cutoff time $t_*$ for the regimes of $k$ that we are interested in.
\begin{prop}
Writing $\rho=\frac{\log k}{\log\log |G_{\mathrm{ab}}|}$, we have the following asymptotics of $t_*(k,G)$:
$$
t_*(k,G)\eqsim
\begin{cases}
k  |G_{\mathrm{ab}}|^{2/k}/(2\pi e) & \quad \text{when}\quad k\ll \log|G_{\mathrm{ab}}|,\\
 k  f(\lambda)  & \quad \text{when}\quad k\eqsim \lambda \log|G_{\mathrm{ab}}|,\\
k/(\kappa \log \kappa) & \quad \text{when}\quad k\gg \log|G_{\mathrm{ab}}|, \frac{\rho}{\rho-1}> \frac{\log |G|}{\log |G_{\mathrm{ab}}|},\\
\log_k|G| & \quad \text{when}\quad k\gg \log|G_{\mathrm{ab}}|, \frac{\rho}{\rho-1}\leq \frac{\log |G|}{\log |G_{\mathrm{ab}}|},
\end{cases}
$$
where $f$ is defined as in \eqref{flambda} and $\kappa:=k/\log |G_{\mathrm{ab}}|$.

\end{prop}

\subsubsection{Typical Event}\label{typ_event}

For simplicity of notation, in this section we will drop the index of time $t$ and write $W:=W(t)$ and $X:=X(t)$. Write $W=(W_1,W_2,\dots,W_k)$, where for each $a\in[k]$, $W_a$ is an independent rate $1/k$ random walk on $\ZZ$.  For each $a\in[k]$, define $W_a^+$ to be the number of steps to the right and $W_a^-$ the number of steps to the left in the walk $W_a$. It is then easy to see $W_a=W_a^+-W_a^-$.

\medskip
To upper bound the key quantity $\P(X=X'| W,W'\in \mathcal{W})$ from Lemma \ref{TV_entropic}, we will separate it into cases according to whether or not $W = W'$:
\begin{align*}
\P(X=X'| W,W'\in \mathcal{W})&=\P(X=X|\{W=W' \}\cap\{W, W'\in \mathcal{W}\})\P(W=W'| W, W'\in \mathcal{W})\\
&\quad +\P(X=X|\{W\neq W' \}\cap\{W, W'\in \mathcal{W}\})\P(W\neq W'| W, W'\in \mathcal{W}).
\end{align*}
As suggested by this decomposition, bounding $\P(W=W'| W, W'\in \mathcal{W})$ plays an important role in the proof. 
Since $W,W'$ are two independent copies,
\begin{align*}
\P(\{W=W'\} \cap\{ W, W'\in \mathcal{W}\})&=\sum_{w\in \mathcal{W}}\P(W=w)\P(W'=w)\leq \max_{w\in \mathcal{W}} \P(W=w),
\end{align*}
i.e.,
$$\P(W=W'| W, W'\in \mathcal{W})\leq \frac{ \max_{w\in \mathcal{W}} \P(W=w)}{\P(W, W'\in \mathcal{W})}.$$
It now becomes clearer that in order to control the probability $\P(W=W'| W, W'\in \mathcal{W})$, we would like to choose $\mathcal{W}$ so that $\P(W, W'\in \mathcal{W})=1-o(1)$ and $\max_{w\in \mathcal{W}} \P(W=w)$ is sufficiently small.

For $t\geq 0$, write $\mu_t$ for the law of $W(t)$, the rate 1 random walk on $\ZZ^k$, so that $\mu_t(w)=\P(W(t)=w)$. Also write $\nu_s$ for the law of $W_1(sk)$ so that $\mu_t=\nu_{t/k}^{\otimes k}$. Also, for each $a\in[k]$, define
$$Q_a(t):=-\log\nu_{t/k}(W_a(t))\quad \text{ and }\quad Q(t):=-\log\mu_t(W(t))=\sum_{a=1}^k Q_a(t).$$ 
Then $\E(Q(t))$ (and respectively $\E(Q_1(t))$) is the entropy of $W(t)$ (and respectively $W_1(t)$).

We need an estimate on the entropy $\E(Q(t))$ shortly after the proposed mixing time, i.e., for $t\geq (1+\ep)t_*(k,G)$, for which  we refer to results in  \cite{hermon2019cutoff}. To state their results, we need to define the following quantity.


\begin{definition}\label{def_h}
Define $h_0$ as follows 
$$h_0:=
\begin{cases}
\log |G_{\mathrm{ab}}| & \text{ when } t_0(k,|G_{\mathrm{ab}}|)> \log_k |G|,\\
(1-\frac{1}{\rho})\log|G|& \text{ when } t_0(k,|G_{\mathrm{ab}}|)\leq \log_k |G|
\end{cases}
$$
where $\rho=\frac{\log k}{\log\log |G_{\mathrm{ab}}|}$. Fix some $\omega$ such that $1\ll \omega \ll \min\{ k, \log|G_{\mathrm{ab}}|\}$, and set $h:=h_0+\omega$.
\end{definition}
\begin{remark}
Note that $\log |G_{\mathrm{ab}}|\leq h_0$ in both cases.
\end{remark}

Lemma 3.9 in \cite{hermon2019cutoff} proves concentration of $Q(t)$ whereas we only need one side of their estimate, which we state below.
\begin{lemma}[Lemma 3.9 in \cite{hermon2019cutoff}]\label{entropy_concentration}
Assume that $\omega \ll \min\{ k, \log|G_{\mathrm{ab}}|\}$. Let $\ep>0$ and $t\geq (1+\ep)t_*(k,G)$. Then 
$$\P(Q(t)\geq h)=\P(\mu_t (W(t))\leq e^{-h})=1-o(1).$$
\end{lemma}

Based on the above discussion and Lemma \ref{entropy_concentration}, it makes sense to define the (global) typical event as follows
\beq\label{glo_typ}
\WW_{glo}:=\{ w\in \ZZ^k: \P(W(t)=w)\leq e^{-h}\}.
\eeq

\bigskip
Write $W^{\pm}:=(W^\pm_1,W^\pm_2,\dots,W^\pm_k)$. Let $w\in \ZZ^k$ denote a realization of $W$ and the corresponding $(w^+,w^-)\in  \ZZ^k_+\times \ZZ^k_+$ a realization of $(W^+,W^-)$. To have better control over the behavior of each coordinate, we further define
\beq\label{local_typ}
\WW_{loc}:=\{ (w^+,w^-)\in \ZZ^k_+\times \ZZ^k_+:  |w^{\pm}_a-\E(W^{\pm}_a(t))|\leq r_*, \forall a\in [k] \}\\
\eeq
where $r_*:=\frac{1}{2}|G_{\mathrm{ab}}|^{1/k}(\log k)^2$. It can be observed that $r_*$ is defined based on $t_0(k,|G_{\mathrm{ab}}|)$ when $k\lesssim \log|G_{\mathrm{ab}}|$ so that $W^\pm(t)\in \WW_{loc}$ whp by a union bound on the $k$ coordinates. In fact, we will use $\WW_{loc}$ only in the regime $k\lesssim \log|G_{\mathrm{ab}}|$.

\medskip
In the regime $k\gtrsim \log|G_{\mathrm{ab}}|$ we have $t_*/k\lesssim 1$. By Poisson thinning, for each $a\in[k]$ the arrivals of the generators $Z_a^{\pm 1}$ follow an independent Poisson process with rate $1/k$. Then $t_*/k\lesssim 1$ implies that each generator is expected to appear for $O(1)$ times in the walk $X$. Thus in this regime we will focus on the collection of generators that appear exactly once. Define, for $(w^+,w^-)\in \ZZ_+^k\times \ZZ_+^k$,
$$\mathcal{J}(w^+,w^-):=\{ a\in [k]:  w^+_a+w^-_a=1\} \quad\text{ and }\quad J(w^+,w^-)=|\mathcal{J}(w^+,w^-)|$$
so that $\mathcal{J}(w^+,w^-)$ is the index set of generators that appear exactly once in the realization $\{W=w\}$. 

Moreover, for a sufficiently small $\ep>0$, define 
\beq\label{once_typ}
\WW_{once}:=\{ (w^+,w^-)\in \ZZ^k_+\times \ZZ^k_+: |J(w^+,w^-)-te^{-t/k}|\leq \frac{1}{2}\ep t e^{-t/k}\}.
\eeq
We can observe that the distribution of $J$ is $\mathrm{Binomial}(k, (t/k)e^{-t/k})$ so that when $te^{-t/k}\gg 1$ (which holds when $k\gtrsim \log|G|$ and $\log k \ll \log|G|$), $\WW_{once}$ occurs with high probability for any $\ep>0$.

\begin{definition}\label{typical_event}
Let $\WW_{glo}, \WW_{loc}, \WW_{once}$ be defined as in \eqref{glo_typ},\eqref{local_typ} and \eqref{once_typ}.
Define the typical event 
$$\typ:=
\begin{cases}
\{ W,W'\in \WW_{glo}\}\cap\{ (W^+,W^-),((W')^+,(W')^-)\in \WW_{loc}\} &\text{ when }k\ll \log|G_{\mathrm{ab}}|,\\
\{ W,W'\in \WW_{glo}\}\cap\{ (W^+,W^-),((W')^+,(W')^-)\in \WW_{loc}\cap \WW_{once}\} &\text{ when }k\eqsim \lambda \log|G_{\mathrm{ab}}|,\\
\{ W,W'\in \WW_{glo}\}\cap\{ (W^+,W^-),((W')^+,(W')^-)\in \WW_{once}\} &\text{ when }k\gg \log|G_{\mathrm{ab}}|.\\
\end{cases}
$$
\end{definition}

\begin{lemma}
$\P(\typ)=1-o(1)$.
\end{lemma}
\begin{proof}
Lemma \ref{entropy_concentration} implies that $ W,W'\in \WW_{glo}$ with high probability. 
The proof for $\P( (W^+,W^-)\in \WW_{loc})=1-o(1)$ when $k\lesssim  \log|G_{\mathrm{ab}}|$ follows from standard large deviation estimation. The proof for  $\P( (W^+,W^-)\in \WW_{once})=1-o(1)$ when $k\gtrsim \log|G_{\mathrm{ab}}|$ also follows from  standard large deviation estimation.
\end{proof}

\mn
\subsection{Lower Bound on Mixing Time}\label{mixing_lower_bound}
As projection does not increasing the total variation distance, to find a lower bound on the mixing time we can consider the projection of the original random walk $X(t)$, defined by $Y(t):=G_2X(t)$, on the projected Cayley graph $\Cay(G_{\mathrm{ab}}, \{G_2Z_i^{\pm 1}: i\in [k]\})$. As $G_{\mathrm{ab}}$ is abelian, the mixing behavior of $Y_t$ is well understood, see \cite{hermon2102cutoff,hermon2019cutoff}. 

The key idea of proving the lower bound comes from a concentration result of the entropy, which has been proved in the literature and is restated here for the sake of self-containment.

\begin{prop}[Proposition 2.3 in \cite{hermon2019cutoff}]\label{lower_concentration}
Assume that $k$ satisfies $1\ll \log k\ll \log |G|$ and recall $t_0:=t_0(k,|G_{\mathrm{ab}}|)$ as in Definition \ref{entropic_time}. Then $\mathrm{Var}(Q(t_0)) \gg 1$ and further, for $\ep>0$, writing $\omega:=(\mathrm{Var}(Q(t_0)))^{1/4}$, we have 
$$\P(Q((1-\ep)t_0)\geq \log|G_{\mathrm{ab}}| - \omega)\to0.$$
\end{prop}

The proof of the lower bound follows from a somewhat conventional argument within the entropic methodology. Our proof is essentially a restatement of the proof in Section 3.3 of \cite{hermon2019cutoff}, which we include for the purpose of being self-contained.
 
\begin{lemma}
Assume that $1\ll \log k\ll \log |G|$. Let $S=\{ s^\pm_i: i\in [k]\}\subseteq G$ be a given generator set. For any $\ep>0$ and $t\leq (1-\ep)t_*(k,G)$,
$$\| \P_S( X(t)\in \cdot)-\pi_G\|_{\mathrm{TV}}\geq 1-o(1),$$
where $\P_S$ denotes the law of the random walk $X$ with $X(0)=\id$ given the generator set $S$.
\end{lemma}
\begin{proof}
Since the argument does not depend on the choice of $S:=\{ s_i^{\pm 1}: i\in [k]\}$ we will suppress it from  notation and write $\P$ for $\P_S$.

Recall that $t_*(k,G)=\max\{ \log_k|G|, t_0(k,|G_{\mathrm{ab}})\}$. To argue that $(1-\ep)\log_k|G|$ is a lower bound on the mixing time, first observe that in $m\in\mathbb{N}$ steps the support of the random walk $X$ is at most of size $k^m$. When $m\leq (1-\ep) \log_k |G|$, the support has size at most $|G|^{1-\ep}$ and hence the walk cannot be mixed in this many steps. 

Let $t:=(1-\ep)t_0(k,|G_{\mathrm{ab}}|)$ and define 
$$\EE:=\{ \mu_t(W(t))\geq |G_{\mathrm{ab}}|^{-1} e^\omega\} =\{ Q(t)\leq \log|G_{\mathrm{ab}}|-\omega\}$$
with $\omega\gg1$ from Proposition \ref{lower_concentration}, by which we have $\P(\EE)=1-o(1)$.

Let $\Pi: G\to G_{\mathrm{ab}}$ denote the canonical projection. Consider 
$$E:=\{ x\in G_{\mathrm{ab}}: \exists w\in \ZZ^k \text{ s.t. } \mu_t(w)\geq |G_{\mathrm{ab}}|^{-1} e^\omega \text{ and } x=\Pi(Z_1^{w_1}\cdots Z_k^{w_k})\} \subseteq G_{\mathrm{ab}}.$$
Based on the definition of $E$ we have $\P( \Pi(X(t))\in E|\EE)=1$. Every element $x\in E$ satisfies $x = \Pi( Z_1^{w^x_1}\cdots Z_k^{w^x_k})$ for some $w^x\in \ZZ^k$ with $\mu_t(w^x)\geq |G_{\mathrm{ab}}|^{-1} e^\omega$.
Hence, for all $x\in E$, we have 
$$\P(\Pi(X(t))=x)\geq \P( W(t)=w^x)=\mu_t(w^x)\geq  |G_{\mathrm{ab}}|^{-1} e^\omega.$$
Summing over $x\in E$ gives
$$1\geq \sum_{x\in E} \P(\Pi(X(t))=x)\geq |E|\cdot  |G_{\mathrm{ab}}|^{-1} e^\omega$$
and hence $|E|/|G_{\mathrm{ab}}|\leq e^{-\omega}-o(1)$. Therefore,
$$\| \P(X(t)\in \cdot )-\pi_G\|_{\mathrm{TV}}\geq \P( X(t)\in \Pi^{-1}(E))-\pi_G(\Pi^{-1}(E))\geq \P(\EE)-|E|/|G_{\mathrm{ab}}|=1-o(1),$$
which completes the proof.

\end{proof}

\subsection{Representation of $X(t)$} \label{representation_X}
Let $N:=N(t)$ be the number of steps taken by the continuous time rate 1 random walk $X:=X(t)$ on the group $G$. For $i\in[N]$, we can write (the increment of) the $i$-th step taken by $X$ as $Z^{\eta_i}_{\sigma_i}$, where $\sigma_i\overset{iid}{\sim} \Unif( [k])$ and $\eta_i \overset{iid}{\sim} \Unif\{ \pm 1\}$. Then we can express $X=\prod_{i=1}^N Z^{\eta_i}_{\sigma_i}$, and similarly $X'=\prod_{i=1}^{N'} Z^{\eta'_i}_{\sigma'_i}$, where $N', \{\eta'_i: i\in [N']\},\{ \sigma'_i: i\in [N']\}$ are independent random variables defined analogously. Note that for $a\in [k]$,  $W_a(t)=\sum_{i=1}^{N(t)} \1\{\sigma_i=a\} \eta_i$ for $a\in[k]$. 

We are interested in the event $\{X=X'\}$, which is equivalent to $\{\mathbf{X}=\id\}$, where
\beq\label{seq_S}
\mathbf{X}:=X(X')^{-1}=\prod_{i=1}^N Z^{\eta_i}_{\sigma_i}\cdot \left( \prod_{i=1}^{N'} Z^{\eta'_i}_{\sigma'_i}  \right)^{-1}=\prod_{i=1}^N Z^{\eta_i}_{\sigma_i} \cdot  \prod_{j=0}^{N'-1} Z^{-\eta'_{N'-j}}_{\sigma'_{N'-j}}.
\eeq
It is easy to see the law of $\mathbf{X}$ is the same as that of a rate 2 simple random walk on the same Cayley graph $\Cay(G,S)$. In fact, for simplicity of notation we will write \eqref{seq_S} as
\beq\label{seq_S_general}
\mathbf{X}=\prod_{i=1}^{N+N'} Z^{\eta_i}_{\sigma_i}
\eeq
where $\sigma_i:=\sigma'_{N'+N-i}$ and $\eta_i:=-\eta'_{N'+N+1-i}$ for $N<i \leq N+N'$. 
 
%

Our goal is to express $X$ (and analogously $\mathbf{X}$) in a way that the role of $W$  (and analogously $W-W'$) is  understood. 
If $G$ is Abelian we can simply rearrange the sequence $X=\prod_{i=1}^N Z^{\eta_i}_{\sigma_i}$ to obtain $X=Z_1^{W_1}\cdots Z_k^{W_k}$. Although we do not have this nice and simple relation between $X$ and $W$ when $G$ is not abelian, we can still rearrange the terms in \eqref{seq_S} and pay the price of adding an extra commutator, i.e., for $x,y\in G$ we can rewrite $xy$ as $yx[x,y]$. For this reason,  in some of our analysis we also care about the specific order in which each generator appears in $\mathbf{X}$, which is why we will sometimes also refer to $\mathbf{X}$ as a sequence to emphasize this perspective. To be more specific, when we refer to $\mathbf{X}$ as a ``sequence" we are referring to the corresponding $(\sigma_i,\eta_i)_{i\in [N+N']}$.

More generally, for $x,y,z\in G$, consider as an example the element $xyzx\in G$. In order to express $xyzx$ in our desired form we can rearrange the terms in the sequence $xyzx$ as follows
\begin{align}\label{example_seq}
\nonumber xyzx&=xy xz[z,x]=x^2y [y,x]z[z,x]= x^2y z [y,x] [ [y,x],z][z,x]\\
&=x^2yz[y,x] [z,x] [ [y,x],z][ [[y,x],z], [z,x]].
\end{align}
We can see from this example that rearranging the whole sequence $\mathbf{X}$ will result in commutators of the form $\{\rho(x_1, ..., x_i): x_j\in G, j\in [i], i\geq 2\}$, which was defined in \eqref{multi_commutator}.  

In terms of the sequence $\mathbf{X}$, it will become clear later that we actually only need to keep track of the two-fold commutators of the form $\{ [Z_a,Z_b]: a,b\in[k]\}$. Hence, we will write $V:=W-W'$ and rearrange the terms in \eqref{seq_S_general} to obtain the following expression
\beq\label{rearranged_S}
\mathbf{X}=Z_1^{V_1}\cdots Z_k^{V_k} \prod_{a,b\in[k]: a<b}[Z_a,Z_b]^{m_{ba}} \varphi(Z_1,\dots,Z_k),
 \eeq
where, letting $(\sigma_i,\eta_i)$ denote the $i$-th generator and its sign in the sequence in \eqref{seq_S_general}, 
\beq\label{mba}
m_{ba}:=-\sum_{i=1}^{N+N'}\sum_{j<i} \eta_i \eta_j  \1\{\sigma_i=a, \sigma_j=b\} \quad \text{ for } 1\leq a<b \leq k,
\eeq
and
$\varphi(Z_1,\dots,Z_k)$ is the residual part as the result of the rearranging. To give a more specific description of $\varphi$, define  $\mathcal{C}_{\mathrm{com}}:=\{ \rho(Z^{\pm 1}_{a_1},\dots,Z_{a_i}^{\pm 1}): a_j\in [k] \text{ for all }j\in [i], i\geq 2\}$ to be the collection of commutators of $\{Z^{\pm 1}_1,\dots,Z^{\pm 1}_k\}$. A multi-fold commutator of  $\{Z^{\pm 1}_1,\dots,Z^{\pm 1}_k\}$ refers to a term of the form $\rho(x_1,\dots,x_i)$ with $i\geq 2$ where $x_j\in \mathcal{C}_{\mathrm{com}}\cup \{Z^{\pm 1}_1,\dots,Z^{\pm 1}_k\}$ for all $j\in [i]$, which is not simply a two-fold commutator of the form $\{[Z_{a_1}^{\pm1},Z_{a_2}^{\pm 1}]: a_1,a_2\in [k]\}$. A multi-fold commutator consisting of $i$ pairs of brackets is said to be $(i+1)$-fold. For example, see \eqref{example_seq} where $[ [y,x],z]$ is a 3-fold commutator and $[ [[y,x],z], [z,x]]$ is a 5-fold commutator.

 It will become clear in our later arguments that the specific order in which terms appear in $\varphi(Z_1,\dots,Z_k)$ is of no interest to us, and thus with some abuse of language we will sometimes refer to $\varphi(\cdot)$ as a polynomial with terms that are multi-fold commutators of $\{Z^{\pm 1}_1,\dots,Z^{\pm 1}_k\}$. 

\subsection{Upper Bound on Mixing Time}\label{mixing_upper_bound}

In this section we prove the upper bound on the mixing time, for which the precise statement is presented below. Recall that $\P_S$ denotes the law of the random walk given the generator set $S$ starting at $X(0)=\id$.

\begin{theorem}\label{mixingtime_ub}
Let $G$ be a nilpotent group with $r(G), L(G)\asymp 1$. Let $S=\{ Z_i^{\pm 1}: i\in [k]\}$ with $Z_1,\dots,Z_k\overset{iid}{\sim} \Unif(G)$ and assume $1\ll \log k \ll \log|G|$. For any $\ep>0$ and $t\geq (1+\ep)t_*(k,G)$, we have $\| \P_S(X(t)\in \cdot)-\pi_G\|_{\mathrm{TV}}=o(1)$ with high probability.
\end{theorem}
Notice that when $1\ll k\ll \frac{\log|G|}{\log\log|G|}$ this theorem follows directly from Theorem \ref{reduction_ab} and hence in the proof we will focus on the regime $k\gtrsim \frac{\log|G|}{\log\log|G|}$.

We will first define the notation that will be used throughout this section and explain why it is useful.

\begin{definition}\label{representatives}
For each $\ell\in [L]$, define $Q_{\ell}:=G_{\ell}/G_{\ell+1}$ and let $r_\ell:=r(Q_\ell)$ denote the rank of $Q_{\ell}$. For each $Q_\ell$, we will choose a set $R_\ell \subseteq G_{\ell}$ such that $|R_\ell|=|Q_\ell|$ and $Q_\ell=\{ G_{\ell+1}g: g\in R_\ell\}$. 
\end{definition}

As 
$$\{X(X')^{-1}=\id\}=\cap_{\ell=1}^{L+1}\{ X(X')^{-1}\in G_{\ell}\}=\cap_{\ell=1}^{L+1} \{ G_\ell X(X')^{-1}=G_\ell\},$$ 
we will be interested in events related to $\{G_{\ell}X(X')^{-1}\}_{\ell \in [L+1]}$. In other words, we will decompose $X(X')^{-1}$ with respect to the quotient groups $\{Q_{\ell} :\ell \in [L]\}$ and derive simplified expressions to make the distribution of $X(X')^{-1}$ more tractable.


It will be tremendously useful if we can express $Z_a$ as a product of elements belonging to each layer of $\{G_\ell: \ell\in [L]\}$. Indeed,  the following lemma asserts that we can construct each generator $Z_a$ as a product of independent random variables $\{Z_{a,\ell}: \ell \in [L]\}$.

\begin{lemma}[Corollary 6.4 in \cite{hermon2102cutoff}]\label{Z_decomp}
Let $\{Z_{a,\ell} : a\in[k], 1\leq \ell\leq L\}$ be independent and such that $Z_{a,\ell} \sim \Unif (R_\ell)$. Then $G_{\ell+1}Z_{a,\ell}\sim \Unif(Q_\ell)$. Moreover, $Z_a:=\prod_{\ell=1}^L Z_{a,\ell}$ are i.i.d. uniform over $G$ for $a\in[k]$. 
\end{lemma}

\mn
\subsubsection{Proof Framework for Theorem \ref{mixingtime_ub}} As suggested by Lemma \ref{TV_entropic}, in order to control the total variation distance to stationarity we will upper bound  $D(t):=|G|\cdot \P(X=X'|\typ)-1$, where we also average over the choice of $S$. Letting $V:=W-W'$, $D(t)$ can be further decomposed  with respect to $V$:
\beq\label{tv_decomp}
D(t)=|G|\cdot \P(X=X', V\neq 0|\typ)+|G|\cdot \P(X=X',V=0| \typ)-1.
\eeq

Again, for simplicity of notation we will suppress the dependence of the mixing times on $k$ and $G$, and write $t_0:=t_0(k,|G_{\mathrm{ab}}|), t_1:=\log_k |G|$ and $t_*:=t_*(k,G)=\max\{ t_0, t_1\}$. 
It follows easily from \eqref{tv_decomp} that in order to prove Theorem \ref{mixingtime_ub} it suffices to prove the following two results.

\begin{prop}\label{nonzeroV}
For any $\ep>0$ and $t\geq (1+\ep)t_*$, we have 
\begin{align*}
|G|\cdot \P(X=X', V\neq 0|\typ)&= 1+o(1).
\end{align*}
\end{prop}

\begin{prop}\label{zeroV}
For any $\ep>0$ and $t\geq (1+\ep)t_*$, we have 
\begin{align*}
|G|\cdot \P(X=X',V=0| \typ)&=o(1).
\end{align*}
\end{prop}

The proofs of both Proposition  \ref{nonzeroV} and \ref{zeroV} rely on the analysis of the events $\{ G_\ell X(X')^{-1}=G_\ell\}_{\ell\in [L+1]}$. In the case where $V\neq 0$, writing $V=(V_1,\dots,V_k)$, one will see in Section \ref{proof_nonzeroV} that the analysis boils down to understanding the distribution of $ (G_{\ell+1}\sum_{a\in[k]} V_aZ_{a,\ell})_{\ell\in [L]}$. When $V=0$ the analysis is significantly more involved, as in this case \eqref{rearranged_S} turns into

\begin{align*}
X(X')^{-1}&=Z_1^{V_1}\cdots Z_k^{V_k} \prod_{a,b\in[k]: a<b}[Z_a,Z_b]^{m_{ba}} \varphi(Z_1,\dots,Z_k)= \prod_{a,b\in[k]: a<b}[Z_a,Z_b]^{m_{ba}} \varphi(Z_1,\dots,Z_k)
\end{align*}
and hence we need to carefully understand the distribution of the product of commutators. The proof of Proposition \ref{zeroV} is therefore considerably more involved, prompting us to provide an outline here that captures the main steps.
\medskip

\mn
\textbf{Proof outline of Proposition  \ref{zeroV}.}  To give an intuitive and brief explanation on why Proposition \ref{zeroV} is true, we will give the outline of the proof of Proposition \ref{zeroV} here. 
We will analyze $\{ G_{\ell} X(X')^{-1}= G_{\ell}\}$ for each $2\leq \ell \leq L+1$ when $V=0$.

The event $\{G_2X(X')^{-1}= G_2\}$ is already guaranteed to occur if we condition on $\{V=0\}$. In general, in order to understand each event $\{ G_{\ell+2}X(X')^{-1}= G_{\ell+2}\}$  for $1\leq \ell\leq L-1$, we hope to show that for all $1\leq \ell\leq L-1$, $G_{\ell+2}X(X')^{-1}$ is close to being uniform over $Q_{\ell+1}:=G_{\ell+1}/G_{\ell+2}$. To analyze the distribution of $G_{\ell+2}X(X')^{-1}$ for each $\ell\in [L-1]$, we observe from the representation of $X(X')^{-1}$ given in \eqref{rearranged_S} that $(m_{ba})_{a,b\in[k], a<b}$, which is defined in \eqref{mba}, plays a crucial role in determining the distribution.

As we will later show in Section \ref{first_level_filt},  after simplifying $G_{\ell+2}X(X')^{-1}$ for $ 2\leq \ell \leq L-1$ (the case where $\ell=1$ is somewhat different and will be discussed separately in Section \ref{sub_proof_mix_com}) one will eventually arrive at a term of the form
$$
G_{\ell+2} \sum_{b\in \mathcal{K}}[\chi_b,Z_{b,\ell}],
$$
where $\chi_b\in G$ satisfies $G_2 \chi_b=G_2(\sum_{a\in[k]:a<b} m_{ba}Z_{a,1}-\sum_{a\in [k]: b<a}m_{ab}Z_{a,1})$ and $\tildeK\subseteq[k]$ is a subset to be chosen, see \eqref{f_unknown}. The terms $(\chi_b)_{ b\in \mathcal{K}}$ explicitly indicate the role of $(m_{ba})_{a,b\in[k], a<b}$.
More specifically, since $(Z_{b,\ell})_{ b\in \mathcal{K}}$ is a collection of independent uniform random variables, Lemma \ref{unif_span} below asserts that the distribution of $G_{\ell+2} \sum_{b\in \mathcal{K}}[\chi_b,Z_{b,\ell}]$ is uniform over the subgroup generated, denoted by $\inprod{\{G_{\ell+2}[\chi_b,g]: b\in \mathcal{K},g \in R_\ell\}}$. As we would like $G_{\ell+2}X(X')^{-1}$ to be close to being uniform over $Q_{\ell+1}$, the choice of $\tildeK$ will be made so that the set  $\inprod{\{G_{\ell+2}[\chi_b,g]: b\in \mathcal{K},g \in R_\ell\}}$ takes up a sufficiently large fraction of $Q_{\ell+1}$. 

\bigskip
The above discussion suggests that we need to specify some conditions on $(m_{ba})_{a,b\in[k], a<b}$ to guarantee the existence of such an index set $\mathcal{K}$ and thus desired behavior of $\{G_{\ell+2}X(X')^{-1}\}_{2\leq \ell\leq L-1}$. These conditions are summarized in Definition \ref{eventA} as a ``good" event $\AA$. As a result, we will obtain the following estimate. Recall the definition of $h$ from Definition \ref{def_h}.
\begin{prop}\label{mix_com}
$|G|\cdot \P(X=X' |\AA, V=0,\typ)\ll e^h$.
\end{prop}

Lastly, we would like $\AA$ to be an event that occurs with sufficiently high probability, i.e.,
\begin{prop}\label{A_complement}
$|G|\cdot \P(\AA^c|V=0,\typ) \ll e^h$.
\end{prop}

\mn
\textit{Proof of Proposition \ref{zeroV} given Proposition \ref{mix_com} and \ref{A_complement}.} It follows from the definition of $\typ$ and a direct calculation that 
\begin{align*}
\P(V=0,\typ)&=\P(W=W', W\in \WW_{glo})=\sum_{w\in \WW_{glo}}\P(W=w)\P(W'=w)\leq e^{-h}.
\end{align*}
That is,
\beq\label{V=0}
\P(V=0|\typ)\leq e^{-h}/\P(\typ).
\eeq

Hence, by Proposition \ref{mix_com}, Proposition \ref{A_complement} and \eqref{V=0},
\begin{align*}
\P(X=X',V=0| \typ)&\leq \left(\P(X=X'|\AA, V=0,\typ)+\P(\AA^c|V=0,\typ)\right)\cdot \P(V=0|\typ)\ll \frac{1}{|G|}.
\end{align*}

\qed

The remaining of this section is organized as follows: in Section \ref{filtrations} we define filtrations that will serve as useful tools in the subsequent proofs; in Section \ref{proof_nonzeroV} we prove
Proposition \ref{nonzeroV}; in Section \ref{prelim_groups} we present some preliminary results on groups that are useful in the proof of Proposition \ref{zeroV}. As discussed above, the proof of Proposition \ref{zeroV} is divided into two parts: Proposition \ref{mix_com} will be proved in Section \ref{proof_mix_com} and Proposition \ref{A_complement} in Section \ref{proof_A_complement}.

\bigskip
\subsection{Useful Filtrations}\label{filtrations}

Recall that the random walk $X$ can be written as a sequence
$$X=\prod_{i=1}^N Z^{\eta_i}_{\sigma_i}=Z_{\sigma_1}^{\eta_1}Z_{\sigma_2}^{\eta_2}\cdots Z_{\sigma_N}^{\eta_N},$$
where $N$ is the number of steps taken by $X$ by time $t$ and $Z^{\eta_i}_{\sigma_i}$ denotes the $i$-th step taken by the random walk. Similarly, we write $X'=\prod_{i=1}^{N'} Z^{\eta'_i}_{\sigma'_i}$. Recall from Lemma \ref{Z_decomp} the decomposition $Z_a=\prod_{\ell=1}^L Z_{a,\ell}$ for $a\in [k]$. 

We will define the following $\sigma$-fields and events that will be useful in later analysis.

\medskip
\begin{definition}\label{filtration}
(i) Let $\widetilde\HH$ be the $\sigma$-field generated by $N, N',  (\sigma_i, \eta_i)_ { i\in [N+N']}$, i.e., $\widetilde\HH$ contains information on the sequences $X, X'$, other than the identities of $(Z_a)_{a\in[k]}$. \\
(ii) For each $\ell\in[L]$, let $\FF_\ell=\sigma(\{Z_{a,i}: a\in [k], 1\leq i\leq \ell\})$. Let $\FF_0$ be the trivial $\sigma$-field. \\
(iii) For $\ell\in [L+1]$, let $\EE_{\ell}=\{ X(X')^{-1}\in G_{\ell}\}$. 
\end{definition}

To make the intuitive picture a bit clearer, note that there are mainly two sources of randomness: \\
(i) At every step a generator is randomly chosen and applied, resulting in the random order of the sequences (encoded in $\widetilde \HH$);\\
(ii) For each $a\in [k]$, the specific choice of the generator $Z_a$ is random (encoded in $\{\FF_\ell :\ell\in [L]\}$). 

\begin{remark}
The following random variables are measurable with respect to $\widetilde \HH$:  $V$, $\1_\typ$ and $(m_{ba})_{ a,b\in [k], a<b}$.
\end{remark}

As preparation for later discussion, we first clarify the measurability of events of interest to us.
\begin{lemma}\label{E_measurability}
For $\ell\in [L]$, \\
(i) $\EE_{\ell}$ is measurable with respect to $\sigma(\FF_{\ell-1},\widetilde\HH)$.\\
(ii) $\{\EE_{\ell+1},V=0\}$ is measurable with respect to $\sigma(\FF_{\ell-1},\widetilde\HH)$.
\end{lemma}

\begin{proof}
Recall from \eqref{rearranged_S} we can write
$$
 X(X')^{-1}=Z_1^{V_1}\cdots Z_k^{V_k} \prod_{a,b\in[k]: a<b}[Z_a,Z_b]^{m_{ba}} \varphi(Z_1,\dots,Z_k),
$$
where $\varphi(Z_1,\dots,Z_k)$ is a product of multi-fold commutators of $\{Z^{\pm 1}_1,\dots,Z^{\pm 1}_k\}$ as a result of the rearranging. (See the end of Section \ref{representation_X} for a detailed description of $\varphi$ and multi-fold commutators.)
By Lemma \ref{Z_decomp}, for each $a\in [k]$, $Z_a$ can be decomposed into $Z_a=\prod_{\ell=1}^L Z_{a,\ell}$. For $a\in [k]$ and $\ell\in [L]$, we will write $Z_{a, < \ell}:= \prod_{j=1}^{\ell-1} Z_{a,j}$. 

Applying the decomposition $Z_a=\prod_{\ell=1}^L Z_{a,\ell}$ to the sequence above yields
\begin{align}\label{mathcalX}
\nonumber&G_{\ell+1}  X(X')^{-1}=G_{\ell+1}Z_1^{V_1}\cdots Z_k^{V_k} \prod_{a,b\in[k]: a<b}[Z_a,Z_b]^{m_{ba}} \varphi(Z_1,\dots,Z_k)\\
=&G_{\ell+1} \left(\prod_{a=1}^k Z^{V_a}_{a,\ell}\right)  \left(\prod_{a=1}^k Z^{V_a}_{a,<\ell}\right)\prod_{a,b\in[k]: a<b}\left [\prod_{i=1}^{\ell-1}Z_{a,i},\prod_{j=1}^{\ell-1}Z_{b,j}\right]^{m_{ba}} \varphi^{(\ell-2)}(\{ Z_{a,u}: a\in[k],u\leq \ell-2\})
\end{align}
where $\varphi^{(\ell-2)}(\{ Z_{a,u}: a\in[k],u\leq \ell-2\})$ is a certain polynomial with terms that are $i$-fold commutators of $\{ Z_{a,u}: a\in[k],u\leq \ell-2\}$ for $i\geq 3$, satisfying that 
 $$G_{\ell+1}\varphi^{(\ell-2)}(\{ Z_{a,u}: a\in[k],u\leq \ell-2\})=G_{\ell+1}\varphi(Z_1,\dots,Z_k).$$
The reason that we restrict our attention above to $u\leq \ell-2$ is because by Proposition \ref{strongly_lower_central} any $i$-fold commutator with $i\geq 3$ that involves $\{Z_{a,u}: a\in [k], u>\ell -2\}$ is in $G_{\ell+1}$. Also note that we can exchange the order of $Z_{a,\ell}$ and $Z_{a',<\ell}$ for any $a,a'\in [k]$ as $[Z_{a,\ell},Z_{a',<\ell}]\in G_{\ell+1}$ by Proposition \ref{strongly_lower_central}. 

On $\{V=0\}$ the right hand side of \eqref{mathcalX} only involves $\{ Z_{a,i} : 1\leq i \leq \ell-1\}$ and hence is measurable with respect to $\sigma(\FF_{\ell-1},\widetilde\HH)$. Therefore, $\{\EE_{\ell+1},V=0\}$ is measurable with respect to $\sigma(\FF_{\ell-1},\widetilde\HH)$, proving (ii).

Next we prove (i). Writing 
\begin{align}\label{f_ell}
\nonumber f^{(\ell-1)}=&f^{(\ell-1)}(\{ Z_{a,u}: a\in[k],u\leq \ell-1\})\\
\nonumber :=& \left(\prod_{a=1}^k Z^{V_a}_{a,<\ell}\right)\prod_{a,b\in[k]: a<b}\left [\prod_{i=1}^{\ell-1}Z_{a,i},\prod_{j=1}^{\ell-1}Z_{b,j}\right]^{m_{ba}} \\
&\quad\varphi^{(\ell-2)}(\{ Z_{a,u}: a\in[k],u\leq \ell-2\}),
\end{align}
by \eqref{mathcalX} we have 
\beq\label{simplified_X_nonzeroV}
G_{\ell+1}  X(X')^{-1}=G_{\ell+1}  \left(\prod_{a=1}^k Z^{V_a}_{a,\ell}\right) f^{(\ell-1)}.
\eeq
Note that $f^{(\ell-1)}$ is measurable with respect to $\sigma(\FF_{\ell-1},\widetilde\HH)$ as it only involves $\{ Z_{a,u} : 1\leq u \leq \ell-1\}$. 
Furthermore, since $\{Z_{a,\ell}: a\in[k]\} \subseteq G_{\ell}$, the same kind of simplification as \eqref{mathcalX} leads to
$$G_{\ell} X(X')^{-1}=G_{\ell} f^{(\ell-1)}(\{ Z_{a,u}: a\in[k],u\leq \ell-1\}),$$
which implies that $\EE_\ell=\{ f^{(\ell-1)}\in G_{\ell}\}$. Therefore, $\EE_{\ell}$ is measurable with respect to  $\sigma(\FF_{\ell-1},\widetilde\HH)$. 
\end{proof}

\subsection{Proof of Proposition \ref{nonzeroV}}\label{proof_nonzeroV}



Linear combinations of independent uniform random variables in an abelian group are themselves uniform on their support. As preparation for the proof we present the following lemma, which is a restatement of Lemma 2.11 and Lemma 6.5 of \cite{hermon2102cutoff}.
\begin{lemma}\label{gcd_unif}
Let $k\in\mathbb{N}$. Let $H$ be an Abelian group and $U_1,\dots,U_k\overset{iid}{\sim} \Unif(H)$. For $v=(v_1,\dots, v_k)\in \ZZ^k$, write $U=(U_1,\dots,U_k)$ and define $v\cdot U:=\sum_{i=1}^k v_iU_i$. We have
$$v \cdot U\sim \Unif (\gamma H) \quad \text{ where } \gamma=\gcd(v_1,\dots,v_k,|H|).$$
Consequently,
\beq\label{lemma6.5}
\max_{h\in H} \P( v\cdot U=h)=\P( v\cdot U=0)
\eeq
\end{lemma}
Applying Lemma \ref{gcd_unif} to $(Z_{1,\ell},\dots,Z_{k,\ell})$ and $H=Q_\ell$ for $\ell \in [L]$ and writing $\gcd(v,|Q_\ell|)=\gcd(v_1,\dots,v_k,|Q_\ell|)$ gives
\beq
G_{\ell+1}\sum_{a\in[k]} v_aZ_{a,\ell} \sim \Unif( \gcd(v,|Q_\ell|)Q_\ell).
\eeq

The key to the proof of Proposition \ref{nonzeroV} is the following estimate, whose proof uses the simplified expression of $\{G_{\ell+1}X(X')^{-1}\}_{\ell\in L}$ derived in the last section.

\begin{lemma}\label{layers_decomp}
We have
$$\P(X=X'|V)\leq \prod_{\ell=1}^L \P( G_{\ell+1}\sum_{a\in[k]} V_aZ_{a,\ell}= G_{\ell+1}|V).$$
\end{lemma}
\begin{proof}

Recall from the discussion in the proof of Lemma \ref{E_measurability} that $\EE_\ell=\{G_\ell X(X')^{-1}=G_\ell\}=\{ f^{(\ell-1)}\in G_{\ell}\}$ is measurable with respect to $\sigma(\FF_{\ell-1},\widetilde\HH)$ for $\ell\in [L]$ (where $\FF_0$ is the trivial $\sigma$-field). Since $\EE_{\ell}\subseteq \EE_{\ell+1}$, it follows from \eqref{simplified_X_nonzeroV} that for $\ell\in [L]$,
\begin{align*}
 \P(\EE_{\ell+1}| \FF_{\ell-1},\widetilde \HH)&=\P( \EE_\ell, G_{\ell+1} \left(\prod_{a=1}^k Z^{V_a}_{a,\ell}\right) f^{(\ell-1)}=G_{\ell+1}\bigg| \FF_{\ell-1},\widetilde \HH)\\
&\leq \1_{\EE_\ell}\cdot \max_{g\in G_\ell} \P( G_{\ell+1} \prod_{a=1}^k Z^{V_a}_{a,\ell} =G_{\ell+1}g\bigg| \FF_{\ell-1},\widetilde \HH)\\
&= \1_{\EE_\ell}\cdot \max_{g\in G_\ell} \P( G_{\ell+1}  \sum_{a\in [k]} V_aZ_{a,\ell}=G_{\ell+1}g\bigg| \FF_{\ell-1},\widetilde \HH),
\end{align*}
where the last line follows from rewriting the term $G_{\ell+1} \left(\prod_{a=1}^k Z^{V_a}_{a,\ell}\right)$ in the form of summation, as $Q_{\ell}:=G_{\ell}/G_{\ell+1}$ is abelian.
It then follows from  \eqref{lemma6.5} in Lemma \ref{gcd_unif} that 
\begin{align*}
 \P(\EE_{\ell+1}| \FF_{\ell-1},\widetilde \HH)&\leq   \1_{\EE_\ell}\cdot\max_{g\in G_\ell} \P(G_{\ell+1}  \sum_{a\in [k]} V_aZ_{a,\ell}=G_{\ell+1}g|\FF_{\ell-1},\widetilde \HH)\\
&=  \1_{\EE_\ell}\cdot \P(G_{\ell+1} \sum_{a\in [k]} V_aZ_{a,\ell}=G_{\ell+1}|\FF_{\ell-1},\widetilde \HH)\\
&= \1_{\EE_\ell}\cdot \P(G_{\ell+1} \sum_{a\in [k]} V_aZ_{a,\ell}=G_{\ell+1}|\widetilde \HH),
\end{align*}
where the last equality follows from observing that $G_{\ell+1} \sum_{a\in [k]} V_aZ_{a,\ell}$ is independent from $\FF_{\ell-1}$.
It then follows from the tower property that
\begin{align}\label{iteration}
\nonumber\P(\EE_{\ell+1}|\widetilde \HH)&=\E[\P(\EE_{\ell+1}| \FF_{\ell-1},\widetilde \HH)|\widetilde\HH]\\
&\leq \E[  \1_{\EE_\ell}\cdot \P(G_{\ell+1} \sum_{a\in [k]} V_aZ_{a,\ell}=G_{\ell+1}|\widetilde \HH)|\widetilde\HH]=\P(\EE_{\ell}|\widetilde \HH)\cdot \P(G_{\ell+1} \sum_{a\in [k]} V_aZ_{a,\ell}=G_{\ell+1}|\widetilde \HH).
\end{align}
Applying \eqref{iteration} iteratively for $\ell\in[L]$, we obtain
\begin{align*}
\P(X=X'|\widetilde\HH)&=\P(\EE_{L+1}|\widetilde\HH)\leq \prod_{\ell=1}^L \P(G_{\ell+1} \sum_{a\in [k]} V_aZ_{a,\ell}=G_{\ell+1}|\widetilde \HH).
\end{align*}
As $V$ is measurable with respect to $\widetilde\HH$, i.e., $\sigma(V)\subseteq \widetilde\HH$, by the tower property we have the desired result
$$\P(X=X'|V)\leq \prod_{\ell=1}^L \P( G_{\ell+1}\sum_{a\in[k]} V_aZ_{a,\ell}= G_{\ell+1}|V).$$

\end{proof}

Given Lemma \ref{layers_decomp}, we can obtain the following upper bound involving the greatest common divider. 
\begin{lemma}\label{gcd_expectation}
Let $\bar r:=\sum_{\ell=1}^L r_\ell$ and $\gcd(v):=\gcd(v_1,\dots,v_k,|G|)$. We have
$$|G|\cdot \P(X=X', V\neq 0|\typ)\leq \E\left[ ( \gcd(V))^{ \bar r} \mathbf{1}\{V\neq 0\}|\typ\right]$$
\end{lemma}
\begin{proof}
For $v\in \ZZ^k\backslash\{0\}$, it follows from Lemma \ref{layers_decomp} that
\begin{align*}
\P(X(X')^{-1}=\id| V=v)&\leq \prod_{\ell=1}^L \P(G_{\ell+1} \sum_{a\in[k]} v_aZ_{a,\ell}=G_{\ell+1})\leq \prod_{\ell=1}^L \frac{1}{|\gcd(v,|Q_\ell|) Q_\ell|}\\
&\leq  \prod_{\ell=1}^L \frac{(\gcd(V,|Q_\ell|))^{r_\ell}}{|Q_\ell|},
\end{align*}
where the last inequality follows from the fact that for an abelian group $H$ and $\gamma\in \mathbb{N}$, $|H|/|\gamma H|\leq \gamma^{r(H)}$, see Lemma 2.12 in \cite{hermon2102cutoff} for instance. 

Recalling from Definition \ref{filtration} the definition of $\widetilde\HH$, one then has that
\begin{align}\label{gcd_bound}
\nonumber|G|\cdot \P(X=X', V\neq 0|\widetilde \HH)&\leq |G|\cdot  \prod_{\ell=1}^L \frac{(\gcd(V,|Q_\ell|))^{r_\ell}\cdot \1\{V\neq 0\}}{|Q_\ell|}\\
&= \prod_{\ell=1}^L (\gcd(V,|Q_\ell|))^{r_\ell}\cdot \1\{V\neq 0\}\leq (\gcd(V))^{\bar r}\1\{V\neq 0\},
\end{align}
where the last inequality follows from the fact that $\gcd(V,|Q_\ell|)\leq \gcd(V)$ for all $1\leq \ell\leq L$. Since $\typ\in \widetilde\HH$, integrating both sides of \eqref{gcd_bound}  over $1_\typ$ gives the desired bound
$$|G|\cdot \P(X=X', V\neq 0|\typ)\leq \E\left[  (\gcd(V))^{\bar r}\1\{V\neq 0\}|\typ\right].$$
\end{proof}

\begin{lemma}\label{gcd_ub}
Fix any $\ep>0$ and let $s:=t/k$ for $t\geq  (1+\ep)t_*$. For all $\gamma\geq 2$ and all $\lambda>0$, there exists a constant $\tilde\delta_\lambda\in (0,1)$ that depends only on $\lambda$ such that
$$\P(\gcd(V)=\gamma, V\neq 0|\typ)\lesssim
\begin{cases}
\min\{ e^{-t} , \frac{ k(2es)^\gamma}{\gamma^\gamma}\}& \text{ when }s\ll 1,\\
(1-\tilde\delta_\lambda)^k&  \text{ when }s\geq f(\lambda),\\
(1/\gamma+s^{-1/2})^k& \text{ when }s\gg 1,
\end{cases}
$$
where $f(\lambda)$ is defined as in \eqref{flambda}.
\end{lemma}
\begin{proof}
\mn
\textbf{Regime: $s\ll 1$.} Let $N_1(s)$ denote a rate 1 Poisson process. Recall that $V_1$ is a rate $2/k$ simple random walk on $\ZZ$. For a random walk to return to the origin it must have taken an even number of steps, which means
$$\P(V_1(t)=0)\leq \P(N_1(2s)\in 2\ZZ_+)\leq \P(N_1(2s)=0)+\sum^\infty_{m=1} \P(N_1(2s)=2m)\leq e^{-2s}+8s^2.$$

To simplify notation, write $\gcd:=\gcd(V)$ from now on. Using a Chernoff bound argument on the Poisson random variable $N_1(2s)$ yields for $\gamma\geq 2$
\begin{align*}
\P(\gcd=\gamma,V\neq 0|\typ)&\leq \P(\gcd=\gamma|\typ)\\
&\lesssim (\P(V_1=0)+\P(N_1(2s)\geq \gamma))^k\leq \left( e^{-2s}+8s^2+ e^{-2s}\cdot\frac{ (2es)^\gamma}{\gamma^\gamma}\right)^k\\
&\leq (e^{-2s}+8s^2+(es)^2)^k\leq (1-s)^k\leq e^{-t}.
\end{align*}

To prove the second part of the upper bound, note that for $\{\gcd=\gamma,V\neq 0\}$  to occur, there must be some $a\in [k]$ such that $N_a(2s)\neq 0$ and $\gamma$ divides $N_a(2s)$, which implies that $N_a(2s)\geq \gamma$. Hence,
$$\P(\gcd=\gamma,V\neq 0|\typ)\lesssim k\cdot \P(N_1(2s)\geq \gamma)\leq k\cdot \frac{ (2es)^\gamma e^{-2s}}{\gamma^\gamma}\leq \frac{ k(2es)^\gamma}{\gamma^\gamma}.$$

\mn
\textbf{Regime: $s\geq f(\lambda)$.} Define $c_{\gamma}:= \P(V_1(t)\in \gamma \ZZ)$. Since $s\geq f(\lambda)$, we have $\P(V_1(t)=0)\leq \tilde c_\lambda< 1$ for some constant $\tilde c_\lambda$ depending on $\lambda$. So we can fix some small $\delta_\lambda>0$ and some large $\gamma_\lambda\in \mathbb{N}$ such that $\P(V_1=0)+1/\gamma_\lambda<1-\delta_\lambda$. That is, for all $\gamma\geq \gamma_\lambda$, 
$$c_\gamma=\P(V_1\in \gamma \ZZ)\leq \P(V_1=0)+\P(V_1\in \gamma \ZZ|V_1\neq 0)\leq \P(V_1=0)+1/\gamma_\lambda<1-\delta_\lambda.$$
Thus there is some $\tilde \delta_{\lambda}>0$  so that
$$\max_{\gamma \geq 2} c_\gamma\leq \max\{ \max_{2\leq \gamma \leq \gamma_\lambda} c_\gamma, 1-\delta_\lambda\}\leq 1-\tilde\delta_\lambda.$$
Therefore, for $\gamma\geq 2$,
$$\P(\gcd=\gamma,V\neq 0|\typ)\lesssim \P(V_1\in \gamma \ZZ)^k\leq (1-\tilde\delta_\lambda)^k.$$

\mn
\textbf{Regime: $s\gg 1$.}
Note that
$$\P(\gcd=\gamma,V\neq 0|\typ)\lesssim \P(V_1\in \gamma \ZZ)^k\leq (\P(V_1=0)+\P(V_1\in \gamma \ZZ|V_1\neq 0))^k.$$
For the second term it follows from Lemma 2.14 of \cite{hermon2102cutoff} that $\P(V_1\in \gamma \ZZ|V_1\neq 0)\leq 1/\gamma$. As $V_1$ is a one dimension SRW with rate $2/k$, Theorem A.4 of \cite{hermon2018supplementary} implies that when $s\gg1$ we have
$$\P(V_1(t)=0)\leq \frac{1}{\sqrt{2\pi (2s)}}\exp\left( \mathcal{O}\left(\frac{1}{\sqrt{2s}}\right)\right)\leq s^{-1/2}.$$ 
Hence,
$$\P(\gcd=\gamma, V\neq 0|\typ)\lesssim (\P(V_1=0)+1/\gamma)^k \leq (1/\gamma+s^{-1/2})^k.$$

%
%
%
%
\end{proof}

\begin{lemma}\label{gcd_analysis} 
Suppose $1\ll \log k \ll \log|G|$. For any $\ep>0$, when $t\geq (1+\ep)t_*$ we have
$$\E\left[  (\gcd(V))^{\bar r}\1\{V\neq 0\}|\typ\right]=1+o(1).$$
\end{lemma}
\begin{proof} Again we will write $\gcd=\gcd(V)$.

\mn
\textbf{Regime: $1\ll k \ll \log|G_{\mathrm{ab}}|$.} Observe that $s:=t/k\geq |G_{\mathrm{ab}}|^{2/k}\gg1$. On $\typ$ we have $\gcd$ is at most $2r_*=|G_{\mathrm{ab}}|^{1/k}(\log k)^2$, which gives
$$\E\left[  \gcd^{\bar r}1\{V\neq 0\}|\typ\right]=1+\sum_{\gamma=2}^{2r_*}\gamma^{\bar r}\P(\gcd=\gamma, V\neq 0|\typ).$$
Let $\delta\in (0,1)$ be sufficiently small. For $2\leq \gamma \leq \delta |G_{\mathrm{ab}}|^{1/k}$, applying Lemma \ref{gcd_ub} gives
$$\P(\gcd=\gamma,V\neq 0|\typ)\lesssim (1/\gamma+1/|G_{\mathrm{ab}}|^{1/k})\leq (1/\gamma+ 1/(\gamma/\delta))^k=(1+\delta)^k/\gamma^k.$$
For $\gamma>\delta|G_{\mathrm{ab}}|^{1/k}$, we use the bound $(a+b)^k\leq 2^k(a^k+b^k)$ to get 
$$\P(\gcd=\gamma|\typ)\lesssim2^k(1/\gamma^k+1/|G_{\mathrm{ab}}|).$$
Therefore, 
\begin{align*}
\E\left[  \gcd^{\bar r}\1\{V\neq 0\}|\typ\right]-1&\lesssim \sum_{\gamma=2}^{\delta|G_{\mathrm{ab}}|^{1/k}}\frac{(1+\delta)^k}{\gamma^{k-\bar r}}+\sum_{\gamma=\delta|G_{\mathrm{ab}}|^{1/k}+1}^{ 2r_*} \gamma^{\bar r}2^k(1/\gamma^k+1/|G_{\mathrm{ab}}|)\\
&\lesssim e^{\delta k}2^{\bar r+1-k}+ 2^k (\delta |G_{\mathrm{ab}}|^{1/k})^{\bar r+1-k}+2^k(\log k)^{2(\bar r+1)} |G_{\mathrm{ab}}|^{(\bar r+1-k)/k}\\
&=o(1)
\end{align*}
as $\bar r \asymp 1$ and $k\ll \log|G_{\mathrm{ab}}|$.

\mn
\textbf{Regime: $k\eqsim \lambda \log|G_{\mathrm{ab}}|$.} In this regime $s\geq (1+\ep)t_0/k\geq f(\lambda)$. It follows from Lemma \ref{gcd_ub} that there exists $\tilde \delta_\lambda \in (0,1)$ such that
\begin{align*}
\E\left[  \gcd^{\bar r}\1\{V\neq 0\}|\typ\right]-1&\leq \sum_{\gamma=2}^{2r_*} \gamma^{\bar r} (1-\tilde\delta_\lambda)^k\\
&\lesssim (\log k)^{2(\bar r+1)} (1-\tilde\delta_\lambda)^k=o(1).
\end{align*}

\mn
\textbf{Regime: $k\gg \log|G_{\mathrm{ab}}|$.} In this regime $t_*/k\ll 1$ and thus for $t\geq (1+\ep)t_*$ there are two regimes of $s=t/k$ to be discussed: $s\ll 1$ and $s\gtrsim 1$. When $s\gtrsim 1$, by the same argument as before we can show $\E\left[  \gcd^{\bar r}\1\{V\neq 0\}|\typ\right]-1=o(1)$. 

It remains to treat the case where $s\ll 1$.
When $ \frac{\log \log k}{k}\ll s\ll1$, we apply the bound $\P(\gcd(V)=\gamma, V\neq 0|\typ)\lesssim e^{-t}$  in Lemma \ref{gcd_ub} to show
\begin{align*}
\E\left[  \gcd^{\bar r}\1\{V\neq 0\}|\typ\right]-1&\leq \sum_{\gamma=2}^{2r_*} \gamma^{\bar r} e^{-t}\lesssim (2r_*)^{\bar r+1} e^{-t}\lesssim (\log k)^{2(\bar r+1)}e^{-t}=o(1).
\end{align*}
When $ s\lesssim \frac{\log \log k}{k}$, we apply the  bound $\P(\gcd(V)=\gamma, V\neq 0|\typ)\lesssim\frac{ k(2es)^\gamma}{\gamma^\gamma}$ in Lemma \ref{gcd_ub} to show
\begin{align*}
\E\left[  \gcd^{\bar r}\1\{V\neq 0\}|\typ\right]-1&\lesssim \sum_{\gamma=2}^{2r_*} \gamma^{\bar r} \frac{ k (2es)^\gamma}{\gamma^\gamma}\leq \sum_{\gamma=2}^{\bar r} \gamma^{\bar r-\gamma} k (2es)^2+\sum_{\bar r<\gamma\leq 2r_*} k(2es)^\gamma\\
&\lesssim C_{\bar r} ks^2 +Cks^{\bar r}\lesssim C'_{\bar r} k s^2\lesssim \frac{(\log\log k)^2}{k}=o(1).
\end{align*}

\end{proof}

The proof of Proposition \ref{nonzeroV} is completed given Lemma \ref{gcd_expectation} and Lemma \ref{gcd_analysis}.

\subsection{Preliminary Results on Groups}\label{prelim_groups}

Throughout this paper, the notation $\inprod{ A}$ denotes the subgroup generated by the set of elements $A\subseteq G$. Recall from Definition  \ref{representatives} that $R_\ell$ are the representatives of $Q_{\ell}=G_{\ell}/G_{\ell+1}$.

\begin{lemma}\label{unif_span}
For fixed $h_1,\dots,h_n\in R_1$ and $U_1,\dots,U_n \overset{iid}{\sim} \Unif(R_\ell)$ with $\ell\geq 1$, we have 
$$G_{\ell+2}\sum_{i\in [n]} [h_i,U_i]\sim \Unif\left( \inprod{\{ G_{\ell+2}[h_i,g]: i\in[n],g\in R_\ell\}}\right).$$
\end{lemma}
\begin{proof}
Write $R_\ell=\{g_u: u\in [|R_\ell|]\}$. Any $G_{\ell+2}x\in  \inprod{\{ G_{\ell+2}[h_i,g]: i\in[n],g\in R_\ell\}}$ can be expressed as 
$$G_{\ell+2}x=G_{\ell+2}\sum_{i\in [n], u\in [|R_\ell|]} c_{i,u} [h_i,g_u]=G_{\ell+2}\sum_{i\in [n]} \left[h_i, \sum_{u\in 
[|R_\ell|]}c_{i,u}g_u\right]$$ 
for some integer coefficients $\{c_{i,u}: i\in [n], u\in [|R_\ell|]\}$. Note that since $R_{\ell}$ is a representative set of $Q_\ell$, there exists a $g'_i\in R_\ell$ such that $G_{\ell+1}\sum_{u\in [|R_\ell|]} c_{i,u}g_u=G_{\ell+1} g'_i$, and thus by Proposition \ref{strongly_lower_central},
$$G_{\ell+2}\sum_{i\in [n]} \left[h_i, \sum_{u\in 
[|R_\ell|]}c_{i,u}g_u\right]=G_{\ell+2}[h_i,g'_i].$$
Therefore, for any $x\in   \inprod{\{ G_{\ell+2}[h_i,g]: i\in[n],g\in R_\ell\}}$, there exists $g'_1,\dots,g'_n\in R_\ell$ such that
$$G_{\ell+2}x=G_{\ell+2}\sum_{i \in [n]} [h_i,g'_i].$$
Hence, for any $x\in   \inprod{\{ G_{\ell+2}[h_i,g]: i\in[n],g\in R_\ell\}}$, by Proposition \ref{p:bilinearity}, 
\begin{align*}
\P(G_{\ell+2}\sum_{i\in[n]}[h_i,U_i]=G_{\ell+2}x)&=\P(G_{\ell+2}\sum_{i\in[n]} [h_i,U_i\cdot (g'_i)^{-1}]=G_{\ell+2})=\P(G_{\ell+2}\sum_{i\in [n]} [h_i,U_i]=G_{\ell+2}),
\end{align*}
which proves the uniformity of  $G_{\ell+2}\sum_{i\in [n]} [h_i,U_i]$.
\end{proof}

\begin{lemma}\label{iid_generation}
Let $n,\alpha \in \mathbb{N}$ be  fixed  and $p$ be a prime. Suppose $U_1,U_2,\dots,U_{n+2}$ are i.i.d. uniform random variables over $\ZZ_{p^\alpha}$. We have 
$$\E\left[\left(\frac{|\ZZ_{p^\alpha}|}{|\langle U_1,\dots,U_{n+2}\rangle|}\right)^n\right]\leq \exp\left(\frac{1}{p^2-1}\right).$$
\end{lemma}
\begin{proof}
Let $A_i=\{ \text{$U_1,\dots,U_{n+2}$ are all divisible by $p^i$}\}$ for $0\leq i\leq \alpha$ and $A_{\alpha+1}=\emptyset$. On the event $A_i\backslash A_{i+1}$ we know $\frac{|\ZZ_{p^\alpha}|}{|\langle U_1,\dots,U_{n+2}\rangle|}=p^i$. Also note that $\P(A_i)=p^{-(n+2)i}$ for $0\leq i\leq \alpha$. It follows that 
\begin{align*}
\E\left[\left(\frac{|\ZZ_{p^\alpha}|}{|\langle U_1,\dots,U_{n+2}\rangle|}\right)^n\right]&=\sum_{i=0}^{\alpha} p^{ni} \P(A_i\backslash A_{i+1})\\
&\leq \sum_{i=0}^{\alpha} p^{ni} \cdot p^{-(n+2)i}\leq \frac{1}{1-p^{-2}}\leq \exp\left(\frac{1}{p^2-1}\right).
\end{align*}
\end{proof}

\begin{lemma}\label{quotient_size}
Let $H$ be a subset of $G$. For any $\ell\in[L-1]$,
$$\frac{|Q_{\ell+1}|}{|\inprod{\{G_{\ell+2}[h,g]: h\in H, g\in R_\ell\}}|}\leq\left(\frac{|G_{\mathrm{ab}}|}{|\inprod{\{G_2h: h\in H\}}|}\right)^{r_{\ell+1}}.$$
\end{lemma}
\begin{proof}
For simplicity of notation, write $N:=\inprod{\{G_2h: h\in H\}}$ and then let $\lambda=|G_{\mathrm{ab}}|/|N|$. Since $G_{\mathrm{ab}}$ is abelian, it can be expressed in the form
$$G_{\mathrm{ab}}=\ZZ_{m_1}\oplus\cdots\oplus \ZZ_{m_r}$$
for some $m_1,\dots,m_r\in \mathbb{N}$ where $r$ is the rank of $G$. This decomposition allows us to see that $\lambda G_{\mathrm{ab}} \trianglelefteq N$. As a consequence, 
$$ \inprod{ \{G_{\ell+2}[h,g]: G_2h\in \lambda G_{\mathrm{ab}},g\in R_\ell\}} \trianglelefteq   \inprod{\{G_{\ell+2}[h,g]: h\in H, g\in R_\ell\}}$$

Note that 
$$ \inprod{ \{G_{\ell+2}[h,g]: G_2h\in \lambda G_{\mathrm{ab}},g\in R_\ell\}}=\inprod{\{\lambda G_{\ell+2} [h,g]: h\in G, g\in R_\ell\}}=\lambda Q_{\ell+1}.$$
We then have 
\begin{align*}
\frac{|Q_{\ell+1}|}{|\inprod{\{G_{\ell+2}[h,g]: h\in H, g\in R_\ell\}}|}&\leq \frac{|Q_{\ell+1}|}{\inprod{ \{G_{\ell+2}[h,g]: G_2h\in \lambda G_{\mathrm{ab}},g\in R_\ell\}}}\\
&=\frac{|Q_{\ell+1}|}{|\lambda Q_{\ell+1}|}\leq  \lambda^{r_{\ell+1}}.
\end{align*}

\end{proof}

Let $R$ be a non-trivial commutative ring with identity, and let $M =(m_{ij})_{n\times n}$ be a matrix over $R$. For a  maximal ideal $\mathcal{I}$ of $R$, let $\pi:R \to R/\mathcal{I}$ be the natural homomorphism. Let $\pi^n: R^n \to (R/\mathcal{I})^n$ be defined as $\pi^n(g_1,\dots,g_n)=(\pi(g_1),\dots,\pi(g_n))$ for $g_1,\dots,g_n\in R$. The following result comes from the theory of matrices over a commutative ring and it is stated in \cite[p. 259]{ching1977linear}.  
\begin{prop}\label{iff_condition}
If $M: R^n\to R^n$ is a homomorphism, then $M$ is surjective if and only if for every maximal ideal $\mathcal{I}$ of $R$, the map $\pi_M:(R/\mathcal{I})^n\to (R/\mathcal{I})^n$ is surjective, where $ \pi_M\circ \pi^n=\pi^n \circ M$.

\end{prop}

We will be interested in $R=\ZZ_{p^\alpha}$ where $p$ is a prime. The unique maximal ideal of $\ZZ_{p^\alpha}$ is $\mathcal{I}=\inprod{p}$. Then $\ZZ_{p^\alpha}/\mathcal I\cong \ZZ_p$.

\begin{lemma}\label{uniform_chi}
Let $\alpha,n \in\mathbb{N}$ and  $U:=(U_1,\dots,U_{n})$ where $U_i \overset{iid}{\sim} \Unif(\ZZ_{p^{\alpha}})$. Let $M$ denote a $n\times n$ matrix over $\ZZ_{p^{\alpha}}$. If $M: (\ZZ_{p^{\alpha}})^{n}\to  (\ZZ_{p^{\alpha}})^{n}$ is surjective, then $\tilde U:=MU$ has independent entries that are uniform over $\ZZ_{p^{\alpha}}$.
\end{lemma}
\begin{proof}
Define a surjective group homomorphism $f:  (\ZZ_{p^{\alpha}})^{n} \to(\ZZ_{p^{\alpha}})^{n}$ by $f(\bm x):=M \bm x$ where $\bm x\in (\ZZ_{p^{\alpha}})^{n}$. Since $f$ is surjective, for  any $\bm y\in(\ZZ_{p^{\alpha}})^{n}$ there exists some $\bm y'\in (\ZZ_{p^{\alpha}})^{n} $ such that $f(\bm y')=\bm y$. It follows that for any $\bm y\in(\ZZ_{p^{\alpha}})^{n}$
\begin{align*}
\P(\tilde U=\bm y)&=\P(f(U)=\bm y)=\P(f(U-\bm y')=0)=\P(f(U)=0)=\P(\tilde U=0)\\
&=\frac{ |Ker(f)|}{|(\ZZ_{p^{\alpha}})^{n}|}=\frac{1}{|Im(f)|}=\frac{1}{p^{\alpha n}},
\end{align*}
where the last equality uses the fact that $f$ is surjective.
\end{proof}

\subsection{Proof of Proposition \ref{mix_com}}\label{proof_mix_com}

\subsubsection{Definitions}\label{def_A&K}

To prepare for the proof of Proposition \ref{mix_com}, we first introduce several useful quantities and describe the selection of the good event $\AA$ and index set $\mathcal{K}\subseteq [k]$, the reason for whose definition will become clearer as we proceed with the proof of Proposition \ref{mix_com}.

To simplify notation we define the following matrix.
\begin{definition}\label{mhat}
Recall the definition of $(m_{ba})_{ a,b\in [k], a<b}$ from \eqref{mba}. Define
\beq\label{hatm}
\hat m_{ba}=
\begin{cases}
m_{ba} & \text{ if }a<b\\
-m_{ab} & \text{ if }b<a\\
0 & \text{ if }b=a.
\end{cases}
\eeq
\end{definition}

\begin{definition}\label{chi&psi}
Let $A=(A_{ba})_{a,b\in [k]}$ be  a $k\times k$ matrix with entries in $\ZZ$, and let $\mathcal{K}:=\mathcal{K}(A)\subseteq [k]$ denote a subset of indices which is known when $A$ is given. For $b\in [k]$, we define $\chi_b(A)$, and respectively $\psi_b(A)$, to be an element in $G$ that satisfies
$$G_2\chi_b(A)=G_2\left(\sum_{a\in[k]} A_{ba}Z_{a,1}\right),$$
and respectively
$$
G_2\psi_b(A)=G_2\left(\sum_{a\in[k]:a<b} A_{ba}Z_{a,1}+\sum_{a\in \tildeK^c:a>b} A_{ba} Z_{a,1}\right).
$$
In particular, for $b\in [k]$, define $\chi_b:=\chi_b(\hat m)$ and $\psi_b:=\psi_b(\hat m)$.

\end{definition}

Since by definition
$$G_2 \chi_b=G_2 \left(\sum_{a\in [k]} \hat m_{ba} Z_{a,1}\right),$$
with slight abuse of notation we will write 
\beq\label{chib}
\chi_b=\sum_{a\in [k]} \hat m_{ba} Z_{a,1} \quad\quad  \text{for }b\in [k],
\eeq
so that we can write, again with slight abuse of notation,
$$\chi:=(\chi_1,\dots,\chi_k)^T=\hat m (Z_{1,1},\dots,Z_{k,1})^T,$$ 
which is the product of the matrix $\hat m$ and the vector $ (Z_{1,1},\dots,Z_{k,1})$.

\bigskip
Let 
\beq\label{size_K}
K:=\sum_{\ell=2}^L r_\ell+2.
\eeq
For a $K\times K$ submatrix $M$ of $\hat m$, define the matrix $M_p:= M \mod p$, i.e., $M_p(i,j):=M(i,j)\mod p$ for the $(i,j)$-th entry in $M$. Note that $M_p$ is a matrix over the field $\mathbb{F}_p$ and hence its row rank is defined as the number of linearly independent rows  in the matrix.  


\begin{definition}\label{eventA}

A fixed $k\times k$ matrix $(A_{ba})_{a,b\in [k]}$ is said to be \textit{good} if it satisfies the following two conditions: 
\begin{enumerate}
\item[(i)]  There exists a set  $\tildeK\subseteq[k]$ such that
\beq\label{def_K}
\{G_2\psi_b(A), G_2\chi_b(A): b\in\tildeK\} \text{ are independent from }\{G_2Z_{b,1}: b\in\tildeK\},
\eeq
\item[(ii)] Let $\Gamma:=\{ p: p \text{ is a prime that divides } |G_{\mathrm{ab}}|\}$. For each $p\in \Gamma$ there exists a $K\times K$ submatrix $M$ of $(A_{ba})_{b\in\tildeK, a\in[k]}$ such that $ M_p:= M \mod p$ has rank $K$ over the field $\mathbb{F}_p$, where as above $K=\sum_{\ell=2}^L r_\ell+2$.
\end{enumerate}
Define $\AA:=\{  \text{$\hat m$ is a good matrix}\}$ and let $\KK$ be the corresponding subset of indices satisfying (i).
\end{definition}
Note that by definition both $\AA$ and $\KK$ are measurable with respect to $\widetilde\HH$. 

\subsubsection{Outline of Proof}
To prove Proposition \ref{mix_com}, we derive an upper bound on $\P(X=X' |\AA, V=0,\typ)$ inductively using the following proposition.
\begin{prop}\label{span_size}
Let $\KK\subseteq [k]$ be measurable with respect to $\widetilde \HH$. For $2\leq \ell\leq L-1$, letting 
$$
H_{\KK,\ell+1}:=\inprod{\{ G_{\ell+2}[\chi_b,g]: b\in \mathcal{K},g \in R_\ell\}},
$$
we have
$$
\1\{\EE_{\ell+1},V=0\}\cdot\P(\EE_{\ell+2}|\FF_{\ell-1},\widetilde\HH)\leq \1\{\EE_{\ell+1},V=0\}\cdot|H_{\KK,\ell+1}|^{-1}.
$$
\end{prop}

Applying Lemma \ref{quotient_size} to $H_{\KK,\ell+1}$ gives
\beq\label{size_of_H}
|H_{\KK,\ell+1}|^{-1}\leq \frac{1}{|Q_{\ell+1}|}\left( \frac{|G_{\mathrm{ab}}|}{|\inprod{ \{G_2\chi_b :b\in \mathcal{K}\}}|}\right)^{r_{\ell+1}}.
\eeq
Our goal is to choose $\KK\subseteq[k]$ properly so that $\inprod{ \{G_2\chi_b :b\in \mathcal{K}\}}$ is sufficiently large compared to $G_{\mathrm{ab}}$. To guarantee such a choice of  $\KK$ exists, we further define a ``good" event $\AA$, see Definition \ref{eventA}, which is measurable with respect to $\widetilde\HH$ and occurs with high probability. 

Note that $\AA, \typ$ and $\{V=0\}$ are measurable with respect to $\widetilde\HH$. By the tower property of conditional expectation and the fact that $\sigma(\FF_1,\widetilde \HH)\subseteq \sigma(\FF_{\ell-1},\widetilde\HH)$ for $2\leq \ell \leq L-1$, Proposition \ref{span_size} leads to
\begin{align*}
 \P(\EE_{\ell+2}, \AA,V=0,\typ|\FF_1,\widetilde\HH)&=\E \left[  \P( \EE_{\ell+2},\AA,V=0,\typ|\FF_{\ell-1},\widetilde\HH) \big| \FF_1,\widetilde\HH\right]\\
\nonumber &\leq  |H_{\KK,\ell+1}|^{-1}\cdot \P(\EE_{\ell+1},\AA,V=0,\typ|\FF_1,\widetilde\HH),
\end{align*}
which implies 
\begin{align*}
\P( \EE_{L+1},\AA,V=0,\typ|\FF_1,\widetilde\HH)&=\left( \prod_{\ell=2}^{L-1}  |H_{\KK,\ell+1}|^{-1}  \right) \cdot \P(\EE_{3},\AA,V=0,\typ|\FF_1,\widetilde\HH).
\end{align*}
Combined with \eqref{size_of_H}, the above yields
\begin{align}\label{keystep_prop}
\nonumber&\P( \EE_{L+1},\AA,V=0,\typ|\FF_1,\widetilde\HH)\\
\leq & \left(\prod_{\ell=2}^{L-1} \frac{1}{ |Q_{\ell+1}|} \right)\cdot \left( \frac{|G_{\mathrm{ab}}|}{|\inprod{ \{G_2\chi_b :b\in \mathcal{K}\}}|}\right)^{\sum_{\ell=2}^{L-1}r_{\ell+1}} \P(\EE_{3},\AA,V=0,\typ|\FF_1,\widetilde\HH).
\end{align}

Upper bounding the expectation of \eqref{keystep_prop} is the key to proving Proposition \ref{mix_com}. The choice of $\AA$ and $\KK$ in Definition \ref{eventA} is made so that the expectation of the right hand side of \eqref{keystep_prop} leads to the desired result.

In Section \ref{first_level_filt} we prove Proposition \ref{span_size}. We complete the proof of Proposition \ref{mix_com} in Section \ref{sub_proof_mix_com}  and finish the proof of a key lemma in Section \ref{proof_size_estimate}.

\subsubsection{Proof of Proposition \ref{span_size}}\label{first_level_filt}

The key to proving Proposition \ref{span_size} is the simplification of $G_{\ell+2}X(X')^{-1}$. The analysis in this section is somewhat similar to that in Section \ref{filtrations}, where we obtained a simplified expression of $G_{\ell+1}X(X')^{-1}$ when $V\neq 0$. However, when $V=0$ the result of simplification is quite different. Instead of a simple quantity of the form $G_{\ell+1}\sum_{a\in[k]} V_aZ_{a,\ell}$ (see Lemma \ref{layers_decomp}), we now have to deal with an expression involving commutators of $\{Z_{a,u}: a\in [k], u\in [L]\}$.

Recall from Definition \ref{filtration} the definition of $\widetilde\HH$, $\{ \FF_\ell :\ell \in [L]\}$ and $\{\EE_{\ell}: \ell \in [L+1]\}$. Define
\beq
 \mathbf{X}^{(\ell+1)}:=\prod_{a,b\in[k]: a<b}\left [\prod_{i=1}^{\ell-1}Z_{a,i},\prod_{j=1}^{\ell-1}Z_{b,j}\right]^{m_{ba}}  \varphi^{(\ell-2)}(\{ Z_{a,u}: a\in[k],u\leq \ell-2\}),
\eeq
which comes from the right hand side of \eqref{mathcalX}.
 
\begin{lemma}\label{simplification_X}
On $\{\EE_{\ell+1},V=0\}$ we have 
$$G_{\ell+2}X(X')^{-1}=G_{\ell+2}\sum_{a,b\in[k]: a<b}m_{ba}([Z_{a,1},Z_{b,\ell}]+[Z_{a,\ell},Z_{b,1}])+G_{\ell+2}(\tilde\varphi \mathbf{X}^{(\ell+1)}),$$
where $G_{\ell+2}(\tilde\varphi \mathbf{X}^{(\ell+1)})\in Q_{\ell+1}$ is measurable with respect to $\sigma(\FF_{\ell-1},\widetilde\HH)$ and  $\tilde \varphi:=\tilde\varphi(\{ Z_{a,u}: a\in[k],u\leq \ell-1\})$ is a polynomial  whose definition will be clarified in the proof.
\end{lemma}

\begin{proof}

Recall from Lemma \ref{E_measurability} that $\{\EE_{\ell+1},V=0\}$ is measurable with respect to $\sigma(\FF_{\ell-1},\widetilde\HH)$.
On the event $\{\EE_{\ell+1},V=0\}$ we can write 
\begin{align}\label{simplify_1}
G_{\ell+2}X(X')^{-1}&=G_{\ell+2}\prod_{a,b\in[k]: a<b}\left[\prod_{i=1}^{\ell}Z_{a,i},\prod_{j=1}^{\ell}Z_{b,j}\right]^{m_{ba}}  \varphi^{(\ell-1)}(\{ Z_{a,u}: a\in[k],u\leq \ell-1\}),
\end{align}
where $\varphi^{(\ell-1)}$ is defined analogously to $\varphi^{(\ell-2)}$ in \eqref{mathcalX}.
Since all terms in $\varphi^{(\ell-2)}$ are present in $\varphi^{(\ell-1)}$ we can express $
G_{\ell+2}\varphi^{(\ell-1)}:=G_{\ell+2} \varphi^{(\ell-2)}\cdot \tilde \varphi,
$
where $\tilde \varphi:=\tilde\varphi(\{ Z_{a,u}: a\in[k],u\leq \ell-1\})$
 is the polynomial that comes from excluding all the terms in $G_{\ell+2}\varphi^{(\ell-2)}$ from $G_{\ell+2}\varphi^{(\ell-1)}$. 

We can rewrite \eqref{simplify_1} as
\begin{align}\label{simplify_2}
G_{\ell+2}X(X')^{-1}&=G_{\ell+2}\prod_{a,b\in[k]: a<b}([Z_{a,1},Z_{b,\ell}]\cdot [Z_{a,\ell},Z_{b,1}])^{m_{ba}} \cdot \tilde\varphi\cdot \mathbf{X}^{(\ell+1)}.
\end{align}
Note that $G_{\ell+2}\tilde \varphi \in G_{\ell+1}$, as every $i$-fold commutator with $i\geq 3$ in $\tilde \varphi$ must involve a $Z_{a,\ell-1}$ for some $a\in[k]$. It follows from the proof of Lemma \ref{E_measurability} that $\mathbf{X}^{(\ell+1)}\in G_{\ell+1}$ on $\{\EE_{\ell+1},V=0\}$. Furthermore, it is easy to see that both $G_{\ell+2}\tilde \varphi$ and $G_{\ell+2} \mathbf{X}^{(\ell+1)}$ are measurable with  respect to $\sigma(\FF_{\ell-1},\widetilde\HH)$ as they only involve terms in $\{Z_{a,i}: 1\leq i\leq \ell-1\}$.

Since $Q_{\ell+1}=G_{\ell+1}/G_{\ell+2}$ is abelian, we can equivalently write \eqref{simplify_2} in terms of addition and obtain the desired expression.
\end{proof}

\bigskip
\noindent
\textit{Proof of Proposition \ref{span_size}.} For simplicity of notation, let $f\in G_{\ell+1}$ be such that 
$$G_{\ell+2}f:=G_{\ell+2}\sum_{a,b\in[k]: a<b}m_{ba}([Z_{a,1},Z_{b,\ell}]+[Z_{a,\ell},Z_{b,1}]).$$
It follows from Lemma \ref{simplification_X} that $G_{\ell+2}X(X')^{-1}=G_{\ell+2}f+G_{\ell+2} (\tilde\varphi \mathbf{X}^{(\ell+1)})$ on $\{\EE_{\ell+1},V=0\}$ and hence
\begin{align}\label{step1}
\nonumber \1\{\EE_{\ell+1},V=0\} \cdot \P(\EE_{\ell+2}|\FF_{\ell-1},\widetilde\HH)&= \1\{\EE_{\ell+1},V=0\}\cdot \P( G_{\ell+2} f+ G_{\ell+2} (\tilde\varphi \mathbf{X}^{(\ell+1)})=G_{\ell+2} |\FF_{\ell-1},\widetilde\HH)\\
&\leq \max_{g_{\ell}\in G_{\ell+1}}   \1\{\EE_{\ell+1},V=0\} \cdot\P(G_{\ell+2}f=G_{\ell+2}g_{\ell}|\FF_{\ell-1},\widetilde\HH)
\end{align}

Let $\GG_{\ell,\mathcal K^c}$ denote the $\sigma$-field generated by $\{Z_{a,\ell} :a\in [k]\backslash \mathcal{K}\}$. Observe that
\begin{align}\label{f_decomp}
\nonumber G_{\ell+2}f&=G_{\ell+2}\sum_{a,b\in[k]: a<b}m_{ba}([Z_{a,1},Z_{b,\ell}]+[Z_{a,\ell},Z_{b,1}])\\
\nonumber&=G_{\ell+2}\sum_{b\in \mathcal{K}, a\in [k]: a<b} m_{ba}[Z_{a,1},Z_{b,\ell}]+G_{\ell+2}\sum_{a\in \mathcal{K}, b\in[k]: a<b} m_{ba}[Z_{a,\ell},Z_{b,1}]\\
\nonumber &\quad +G_{\ell+2} \sum_{b\in \mathcal{K}^c, a\in [k]: a<b} m_{ba}[Z_{a,1},Z_{b,\ell}]+G_{\ell+2}\sum_{a\in \mathcal{K}^c, b\in[k]: a<b} m_{ba}[Z_{a,\ell},Z_{b,1}],\\
&=:G_{\ell+2}f_{unknown}+G_{\ell+2}f_{known}
\end{align}
where the second-to-last line is known under $\sigma(\GG_{\ell,\mathcal{K}^c},\FF_{\ell-1},\widetilde\HH)$ and thus will be denoted by $G_{\ell+2}f_{known}$.
It remains to consider the third-to-last line, i.e.,
\begin{align}\label{f_unknown}
\nonumber G_{\ell+2}f_{unknown}:=&G_{\ell+2}\sum_{b\in \mathcal{K}, a\in [k]: a<b} m_{ba}[Z_{a,1},Z_{b,\ell}]+G_{\ell+2}\sum_{a\in \mathcal{K}, b\in[k]: a<b} m_{ba}[Z_{a,\ell},Z_{b,1}]\\
\nonumber =&G_{\ell+2} \sum_{b\in \mathcal{K}}\left[\sum_{a\in[k]:a<b} m_{ba}Z_{a,1}-\sum_{a\in [k]: b<a}m_{ab}Z_{a,1}, Z_{b,\ell}\right]\\
=:&G_{\ell+2} \sum_{b\in \mathcal{K}}[\chi_b,Z_{b,\ell}],
\end{align}
where $\chi_b$ is as in Definition \ref{chi&psi}, i.e., $\chi_b$ is such that $G_2\chi_b=G_2 (\sum_{a\in[k]:a<b} m_{ba}Z_{a,1}-\sum_{a\in [k]: b<a}m_{ab}Z_{a,1})$. 

Lemma \ref{unif_span} shows that $G_{\ell+2} \sum_{b\in \mathcal{K}} [\chi_b,Z_{b,\ell}]$ is uniform over $\inprod{\{ G_{\ell+2}[\chi_b,g]: b\in \mathcal{K},g \in R_\ell\}}$, which is measurable with respect to $\sigma(\FF_{\ell-1},\widetilde\HH)$ since $(\chi_b)_{ b\in \mathcal{K}}$ are measurable with respect to $\sigma(\FF_1,\widetilde\HH)\subseteq \sigma(\FF_{\ell-1},\widetilde\HH)$. It turns out that if we further condition on the $\sigma$-field $\GG_{\ell,\mathcal{K}^c}$, one has that for any $g_\ell \in G_{\ell+1}$,
\begin{align}\label{step2}
\nonumber &\P(G_{\ell+2}f=G_{\ell+2}g_\ell|\FF_{\ell-1},\widetilde\HH,\GG_{\ell,\mathcal{K}^c})\\
\nonumber =&\P(G_{\ell+2} \sum_{b\in \mathcal{K}} [\chi_b,Z_{b,\ell}]+G_{\ell+2}f_{known}=G_{\ell+2}g_\ell|\FF_{\ell-1},\widetilde\HH,\GG_{\ell,\mathcal{K}^c})\\
\nonumber\leq &\max_{\tilde g_\ell \in G_{\ell+1}} \P(G_{\ell+2} \sum_{b\in \mathcal{K}} [\chi_b,Z_{b,\ell}]=G_{\ell+2}\tilde g_\ell|\FF_{\ell-1},\widetilde\HH,\GG_{\ell,\mathcal{K}^c})\\
\leq&|\{\inprod{G_{\ell+2}[\chi_b,g]: b\in \mathcal{K},g \in R_\ell\}}|^{-1}=|H_{\KK,\ell+1}|^{-1}.
\end{align}
Therefore, we can bound \eqref{step1} from above by 
\begin{align}\label{step3}
 \P(G_{\ell+2}f=G_{\ell+2}g_\ell|\FF_{\ell-1},\widetilde\HH)&=\E_{\GG_{\ell,\mathcal K^c}}\left[\P(G_{\ell+2}f=G_{\ell+2}g_\ell|\FF_{\ell-1},\widetilde\HH,\GG_{\ell,\mathcal{K}^c})\right]\leq |H_{\KK,\ell+1}|^{-1}.
\end{align}
where $\E_{\GG_{\ell,\mathcal K^c}}[\cdot]$ means we are taking the expectation over $\{Z_{a,\ell} :a\in [k]\backslash \mathcal{K}\}$. Combining \eqref{step1}, \eqref{step2} and \eqref{step3} then yields the conclusion of Proposition \ref{span_size}, i.e., 
$$
\1\{\EE_{\ell+1},V=0\}\cdot \P(\EE_{\ell+2}|\FF_{\ell-1},\widetilde\HH)\leq \1\{\EE_{\ell+1},V=0\}\cdot |H_{\KK,\ell+1}|^{-1}.$$
\qed

\subsubsection{Proof of Proposition \ref{mix_com}}\label{sub_proof_mix_com}
We begin by addressing the last term in \eqref{keystep_prop}. Recall the definition of $(\chi_b)_{b\in[k]}, (\psi_b)_{b\in[k]}$ from Definition \ref{chi&psi}. Recall that $\mathcal{K}$ is measurable with respect to $\widetilde\HH$.
\begin{lemma}\label{intermediate_E3}
Let $\sigma(\GG,\widetilde\HH):=\sigma((G_2\psi_b)_{b\in\tildeK}, (G_2\chi_b)_{b\in\tildeK},\widetilde\HH)$. Then 
$$\P(\EE_{3},\AA,V=0,\typ|\GG,\widetilde\HH)\leq  \frac{\1\{ \AA,V=0,\typ\}}{|Q_2|} \left(\frac{|G_{\mathrm{ab}}|}{|\inprod{\{G_2\psi_b: b\in \tildeK\}}|}\right)^{r_{2}}.$$ 
\end{lemma}

\begin{proof}
Observe that 
\begin{align}
\nonumber G_3X(X')^{-1}&=G_3 \sum_{a<b}m_{ba}[Z_{a,1},Z_{b,1}]=G_3\sum_{b\in[k]} \sum_{a<b}m_{ba}[Z_{a,1},Z_{b,1}]\\
\nonumber&=G_{3}\sum_{b\in \tildeK, a\in [k]: a<b} m_{ba}[Z_{a,1},Z_{b,1}]+G_{3}\sum_{b\in \tildeK^c, a\in\tildeK: a<b} m_{ba}[Z_{a,1},Z_{b,1}]\\
\nonumber &\quad +G_{3} \sum_{b\in \tildeK^c, a\in \tildeK^c: a<b} m_{ba}[Z_{a,1},Z_{b,1}]\\
\nonumber &=G_{3}\sum_{b\in\tildeK, a\in [k]: a<b} m_{ba}[Z_{a,1},Z_{b,1}]-G_{3}\sum_{b\in \tildeK, a\in\tildeK^c: a>b} m_{ab}[Z_{a,1},Z_{b,1}]\\
\nonumber&\quad +G_{3} \sum_{b\in \tildeK^c, a\in\tildeK^c: a<b} m_{ba}[Z_{a,1},Z_{b,1}]\\
&=:G_{3}\tilde f+G_{3}\tilde f_{c},
\end{align}
where $$G_{3}\tilde f_{c}:=G_{3} \sum_{b\in \tildeK^c, a\in\tildeK^c: a<b} m_{ba}[Z_{a,1},Z_{b,1}],$$
 and
\begin{align*}
G_3 \tilde f&:=G_{3}\sum_{b\in\tildeK, a\in [k]: a<b} m_{ba}[Z_{a,1},Z_{b,1}]-G_{3}\sum_{b\in\tildeK, a\in\tildeK^c: a>b} m_{ab}[Z_{a,1},Z_{b,1}]\\
&=G_3 \sum_{b\in \tildeK} \left[ \sum_{a\in[k]:a<b} m_{ba}Z_{a,1}+\sum_{a\in \tildeK^c:a>b} \hat m_{ba} Z_{a,1}, Z_{b,1}\right]=:G_3\sum_{b\in\tildeK}\left[ \psi_b, Z_{b,1}\right],
\end{align*}
where $\psi_b$ is as in Definition \ref{chi&psi}.

By Definition \ref{eventA}, on $\AA$ there exists  a set $\tildeK \subseteq [k]$ such that conditionally on $\widetilde \HH$, $(G_2\chi_b)_{b\in\tildeK}, (G_2\psi_b)_{b\in\tildeK}$ are independent from $(G_2Z_{b,1})_{b\in\tildeK}$. Hence, conditioning on $\sigma(\GG,\widetilde\HH)$, $(G_2Z_{b,1})_{b\in\tildeK}$ are i.i.d. uniform over $G_{\mathrm{ab}}$. In addition note that  $G_3\tilde f_c$  is  independent from $(Z_{b,1})_{b\in\tildeK}$ since $G_3\tilde f_c$ only involves $(Z_{b,1})_{ b\in\tildeK^c}$. Hence,
\begin{align*}
&\P(G_3X(X')^{-1}=G_3|\GG,\widetilde \HH)\\
=&\P\left( G_3\sum_{b\in\tildeK} [\psi_b,Z_{b,1}]=-G_3\tilde f_c\bigg|\GG,\widetilde \HH\right)\leq \max_{\tilde g \in G_2} \P\left( G_3\sum_{b\in\tildeK} [\psi_b,Z_{b,1}]=G_3\tilde g\bigg|\GG,\widetilde\HH\right)\\
\leq &   |\inprod{\{ G_3[\psi_b,g]: b\in\tildeK, g \in R_1\}}|^{-1}\leq \frac{1}{|Q_2|} \left(\frac{|G_{\mathrm{ab}}|}{|\inprod{\{G_2\psi_b: b\in \tildeK\}}|}\right)^{r_{2}},
\end{align*}
where the last line follows from Lemma \ref{unif_span} and Lemma \ref{size_estimate}.
\end{proof}

The last ingredient needed to  complete the proof of Proposition \ref{mix_com} is the following estimate whose proof will be delayed till Section \ref{proof_size_estimate}.
\begin{lemma}\label{size_estimate}
Let $(\chi_b)_{b\in[k]},(\psi_b)_{b\in [k]}$ be defined as in Definition \ref{chi&psi}.
Let $\AA$ and $\mathcal{K}$ be defined as in Definition \ref{eventA}.  Then
$$\1\{ \AA,V=0,\typ\}\cdot \E\left[\left( \frac{|G_{\mathrm{ab}}|}{|\inprod{ \{G_2\chi_b :b\in \mathcal{K}\}}|}\right)^{K-2}\bigg| \widetilde \HH\right]\leq \1\{ \AA,V=0,\typ\}\cdot  \exp\left(\sum_{i=1}^\infty \frac{r}{i^2}\right)$$
and 
$$\1\{ \AA,V=0,\typ\}\cdot \E\left[\left( \frac{|G_{\mathrm{ab}}|}{|\inprod{ \{G_2\psi_b :b\in \mathcal{K}\}}|}\right)^{K-2}\bigg| \widetilde \HH\right]\leq \1\{ \AA,V=0,\typ\}\cdot  \exp\left(\sum_{i=1}^\infty \frac{r}{i^2}\right).$$
\end{lemma}

\mn
\textit{Proof of Proposition \ref{mix_com}.} Recall that $\sigma(\GG,\widetilde\HH)=\sigma((\psi_b)_{b\in\tildeK}, (\chi_b)_{b\in\tildeK},\widetilde\HH)$ so that $(G_2\chi_b)_{b\in \KK}$ are measurable with respect to $\sigma(\GG,\widetilde\HH)$. Recall that $K=\sum_{\ell=2}^L r_\ell+2$. Noting that $\sigma(\GG,\widetilde\HH)\subseteq \sigma(\FF_1,\widetilde\HH)$, applying the tower property to \eqref{keystep_prop} gives
\begin{align*}
\nonumber&\P( \EE_{L+1},\AA,V=0,\typ|\GG,\widetilde\HH)\\
\leq & \left(\prod_{\ell=2}^{L-1} \frac{1}{ |Q_{\ell+1}|}\right) \E\left[ \left( \frac{|G_{\mathrm{ab}}|}{|\inprod{ \{G_2\chi_b :b\in \mathcal{K}\}}|}\right)^{K-2-r_2} \1\{\EE_{3},\AA,V=0,\typ\}\bigg|\GG,\widetilde\HH\right]\\
=&\left(\prod_{\ell=2}^{L-1} \frac{1}{ |Q_{\ell+1}|}\right)  \left( \frac{|G_{\mathrm{ab}}|}{|\inprod{ \{G_2\chi_b :b\in \mathcal{K}\}}|}\right)^{K-2-r_2}  \P(\EE_{3},\AA,V=0,\typ|\GG,\widetilde\HH)\\
\leq & \left(\prod_{\ell=1}^{L-1} \frac{1}{ |Q_{\ell+1}|}\right) \cdot \1\{\AA,V=0,\typ\}  \left( \frac{|G_{\mathrm{ab}}|}{|\inprod{ \{G_2\chi_b :b\in \mathcal{K}\}}|}\right)^{K-2-r_2} \left(\frac{|G_{\mathrm{ab}}|}{|\inprod{\{G_2\psi_b: b\in \tildeK\}}|}\right)^{r_{2}}
\end{align*}
where the last line follows from Lemma \ref{intermediate_E3}. Then
\begin{align*}
&\P( \EE_{L+1},\AA,V=0,\typ|\widetilde\HH)=\E\left[ \P( \EE_{L+1},\AA,V=0,\typ|\GG,\widetilde\HH)\big| \widetilde\HH \right]\\
\leq&\left(\prod_{\ell=1}^{L-1} \frac{1}{ |Q_{\ell+1}|}\right) \cdot \1\{ \AA,V=0,\typ\} \cdot \E\left[ \left( \frac{|G_{\mathrm{ab}}|}{|\inprod{ \{G_2\chi_b :b\in \mathcal{K}\}}|}\right)^{K-2-r_2} \left(\frac{|G_{\mathrm{ab}}|}{|\inprod{\{G_2\psi_b: b\in \tildeK\}}|}\right)^{r_{2}}\bigg|\widetilde\HH\right].
\end{align*}
We can bound the last term above by H\"{o}lder's inequality and Lemma \ref{size_estimate},
\begin{align*}
&\1\{ \AA,V=0,\typ\}\cdot \E\left[ \left( \frac{|G_{\mathrm{ab}}|}{|\inprod{ \{G_2\chi_b :b\in \mathcal{K}\}}|}\right)^{K-2-r_2} \left(\frac{|G_{\mathrm{ab}}|}{|\inprod{\{G_2\psi_b: b\in \tildeK\}}|}\right)^{r_{2}}\bigg|\widetilde\HH\right]\\
\leq &\1\{ \AA,V=0,\typ\}\cdot \E\left[\left( \frac{|G_{\mathrm{ab}}|}{|\inprod{ \{G_2\chi_b :b\in \mathcal{K}\}}|}\right)^{K-2}\bigg| \widetilde \HH\right]^{\frac{K-2-r_2}{K-2}} \E\left[\left( \frac{|G_{\mathrm{ab}}|}{|\inprod{ \{G_2\psi_b :b\in \mathcal{K}\}}|}\right)^{K-2}\bigg| \widetilde \HH\right]^{\frac{r_2}{K-2}}\\
\leq &\1\{ \AA,V=0,\typ\} \cdot \exp\left(\sum_{i=1}^\infty \frac{r}{i^2}\right). 
\end{align*}
That is,
$$\P( \EE_{L+1},\AA,V=0,\typ|\widetilde\HH)\leq  \1\{ \AA,V=0,\typ\}\cdot \prod_{\ell=1}^{L-1}\frac{1}{ |Q_{\ell+1}|}   \exp\left(\sum_{i=1}^\infty \frac{r}{i^2}\right).$$
Taking expectation over $\widetilde\HH$ on both sides and letting $C:=\exp\left(\sum_{i=1}^\infty \frac{r}{i^2}\right)$, we have 
$$|G|\cdot \P( \EE_{L+1}|  \AA,V=0,\typ)\leq |G|\cdot C\prod_{\ell=1}^{L-1}\frac{1}{ |Q_{\ell+1}|}=C|G_{\mathrm{ab}}|\ll e^h,$$
where $h$ is defined as in Definition \ref{def_h} and by definition $|G_{\mathrm{ab}}|\ll e^h$.
\qed

\subsubsection{Proof of Lemma \ref{size_estimate}}\label{proof_size_estimate}
\begin{proof}
The proof of the two inequalities is essentially the same. Without loss of generality we only prove the first inequality.

We can write 
\beq\label{decompose_R1}
 G_{\mathrm{ab}}=F_{p_1}\oplus\cdots\oplus F_{p_\gamma}
 \eeq
where $p_1,\dots,p_\gamma$ are distinct primes and $F_p$ is a Sylow $p$-subgroup of $G_{\mathrm{ab}}$. 
Each $F_{p_i}$ has the form
$$F_{p_i}=\oplus_{j=1}^{\beta_i} \ZZ_{p_i^{\alpha_{i,j}}}$$ 
and hence
\beq\label{ab_decomp}
 G_{\mathrm{ab}}\cong\oplus_{i=1}^\gamma \oplus_{j=1}^{\beta_i} \ZZ_{p_i^{\alpha_{i,j}}},
\eeq
where we can observe that $\max_{i\in [\gamma]} \beta_i\leq r$.

Since $G_2 Z_{a,1} \overset{iid}{\sim}\Unif(G_{\mathrm{ab}})$, for each $Z_{a,1}$ with $a\in [k]$, $G_2 Z_{a,1}$ can be represented in the following form
$$\oplus_{i=1}^\gamma \oplus_{j=1}^{\beta_i} Z^{(a)}_{i,j}$$
for a collection of independent random variables $\{ Z^{(a)}_{i,j}: 1\leq i\leq \gamma, 1\leq j\leq \beta_i\}$ such that $Z^{(a)}_{i,j}$ is uniform over $\ZZ_{p_i^{\alpha_{i,j}}}$. With slight abuse of notation, we will write
$$G_2 Z_{a,1}=\oplus_{i=1}^\gamma \oplus_{j=1}^{\beta_i} Z^{(a)}_{i,j}.$$

Based on this we can further write 
$$G_2\chi_{b}= \oplus_{i=1}^\gamma \oplus_{j=1}^{\beta_i} \left(\sum_{a\in[k]} \hat m_{ba}Z^{(a)}_{i,j}\right)=:\oplus_{i=1}^\gamma \oplus_{j=1}^{\beta_i}\chi^{(b)}_{i,j},$$
where $\chi^{(b)}_{i,j}:=\sum_{a\in[k]} \hat m_{ba}Z^{(a)}_{i,j}$ is an element in $\ZZ_{p_i^{\alpha_{i,j}}}$. Under the $\sigma$-field $\widetilde \HH$, the coefficients $\{\hat m_{ba}: a,b\in [k]\}$ are known. Hence the collections $\{\chi^{(b)}_{i,j}: b\in [k]\}$ are independent for different $(i,j)$'s,  and we can consider the generation of each subgroup $\ZZ_{p_i^{\alpha_{i,j}}}$ by  $\{ \chi^{(b)}_{i,j} :b\in \mathcal K\}$ separately, i.e.,
\begin{align}\label{group_size_decomp}
\E\left[\left( \frac{|G_{\mathrm{ab}}|}{|\inprod{ \{G_2\chi_b :b\in \mathcal{K}\}}|}\right)^{r_{\ell+1}}\bigg| \widetilde \HH\right]&\leq \prod_{i=1}^\gamma \prod_{j=1}^{\beta_i} \E\left[ \left(\frac{|\ZZ_{p_i^{\alpha_{i,j}}}|}{ |\inprod{ \chi^{(b)}_{i,j} :b\in \mathcal K}|}\right)^{K-2}\bigg| \widetilde \HH\right],
\end{align}
where as above $K=\sum_{\ell=2}^L r_\ell+2$.

For any $p_i\in\Gamma$, we will argue that on the event $\AA$ there exists a set of indices $\mathcal K_{p_i}\subseteq \tildeK$ such that for any $j\in [\beta_i]$, $(\chi^{(b)}_{i,j})_{ b\in \mathcal K_{p_i}}$ is a collection of $K$ i.i.d. uniform random variables over $\ZZ_{p_i^{\alpha_{i,j}}}$, so that one can simply apply Lemma \ref{iid_generation} to the right hand side of \eqref{group_size_decomp} to obtain the desired conclusion.

By Definition \ref{eventA}, on the event $\AA$ there exists a $K\times K$ submatrix $M$ of $(\hat m_{ba})_{b\in\tildeK,a\in[k]}$ such that $M_{p_i}$ has rank $K$. We will collect the column indices of $M$ into the set $\tildeK_{p_i,col}:=\{a_1,\dots,a_{K}\}$ (let $\tildeK_{p_i,col}=\emptyset$ if $\AA$ does not occur). Then we can define the $\sigma$-field $\GG_{i,j}:=\sigma( (Z^{(a)}_{i,j})_{a\notin \mathcal K_{p_i,col}},  \widetilde \HH)$, and express the vector $(\chi^{(b)}_{i,j})_{b\in \mathcal K_{p_i}}$ as a sum of two parts, one of which unknown under $\GG_{i,j}$ whereas the other known:
$$M (Z^{(a_1)}_{i,j},\dots, Z^{(a_{K})}_{i,j})+e( (Z^{(a)}_{i,j})_{a\notin \mathcal K_{p_i,col}}),$$ 
where $e(\cdot)$ is a function of $(Z^{(a)}_{i,j})_{a\notin \mathcal K_{p_i,col}}$ whose value is known under $\GG_{i,j}$. Proposition \ref{iff_condition} shows that $M$ is a surjective map from $(\ZZ_{p_i^{\alpha_{i,j}}})^{K}$ to $(\ZZ_{p_i^{\alpha_{i,j}}})^K$. Lemma \ref{uniform_chi} further implies that the $K$ entries of $M\cdot (Z^{(a_1)}_{i,j},\dots, Z^{(a_{K})}_{i,j})$ are i.i.d. uniform over $\ZZ_{p_i^{\alpha_{i,j}}}$. It is then straightforward to see that $(\chi^{(b)}_{i,j})_{b\in \mathcal K_{p_i}}$ are i.i.d. uniform over  $\ZZ_{p_i^{\alpha_{i,j}}}$ given $\GG_{i,j}$. That is, for any $\mathbf{x}:=(x_b)_{b\in \mathcal K_{p_i}}\in (\ZZ_{p_i^{\alpha_{i,j}}})^K$, we have
$$\1_{\AA}\cdot \P\left((\chi^{(b)}_{i,j})_{b\in \mathcal K_{p_i}}=\mathbf{x}|\widetilde \HH \right)=\1_{\AA}\cdot \E\left(\P\left((\chi^{(b)}_{i,j})_{b\in \mathcal K_{p_i}}=\mathbf{x}|\GG_{i,j}\right)\big|\widetilde \HH \right)=\1_{\AA}\cdot  (p_i^{\alpha_{i,j}})^{-K}.$$


Therefore, applying Lemma \ref{iid_generation} to the i.i.d. uniform $\{\chi^{(b)}_{i,j}:b\in \mathcal K_{p_i}\}$ gives
\begin{align*}
\textbf{1}_{\AA}\cdot \E\left[ \left(\frac{|\ZZ_{p_i^{\alpha_{i,j}}}|}{ |\inprod{ \chi^{(b)}_{i,j} :b\in \mathcal K}|}\right)^{K-2}\bigg| \widetilde \HH\right]&\leq \textbf{1}_{\AA}\cdot \E\left[ \left(\frac{|\ZZ_{p_i^{\alpha_{i,j}}}|}{ |\inprod{ \chi^{(b)}_{i,j} :b\in \mathcal K_{p_i}}|}\right)^{K-2}\bigg| \widetilde \HH\right]\\
&\leq \textbf{1}_{\AA}\cdot \exp\left( \frac{1}{p_i^2-1}\right).
\end{align*}
Combining all the subgroups $\{\ZZ_{p_i^{\alpha_{i,j}}}: i\in [\gamma], j\in [\beta_i]\}$, we can obtain using \eqref{group_size_decomp} that
\begin{align*}
\nonumber&\textbf{1}\{ \AA,V=0,\typ\}\cdot \E\left[\left( \frac{|G_{\mathrm{ab}}|}{|\inprod{ \{G_2\chi_b :b\in \mathcal{K}\}}|}\right)^{K-2}\bigg| \widetilde \HH\right] \\
\leq& \textbf{1}\{ \AA,V=0,\typ\}\cdot \prod_{i=1}^\gamma \prod_{j=1}^{\beta_i}\exp\left( \frac{1}{p_i^2-1}\right)\leq   \textbf{1}\{ \AA,V=0,\typ\}\cdot \exp\left(\sum_{i=1}^\infty \frac{r}{i^2}\right)
\end{align*}
as $\max_{i\in [\gamma]} \beta_i\leq r$.
\end{proof}

\subsection{Proof of Proposition \ref{A_complement}}\label{proof_A_complement}

In the proof of Proposition \ref{mix_com} we have conditioned on the ``good event" $\AA$ that guarantees the existence of a subset $\tildeK\subseteq[k]$ of indices that plays a critical role in the proof. In this section we aim to prove that indeed $\AA$ will occur with a sufficiently high probability.

\subsubsection{Regime: $\frac{\log|G_{\mathrm{ab}}|}{\log\log|G_{\mathrm{ab}}|}\lesssim k \ll \log|G_{\mathrm{ab}}|$}\label{relative_indep_regime}
Recall the definition of $h$ from Definition \ref{def_h}. In this regime we will work with the unconditional probability and prove the following stronger bound than that in Proposition \ref{A_complement}:
\beq\label{ub_A}
\P(\AA^c)\ll \frac{e^h}{|G|}=\frac{e^\omega}{|G_2|}.
\eeq
Indeed, with \eqref{ub_A} and $\P(\typ)\asymp 1$ this yields the statement in Proposition \ref{A_complement} as follows
$$|G|\cdot \P(\AA^c,V=0|\typ)\leq \frac{|G|\cdot  \P(\AA^c)}{\P(\typ)} \ll e^h.$$

First we will specify the choice of $\tildeK$ in this regime and verify this choice satisfies \eqref{def_K} in Definition \ref{eventA},
\begin{align}\label{tildeK}
\tildeK&=\{ b>k/2: \gcd(\{m_{ba}: a\leq k/2\}) \text{ and }|G_{\mathrm{ab}}| \text{ are coprime}\}.
\end{align}
Note this choice purely depends on $(m_{ba})_{ a,b\in [k],a<b}$ and hence is measurable with respect to $\widetilde \HH$. To see that $(G_2\psi_b)_{b\in\tildeK}$ are conditionally independent from $(G_2 Z_{b,1})_{b\in \tildeK}$ given $\widetilde \HH$, we observe that for any $b\in \tildeK$,
$$G_2\psi_b= G_2\sum_{a\leq k/2: a<b} \hat m_{ba}Z_{a,1}+G_2\left(\sum_{a>k/2: a<b} \hat m_{ba}Z_{a,1}+\sum_{a\in \tildeK^c: a>b} \hat m_{ba} Z_{a,1}\right),$$
where $G_2\sum_{a\leq k/2: a<b} \hat m_{ba}Z_{a,1} \sim \Unif (G_{\mathrm{ab}})$ due to the definition of $\tildeK$, see Lemma \ref{gcd_unif}. 

Furthermore, we can see that $G_2\sum_{a\leq k/2: a<b} \hat m_{ba}Z_{a,1}$ involves terms in $(Z_{a,1})_{a\leq k/2}$, whereas $G_2\left(\sum_{a>k/2:a<b} \hat m_{ba}Z_{a,1}+\sum_{a\in \tildeK^c:a>b} \hat m_{ba} Z_{a,1}\right)$ involves only terms in $(Z_{a,1})_{a>k/2}$ (since $b\in \KK$, in the second sum the condition $a>b$ leads to $a>k/2$). Therefore, conditioned on $\widetilde\HH$, we have that $G_2\sum_{a\leq k/2: a<b} \hat m_{ba}Z_{a,1}$ is independent from $G_2\left(\sum_{a>k/2:a<b} \hat m_{ba}Z_{a,1}+\sum_{a\in \tildeK^c:a>b} \hat m_{ba} Z_{a,1}\right)$. By the same reasoning we also have that $G_2\sum_{a\leq k/2: a<b} \hat m_{ba}Z_{a,1}$ is independent from $(G_2Z_{b,1})_{b\in \tildeK}$. With the information combined, we can see that conditionally on $\widetilde \HH$, for any $b\in\tildeK$, $G_2\psi_b$ is independent from $(G_2Z_{b,1})_{b\in \tildeK}$ and uniform over $G_{\mathrm{ab}}$. 

Noting that 
$$G_2\chi_b=G_2\sum_{a\in[k]} \hat m_{ba}Z_{a,1}=G_2\sum_{a\leq k/2} \hat m_{ba}Z_{a,1}+G_2 \sum_{a>k/2}  \hat m_{ba}Z_{a,1},$$
by the same reasoning as above we have that $G_2\sum_{a\leq k/2} \hat m_{ba}Z_{a,1}$ is independent from $G_2 \sum_{a>k/2}  \hat m_{ba}Z_{a,1}$ and $(G_2Z_{b,1})_{b\in \tildeK}$. Again by the definition of $\KK$ we have $G_2\sum_{a\leq k/2} \hat m_{ba}Z_{a,1} \sim \Unif (G_{\mathrm{ab}})$. That is, conditionally on $\widetilde \HH$, for any $b\in\tildeK$, $G_2\chi_b$ is independent from $(G_2Z_{b,1})_{b\in \tildeK}$ and uniform over $G_{\mathrm{ab}}$. Therefore we have verified the condition that $(G_2\psi_b)_{b\in \KK},(G_2 \chi_b)_{b\in \KK}$ are independent from $(G_2Z_{b,1})_{b\in\KK}$ in Definition \ref{eventA}.

Given the choice of $\tildeK$ in \eqref{tildeK}, for each $p\in\Gamma$ we can define $\AA_p(\tildeK)$ to be the event that there exists a $K\times K$ submatrix $M$ of $(\hat m_{ba})_{b\in\tildeK, a\in[k]}$ such that $ M_p:= M \mod p$ has rank $K$. Observe that $ \cap_{p\in\Gamma} \AA_p(\tildeK)\subseteq \AA$ and hence
$$\P(\AA^c)\leq \sum_{p\in\Gamma} \P( \AA_p(\tildeK)^c).$$

Recall that our goal is to show $\P(\AA^c)\ll \frac{e^h}{|G|}$, where $h$ is as in Definition \ref{def_h}.
As $|\Gamma|\lesssim \log|G|$, it suffices to prove $\max_{p\in\Gamma}  \P( \AA_p(\tildeK)^c)\ll \frac{e^h}{|G|\log|G|}$. In this regime we can prove a much stronger estimate.
\begin{prop}\label{ub_Ap}
When $\frac{\log|G_{\mathrm{ab}}|}{\log\log|G_{\mathrm{ab}}|}\lesssim k \ll \log|G_{\mathrm{ab}}|$, for any $C>0$,
$$ 
\max_{p\in\Gamma}  \P( \AA_p(\tildeK)^c)=o(|G|^{-C}).
$$
 \end{prop}
 In order to prove Proposition \ref{ub_Ap} we begin by discussing a useful property, which we refer to as the relative independence of the coefficient matrix $(\hat m_{ba})_{a,b\in[k]}$, that holds in the regime $\frac{\log|G_{\mathrm{ab}}|}{\log\log|G_{\mathrm{ab}}|}\lesssim k \ll \log|G_{\mathrm{ab}}|$. 


\mn
\textbf{Relative independence in $\hat m$.}
Recall that in this regime $t_*=t_0\asymp k  |G_{\mathrm{ab}}|^{2/k}$. We will consider $t\geq (1+\ep)t_0$ and hence $s:=t/k\gg1$, i.e., in both of the random walks $X(t)$ and $X'(t)$ each generator in $\{Z_a: a\in [k]\}$ should typically appear many times. As a result, we will see  certain ``relative independence" among a subset of terms in $\{ m_{ba}: a<b\}$. Before explaining the meaning of relative independence we need to define some notation.

We can view $\mathbf{X}=X(X')^{-1}$ as a sequence of generators, and express this sequence by $(\sigma_i,\eta_i)_{i\in [N+N']}$. In the following construction, we will obtain partial information on $(\sigma_i)_{i\in [N+N']}$ while conditioning on $\sigma((\eta_i)_{ i\in \mathbb{N}})$. That is, at this stage of construction we will treat $(\eta_i)_{i\in \mathbb{N}}$ as known.

Let $P\subseteq [k]$ be a subset of indices with size $d$, where $d$ is to be chosen later. 
We will denote by $\mathbf{X}_{P}$ the subsequence of $\mathbf{X}$ consisting of only generators in $\{ Z_{a}^{\pm 1}:a\in P\}$ and let $\NN_{P}(t)$ denote the length of $\mathbf{X}_{P}$, which follows a Poisson distribution with rate $\frac{2td}{k}$ by Poisson thinning.

Based on the subsequence $\mathbf{X}_{P}$, we will construct a collection of disjoint sets $\{B_{ab}: a,b\in P, a<b\}$ as follows:
\begin{enumerate}
\item Partition $\mathbf{X}_{P}$ into $\floor{\NN_{P}(t)/2}$ disjoint pairs, each pair consisting of the $2i-1$ and $2i$-th elements in $\mathbf{X}_{P}$. 
\item For $a,b\in[k]$ with $a<b$, we look at the $i$-th pair in $\mathbf{X}_{P}$ for each $1\leq i \leq \floor{\NN_{P}(t)/2}$. If the two corresponding generators in the $i$-th pair are $Z^{\eta_a}_a, Z^{\eta_b}_b$ for some $\eta_a,\eta_b\in\{\pm 1\}$ (where $\eta_a,\eta_b$ are known) regardless of the order in which they appear, then we record the pair of locations $(2i-1,2i)$ in the corresponding set $B_{ab}$. 
\end{enumerate}
Note that in the construction of $B_{ab}$, we have no knowledge on the exact order of the pair, that is, we only know the pair is either $(Z_a^{\eta_a},Z_b^{\eta_b})$ or $(Z_b^{\eta_b},Z_a^{\eta_a})$ with equal probability.

The subset $P$ is said to be \textit{nice} if $|B_{ab}| \geq \floor{\frac{t}{2dk}}$ for all $a,b\in P$ with $a<b$. As shown in the following lemma, $P$ is nice with high probability when we pick the right size $d$.

\begin{lemma}\label{nice_set}
The probability that the subset $P$ is nice is at least $1-Cd^2\exp(-\frac{t}{10dk})$ for some positive constant $C>0$ independent of $d$.
\end{lemma}
\begin{proof}

Since $\NN_{P}(t)\sim \mathrm{Poisson}(\frac{2td}{k})$, by a Chernoff bound argument
$$\P\left(\NN_{P}(t)<\frac{td}{k}\right)\leq  e^{-(c_0td)/k}$$
for some constant $c_0>0$.
Since each pair in the subsequence belongs to some $B_{ab}$ independently with probability $\frac{2}{d^2}$,
\begin{align*}
\P( \text{$P$ is not nice})&\leq d^2\P\left(|B_{12}|<\floor{\frac{t}{2dk}}\bigg| N_{P}(t)\geq \frac{td}{k}\right)+e^{-(c_0td)/k}\\
&\leq d^2 \P\left(\mathrm{Binomial}\left(\floor{\frac{td}{2k}}, \frac{2}{d^2}\right)<\floor{\frac{t}{2dk}}\right)+e^{-(c_0td)/k}\\
&\leq d^2\exp\left(-\frac{t}{10d k}\right)+e^{-(c_0td)/k}\leq Cd^2\exp\left(-\frac{t}{10dk}\right)
\end{align*}
 where the third line follows from a Chernoff bound.
\end{proof}

The reason to look at the collection $\{B_{ab}: a,b\in P, a<b\}$ is the following simple observation: when revealing the relative order of a pair of $\{Z_a^{\eta_a},Z_b^{\eta_b}\}$ in $B_{ab}$, if we have $(Z_b^{\eta_b},Z_a^{\eta_a})$ then it leads to an increment of $-\eta_a\eta_b$ on the value of $m_{ba}$; if instead we have $(Z_a^{\eta_a},Z_b^{\eta_b})$ then the resulting increment on $m_{ba}$ is 0. Taking account of the fact that for different pairs in $B_{ab}$, the corresponding $-\eta_a\eta_b$'s are i.i.d. random variables uniform in $\{ \pm 1\}$, we can expect diffusive behavior that ``smoothes out" the probability of $m_{ba}$ taking a certain value.  

Next we will  translate this intuition into a rigorous proof.  We first define the following quantity:
\beq
q(n):=\max_{x\in \ZZ} \P( \mathrm{Bin}(n,1/2)=x).
\eeq
It is well known that $q(n)$ is non-increasing with respect to $n$ and  there exists some $C>0$ such that $q(n)\leq Cn^{-1/2}$ for all $n\in \mathbb{N}$.

\begin{lemma}\label{lazyRW_modp}
Let $S_n\sim \mathrm{Bin}(n,1/2)$ for some $n\in \mathbb{N}$. For any prime $p$,
$$\max_{x\in \ZZ} \P( S_n= x \mod p)\leq \min\{2/p, 1/2\}+q(n).$$
\end{lemma}
\begin{proof}
The proof follows from the idea used in Lemma 2.14 of \cite{hermon2102cutoff}. It is easy to see that any distribution on $\mathbb{N}$ whose probability mass function is non-decreasing can be written as a mixture of $\Unif(\{1, ..., Y\})$ distributions, for different $Y\in \mathbb{N}$. 

For an even $n\in\mathbb{N}$ the binomial distribution $\mathrm{Bin}(n,1/2)$ is known to be unimodal, i.e., the mode $x_*:=\mathrm{argmax}_x \P(S_n=x)$ is unique. When $n\in \mathbb{N}$ is odd, the maximum of $\P(S_n=x)$ is achieved at two adjacent values. Without loss of generality, let $x_*:=\min\{\mathrm{argmax}_x \P(S_n=x)\}$.

 Letting $\bar S_n:=S_n-x_*$, it is easy to see the map $m\mapsto \P(|\bar S_n|=m)$  is non-increasing on $\mathbb{N}$ and hence we can write
$$|\bar S_n|\sim \Unif(\{1, ..., Y \}) \quad\text{conditional on} \quad \bar S_n\neq 0,$$
for some random variable  $Y\in \mathbb{N}$ whose law is insignificant to us. As a consequence, when $p>2$ for $x\in \ZZ$,
\begin{align*}
\P(\bar S_n=x \mod p| \bar S_n\neq 0)&\leq \P( |\bar S_n| +x\in  p\ZZ |\bar S_n\neq 0)+\P(|\bar S_n|-x \in p\ZZ | \bar S_n\neq 0)\\
&\leq 2\E( \floor{Y/p}/Y)\leq 2/p.
\end{align*}
That is, for any $x\in\ZZ$, $\P(S_n=x \mod p)=\P(\bar S_n=(x-x_*)\mod p)\leq 2/p$.

When $p=2$, by the same reasoning we have the following bound
$$\max_{x\in \{0,1\}}\P(\bar S_n=x \mod 2| \bar S_n\neq 0)\leq 1/2.$$
Therefore,
\begin{align*}
\max_{x\in \ZZ}\P( S_n= x \mod p)&=\max_{x\in \ZZ}\P( \bar S_n=x \mod p)\\
&\leq \P(\bar S_n=0)+\max_{x\in \ZZ}\P(\bar S_n=x \mod p|\bar S_n\neq 0)\leq q(n)+ \min\{2/p, 1/2\}.
\end{align*}

\end{proof}

Now we are ready to state the meaning of ``relative independence" in the terms of $(m_{ba})_{a,b\in P, a<b}$. The following lemma implies that conditioned on $P$ being nice, for any $(x_{ba})_{a,b\in P: a<b}$, the collection $(\textbf{1}\{m_{ab}=x_{ba}\})_{ a,b\in P, a<b}$ (respectively, $(\textbf{1}\{m_{ab}=x_{ba} \mod p\})_{ a,b\in P, a<b}$) is stochastically dominated by i.i.d. Bernoulli random variables with probability $q(\floor{\frac{t}{2dk}})$ (respectively,  $\min\{2/p, 1/2\}+q(\floor{\frac{t}{2dk}})$).

\begin{lemma}\label{relative_indep}
Let $q:=q(\floor{\frac{t}{2dk}})$. For $A\subseteq \{(a,b): a,b\in P, a<b\}$, let $\FF_{A^c}=\sigma( (m_{ba})_{ (a,b)\notin A, a,b\in P,a<b})$. Then for any $(x_{ba})_{(a,b)\in A}\in \ZZ^{|A|}$  we have that 
\beq\label{indep_m}
 \P( \cap_{(a,b)\in A} \{m_{ba}= x_{ba}\} |\FF_{A^c},P \text{ is nice})\leq q^{|A|},
\eeq
and furthermore, 
\beq
 \P( \cap_{(a,b)\in A} \{m_{ba}= x_{ba} \mod p\} |\FF_{A^c}, P \text{ is nice})\leq (\min\{2/p, 1/2\}+q)^{|A|}.
\eeq
\end{lemma}
\begin{proof}
For $a,b\in P$ with $a<b$, we first recall  from \eqref{mba} the definition of $m_{ba}$ and observe that  $m_{ba}$ is in fact determined purely by the subsequence $\mathbf{X}_P$. When referring to $\mathbf{X}_P$ as a subsequence, we essentially use the information given by $(\sigma_{P,i},\eta_{P,i})_{i\in [\NN_P]}$, where $\sigma_{P,i}$ denotes the index of the $i$-th term in $\mathbf{X}_P$ and $\eta_{P,i}$ its sign. Hence we can write
\beq\label{mba_subseq}
m_{ba}=-\sum_{j=1}^{\NN_P}\sum_{i<j} \eta_i \eta_j  \1\{\sigma_{P,j}=a, \sigma_{P,i}=b\} \quad \text{ for }  a,b\in P \text{ with }a<b.
\eeq

With a slight abuse of notation, let $B^c:=\{ i\in [\NN_P]: i \notin \cup_{a,b\in P,a<b} B_{ab}\}$ denote the collection of locations in $\mathbf{X}_P$ that are not recorded in any $B_{ab}$'s. Let $\GG_\eta:=\sigma(\{\eta_{P,i}:i\in\mathbb{N}\})$ (we reveal all the signs $\eta_{P,i}$ regardless of the value of $\NN_P$) and define 
$$\GG:=\sigma( \{B_{ab}: a,b\in P, a<b\}, (\sigma_i)_{i\in B^c}, \GG_\eta).$$

From now on, we will condition on the $\sigma$-field $\GG$. By conditioning on $\GG$, we are treating all the signs as known and revealing the indices of generators that are not in any $B_{ab}$'s. Hence, by \eqref{mba_subseq} $m_{ba}$ can be written as the sum of two parts, one independent from $\GG$ and the other known under $\GG$:
\beq\label{mba_decomp}
m_{ba}=\sum_{i\in\mathbb{N}:(2i-1,2i)\in B_{ab}} -\eta_{2i-1}\eta_{2i}  \1\{\sigma_{P,2i}=a, \sigma_{P,2i-1}=b\} + m^{known}_{ba}
\eeq
where $m_{ba}^{known}$ represents the cumulated increment from the pairs of $\{Z_a^\pm,Z_b^\pm\}$ that are not in $B_{ab}$, which is known under $\GG$. We can further observe that the collection of random variables $( \1\{\sigma_{P,2i}=a, \sigma_{P,2i-1}=b\})_{ (2i-1,2i)\in B_{ab}}$ from the first part of \eqref{mba_decomp} are i.i.d. $\mathrm{Bernoulli}(1/2)$ and independent from $\GG$.

Let $n_{ab}^\pm:=| \{ (2i-1,2i)\in |B_{ab}| : -\eta_{2i-1}\eta_{2i}=\pm 1\}|$.  (Note that $n^+_{ab}+n^-_{ab}=|B_{ab}|$ and $n_{ab}^\pm$ are measurable with respect to $\GG$.) We can then define two independent binomial random variables $Y^+_{ab}\sim \mathrm{Bin}(n_{ab}^+, 1/2)$, $Y^-_{ab}\sim \mathrm{Bin}(n^-_{ab},1/2)$ so that the increment on $m_{ba}$ resulted from revealing the orders of the pairs in $B_{ab}$ is given by $Y^+_{ab}-Y^-_{ab}=Y_{ab}-n^-_{ab}$, where $Y_{ab}:=Y^+_{ab}+(n^-_{ab}-Y^-_{ab}) \sim \mathrm{Bin}(|B_{ab}|,1/2)$. 
It follows from \eqref{mba_decomp} that, for any $x\in \ZZ$,
\begin{align}\label{binomial_approx}
\nonumber \P(m_{ba}=x|\GG,P \text{ is nice} )&\leq \max_{x'\in \ZZ} \P\left (\sum_{(2i-1,2i)\in B_{ab}} -\eta_{2i-1}\eta_{2i}  \1\{\sigma_{P,2i}=a, \sigma_{P,2i-1}=b\}=x' \bigg|\GG,P \text{ is nice}\right)\\
\nonumber &= \max_{x'\in \ZZ}  \P\left (Y^+_{ab}-Y^-_{ab}=x'\bigg|\GG,P \text{ is nice}\right)\\
\nonumber&= \max_{x''\in \ZZ}  \P\left (Y_{ab}=x''\bigg|\GG,P \text{ is nice}\right)\\
&= \max_{x''\in \ZZ}  \P\left( \mathrm{Bin}(|B_{ab}|,1/2)=x''|P \text{ is nice}\right)\leq q.
\end{align}
Hence we have the desired upper bound for a given pair of $a,b\in P, a<b$.
 $$\P(m_{ba}=x|P \text{ is nice} )\leq q.$$

Now note that for different pairs $(a,b)\neq (a',b')$ with $a, a', b,b'\in P$ and $a<b,a'<b'$, by our construction their corresponding $B_{ab}$ and $B_{a'b'}$ are disjoint. Hence $(Y_{ab})_{a,b\in P, a<b}$ is a collection of independent random variables. Carrying out a calculation similar to \eqref{binomial_approx} using $(Y_{ab})_{a,b\in P, a<b}$ leads to  
\begin{align*}
\max_{(x_{ba})_{(a,b)\in A}} \P(  \cap_{(a,b)\in A} \{m_{ba}= x_{ba}\} |\FF_{A^c}, \text{$P$ is nice})&\leq \max_{(y_{ba})_{(a,b)\in A}} \P( \cap_{(a,b)\in A} \{Y_{ab}=y_{ba}\}|\FF_{A^c}, \text{$P$ is nice}) \\
&\leq  q^{|A|}.
\end{align*}
Similarly, by Lemma \ref{lazyRW_modp}
\begin{align*}
&\max_{(x_{ba})_{(a,b)\in A}} \P( \cap_{(a,b)\in A} \{m_{ba}= x_{ba} \mod p\} |\FF_{A^c}, \text{$P$ is nice})\\
\leq &\max_{(y_{ba})_{(a,b)\in A}}  \P( \cap_{(a,b)\in A} \{Y_{ab}=y_{ba} \mod p\} |\FF_{A^c}, \text{$P$ is nice})\\
\leq & \max_{(y_{ba})_{(a,b)\in A}} \prod_{(a,b)\in A} \P( \mathrm{Bin}( \floor{\frac{t}{2dk}},1/2)= y_{ba} \mod p) 
\leq (\min\{2/p, 1/2\}+q)^{|A|}.
\end{align*}

\end{proof}

\noindent
\textit{Proof of Proposition \ref{ub_Ap}.} Recall that $K=\sum_{\ell=2}^L r_\ell+2$. Let $d\geq K$ be an even integer to be chosen later. Partition $\{a\in [k]: a\leq k/2\}$ into subsets $\{ \mathcal J_{1,i}: 1\leq i\leq \floor{ k/(2d)}\}$, each of size $d\geq K$, and omit the rest of the generators. Without loss of generality assume $d$ is even. Similarly we partition $\{b\in [k]: b>k/2\}$ into subsets  $\{ \mathcal J_{2,i}: 1\leq i\leq \floor{ k/(2d)}\}$ of size $d$. Let $P_i=\mathcal{J}_{1,i}\cup \mathcal J_{2,i}$. Since the arrivals of generators whose indices are in disjoint sets $\{P_i: 1\leq i\leq \floor{k/(2d)}\}$ are independent, we can try independently for $ \floor{ k/(2d)}$ times to search for a $K\times K$ submatrix $M$ in $(\hat m_{ba})_{b\in\tildeK\cap \JJ_{2,i}, a\in\JJ_{1,i}}$ such that $M_p$ has rank $K$.

For each $1\leq i \leq \floor{ k/(2d)}$ we perform the following trial. Let $\tildeK$ be as in \eqref{tildeK}. Since we will only be looking at $b\in \JJ_{2,i}\cap\tildeK$, we need to determine if $|\JJ_{2,i}\cap\tildeK|$ is large enough to begin with. 
If $|\JJ_{2,i}\cap\tildeK|\geq d/2$, we will look for  the desired $K\times K$ submatrix $M$ in $(\hat m_{ba})_{b\in \JJ_{2,i}\cap\tildeK, a\in \JJ_{1,i}}$. We will look at a batch of $K$ indices in $\JJ_{2,i}\cap\tildeK$ at a time, which will be denoted by $\{b_1,\dots,b_K\}$. Let $\JJ^{\mathrm{1st}}_{1,i}$ denote the first half of $\JJ_{1,i}$ and $\JJ^{\mathrm{2nd}}_{1,i}$ the second half. Our goal is to search for column indices in $\JJ^{\mathrm{2nd}}_{1,i}$, which will be labelled $a_1,\dots,a_K$, such that the submatrix $M$ induced by the rows $\{b_1,\dots,b_K\}$ and the columns $\{a_1,\dots,a_K\}$ of $\hat m$ satisfies the condition that $M_p$ has rank $K$. 

We will describe the steps of the $i$-th trial now: 
\begin{enumerate}
\item If $|\JJ_{2,i}\cap\tildeK|\geq d/2$ proceed to the search of submatrix $M$. Otherwise declare failure for this trial. 
\item We look for the first $a\in \JJ^{\mathrm{2nd}}_{1,i}$ such that $\hat m_{b_{1},a} \neq 0\mod p$ and set it as $a_{1}$. If there is no such $a$ declare this trial to be a failure.
\item The search will then proceed iteratively: for $u\geq 1$, given the choice of $\{a_{1},\dots,a_{u}\}$, we will look for the first $a\in \JJ^{\mathrm{2nd}}_{1,i}\backslash\{a_1,\dots,a_u\}$ such that the last row $(m_{b_{u+1},a_1},\dots, m_{b_{u+1},a_{u+1}}) \mod p$ is not in the vector space spanned by the previous rows $\{ (m_{b_i, a_1},\dots, m_{b_i, a_{u+1}}) \mod p: 1\leq i\leq u\}$. If there is no $a\in \JJ^{\mathrm{2nd}}_{1,i}\backslash\{a_1,\dots,a_u\}$ that works, then we declare failure for this trial.
\item The trial is a success if we have found $\{a_1,\dots,a_K\}$. 
\end{enumerate}

It remains to estimate the success probability for each trial, i.e., upper bound the failure probability in each step of the trial. 

\mn
\textbf{Step 1.} For any $b\in \JJ_{2,i}$, to upper bound the probability that $b\notin\tildeK$, we will reveal the orders in $\{ B_{ab}: a\in \JJ^{\mathrm{1st}}_{1,i}\}$. (By the relative independence in $\hat m$, we can still control the failure probability of our later search in $\JJ^{\mathrm{2nd}}_{1,i}$ given that $\{ B_{ab}: a\in \JJ^{\mathrm{1st}}_{1,i}\}$ has been revealed.)

Note that
\begin{align*}
\{ b\notin\tildeK\}&\subseteq \{  \gcd(\{m_{ba}: a\in \JJ^{\mathrm{1st}}_{1,i}\}) \text{ is not coprime with }|G_{\mathrm{ab}}|\}\\
&=\cup_{p\in\Gamma} \cap_{a\in  \JJ^{\mathrm{1st}}_{1,i}} \{ m_{ba}=0 \mod p\}.
\end{align*}

For all $a\in [k]$, let $N_a:=N_a(t)$ (and respectively $N'_a:=N'_a(t)$) denote the number of times the generator $Z_a^{\pm 1}$ appears in $X$ (and respectively $X'$). Recall that the arrivals of each generator can be viewed as an independent Poisson process with rate $1/k$. Letting 
$$\mathcal{C}:=\{ N_c\leq 2(t/k), N'_c\leq 2(t/k) \text{ for all }c\in \{b\}\cup \JJ^{\mathrm{1st}}_{1,i}\},$$
we have
$$\P(\mathcal{C}^c)\leq 2(|\JJ^{\mathrm{1st}}_{1,i}|+1)\P( N_a>2(t/k))\leq  2d\exp(-\Omega(t/k)).$$

Further observe that by \eqref{mba_subseq} $|m_{ba}| \leq (N_b+N'_b)(N_a+N'_a)$ and hence $\mathcal{C}\subseteq \cap_{a\in \JJ^{\mathrm{1st}}_{1,i}}\{ |m_{ba}|\leq  16(t/k)^2\}$. This implies that when $p>16(t/k)^2$, for $\mathcal{C}\cap (\cap_{a\in  \JJ^{\mathrm{1st}}_{1,i}} \{ m_{ba}=0 \mod p\})$ to occur the only possible case is when $m_{ba}=0$ for all $a\in \JJ^{\mathrm{1st}}_{1,i}$. That is,
\begin{align*}
\{ b\notin\tildeK\}&\subseteq\{ \{b\notin\tildeK\} \cap \mathcal{C}\} \cup \mathcal{C}^c\\
& \subseteq \left(\cup_{p\in\Gamma: p\leq 16(t/k)^2} \cap_{a\in  \JJ^{\mathrm{1st}}_{1,i}} \{ m_{ba}=0 \mod p\}\cap \mathcal{C}\right) \\
&\quad \cup (\cap_{a\in \JJ^{\mathrm{1st}}_{1,i}} \{m_{ba}=0\}) \cup \mathcal{C}^c
\end{align*}

By Lemma \ref{nice_set} it is easy to see $\P(\text{$P$ is nice})\geq 1/2$ when $t\gg kd\log d$, which will be guaranteed by our choice of $d$. Recall that  $q:=q(\floor{\frac{t}{2dk}})$. It follows from Lemma \ref{relative_indep} that 
\begin{align*}
\P(b\notin\tildeK| \text{$P$ is nice} )&\leq \sum_{ p\in \Gamma: p\leq 16(t/k)^2} \P( \cap_{a\in  \JJ^{\mathrm{1st}}_{1,i}} \{ m_{ba}=0 \mod p\}| \text{$P$ is nice})\\
&\quad+\P( \cap_{a\in  \JJ^{\mathrm{1st}}_{1,i}} \{ m_{ba}=0 \}| \text{$P$ is nice})+\P(\mathcal{C}^c|\text{$P$ is nice})\\
&\leq 16(t/k)^2 \left( \min\{ 2/p,1/2\}+q\right)^{d/2} + q^{d/2}+4d\exp(-\Omega(t/k))\lesssim \exp( -\Omega( d))=:\tilde q
\end{align*}
when $ \log(t/k)\ll d\ll (t/k)$. Conditioned on $P$ being nice, by the relative independence the collection $(\1\{ b\notin \tildeK\})_ {b\in \JJ_{2,i}}$ is dominated by i.i.d. Bernoulli random variables with probability $\tilde q$. Since $\JJ_{2,i}$ has size $d$, 
\beq\label{J2_good}
\P( |\JJ_{2,i}\cap\tildeK|<d/2|\text{$P$ is nice})\leq \P( \mathrm{Binomial}(d, \tilde q)>d/2)\leq 2^d(\tilde q)^{d/2}\leq \exp(-\Omega(d^2)).
\eeq

\mn
\textbf{Step 2 and 3: the search for $a_u$ with $1\leq u \leq K$.}
Since $|\JJ_{2,i}\cap\tildeK|\geq d/2$ we can try at least $\floor{ d/(2K)}$ batches of $K$ indices in $|\JJ_{2,i}\cap\tildeK|\geq d/2$. These trials are not exactly independent, but using the relative independence for $\{\hat m_{ba}: a,b\in P_i, a<b\}$ we can upper bound the probability that all trials are failures. 

We begin by estimating the failure probability for a single trial. Consider the $i$-th trial where we will look for the candidates $\{a_u :1\leq u\leq K\}$ in $ \JJ^{\mathrm{2nd}}_{1,i}$. Write $\{b_1,\dots,b_K\}$ as the corresponding row indices in this trial. The probability that we fail to find a $a_1$ is 
$\P(\hat m_{b_1,a}=0 \mod p \text{ for all }a\in \mathcal J_1|\text{$P$ is nice})$, which, by Lemma \ref{relative_indep}, satisfies
$$\P(\hat m_{b_1,a}=0 \mod p \text{ for all }a\in \mathcal J_1|\text{$P$ is nice})\leq (\min\{2/p, 1/2\}+q)^{\floor {d/2}}.$$

For $u\in [K]$, suppose we have found $\{a_1,\dots,a_u\}$ such that the matrix induced by $\{b_1,\dots,b_u\}$ and $\{a_1,\dots,a_u\}$ has linearly independent rows. If a candidate $a\in \JJ^{\mathrm{2nd}}_{1,i}\backslash\{a_1,\dots,a_u\}$ fails, it means that the new row 
$$(\hat m_{b_{u+1},a_1},\dots,\hat m_{b_{u+1},a_u}, \hat m_{b_{u+1},a})\mod p$$ 
is in the vector space spanned by previous $u$ rows $\{(\hat m_{b_{i},a_1},\dots,\hat m_{b_{i},a_u}, \hat m_{b_{i},a})\mod p: 1\leq i\leq u\}$. Since by assumption the matrix induced by $\{b_1,\dots,b_u\}$ and $\{a_1,\dots,a_u\}$ has independent rows, there exists a unique linear combination $(c_1,\dots,c_u)\in \ZZ_p^u$ such that 
$$(\hat m_{b_{u+1},a_1},\dots,\hat m_{b_{u+1},a_u})= \sum_{i=1}^u c_i (\hat m_{b_i,a_1} ,\dots, \hat m_{b_i,a_u}) \mod p,$$
and thus the last column needs to satisfy
$$\hat m_{b_{u+1},a}=\sum_{i=1}^u c_i \hat m_{b_i,a}\mod p.$$
Therefore, by Lemma \ref{relative_indep} the failure probability for a candidate $a\in\JJ^{\mathrm{2nd}}_{1,i}\backslash\{a_1,\dots,a_u\}$ is at most 
$$\P( \hat m_{b_{u+1},a}=\sum_{i=1}^u c_i \hat m_{b_i,a}\mod p|\text{$P$ is nice})\leq \min\{2/p, 1/2\}+q.$$

The relative independence in $\{\hat m_{ba}: a,b\in P_i, a<b\}$ implies that the probability of failing to find $a_{u+1}$ in the set $\JJ^{\mathrm{2nd}}_{1,i}\backslash\{a_1,\dots,a_u\}$ is at most $(\min\{2/p, 1/2\}+q)^{\floor{d/2}-u}$. Through a simple union bound we see that the batch $\{b_1,\dots,b_K\}$ fails with probability at most 
\beq\label{batch_failure}
\sum_{u=1}^K \P( \text{the search for $a_u$ fails}|\text{$P$ is nice})\leq \sum_{u=1}^K (\min\{2/p, 1/2\}+q)^{\floor{d/2}-r}\leq K(\min\{2/p, 1/2\}+q)^{d/2-K}.
\eeq

Combining all these failure probabilities, i.e., Lemma \ref{nice_set}, \eqref{J2_good}, \eqref{batch_failure}, and using the fact that $\{\hat m_{ba}: a,b\in P_i\}$ are independent for the disjoint index sets $\{P_i: 1\leq i\leq \floor{k/(2d)}\}$,  we have
\begin{align}\label{hatA_complement}
\nonumber\P(\AA_p^c)&\leq \P( \text{the 1st trial fails})^{\floor{k/(2d)}}\\
\nonumber&\leq\left( \P(P_1 \text{ is not nice})+  \P(|\JJ_{2,i}\cap\tildeK|<d/2|\text{$P$ is nice})+\left(K(q+\min\{2/p, 1/2\})^{(d/2)-K}\right)^{\floor{d/(2K)}}\right)^{\floor{k/(2d)}}\\
\nonumber&\leq \left( Cd^2\exp\left(-\Omega\left(\frac{t}{dk}\right)\right)+\exp(-\Omega(d^2))+\left(K(q+\min\{2/p, 1/2\})^{(d/2)-K}\right)^{\floor{d/(2K)}}\right)^{\floor{k/(2d)}}\\
&\leq 3^{k/(2d)} \left(\left(Cd^2\exp\left(-\Omega\left(\frac{t}{dk}\right)\right)\right)^{\floor{k/(2d)}}+ \exp(-\Omega(kd))+ \exp( -\Omega(kd))\right)
\end{align}
by the elementary inequality $(a+b+c)^n\leq 3^n (a^n+b^n+c^n)$ for $n\geq 1$.

Recall that $t\geq (1+\ep)t_*(k,G)\asymp k|G_{\mathrm{ab}}|^{2/k}$ in the currently considered regime. Our choice of $d$ should satisfy $1\ll d\ll k$ and $\frac{t}{dk}\gg 1$ so that $q=o(1)$. It was also required that $\log(t/k)\ll d\ll (t/k)$ right before \eqref{J2_good}. Furthermore, we will choose $d$ satisfying $\frac{t}{dk}\gg \log d$ and $\frac{t}{d^2}\gg \log|G|$ so that the first term in \eqref{hatA_complement} is $o(|G|^{-C})$ for any $C>0$. To control the second and third terms in \eqref{hatA_complement}, the choice of $d$ also needs to satisfy $kd \gg \log |G|$. 

We choose $d=|G_{\mathrm{ab}}|^{\delta/k}$ for some sufficiently small $\delta>0$, which satisfies all the conditions listed above. 
Consequently, for any $C>0$,
$$\P(\AA_p^c)=o(|G|^{-C}),$$
which leads to Proposition \ref{ub_Ap}.
\qed

\subsubsection{Regime: $k\gtrsim\log|G_{\mathrm{ab}}|$}

In this regime $s=t/k \lesssim 1$ and we no longer have the relative independence in $\hat m$, so we will take a different approach.

Recall that $W_a^{\pm }:=\sum_{i=1}^{N(t)} 1\{\sigma_i=a, \eta_i=\pm 1\}$ tracks the number of times $Z^{\pm 1}_a$ appears in $X(t)$. We will only look at these set of generators
$$\JJ=\{ a\in [k]: W_a^+ +W_a^-=1\},$$
that appear exactly once in $X:=X(t)$. In this regime, we will be conditioning on $\{V=0,\typ\}$, see Definition \ref{typical_event} for the definition of $\typ$ in this regime. Denote by $X|_\JJ$ the subsequence of $X$ that contains only generators in $\JJ$. For simplicity of notation, for any $a\in \JJ$ we will assume that $Z_a$ (instead of $Z_a^{-1}$) appears in $X$. Otherwise we can just relabel $Z_a^{-1}$ as a new $Z_a$. We also want to emphasize that only the order in which $(Z_a)_{a\in \JJ}$ appear in $X|_\JJ$ is important in the following argument, and the values of $(Z_a)_{a\in \JJ}$ are not.

By \eqref{mba_subseq} it is easy to see that on the event $\{V=0\}$ we have $m_{ba}\in \{-1,0,1\}$ for $a,b\in \JJ$ since $Z_a,Z_b$ only appear once in both $X$ and $X'$. Basically, if $\{Z_a,Z_b\}$ appears in the same order in $X'$ as they do in $X$ then we have $\hat m_{ba}=0$, otherwise $|\hat m_{ba}|=1$. 

Our goal is to look for an upper triangular submatrix $M$ of $\hat m$ that has full rank $K$. Indeed, since $M$ is upper triangular and all the entries of $M$ are in $ \{-1,0,1\}$, if $M$ has full rank $K$ then for all $p\in \Gamma$, $M_p$ also has full rank over the field $\mathbb{F}_p$. The proof is based on a combinatorial argument where we interpret $(|\hat m_{ba}|)_{a,b\in \JJ}$ in terms of the relative order of $Z_a, Z_b$ in $X|_\JJ$ and $X'|_\JJ$. 

Let $\{c_i: i\in \mathbb{N}\}$ be a set of distinct colors. Let $d$ be a positive integer to be determined later and $n:=\floor{|\JJ|/d}$. For $1\leq i\leq n$, we will color the $((i-1)d+1)$-th to the $id$-th generator that appears in $X|_\JJ$ as color $c_{n-i+1}$. In other words, the first $d$ generators in $X|_\JJ$ have color $c_n$, followed by $d$ generators of color $c_{n-1}$ and so on. Let $\JJ_i$ denote the set of generators in $\JJ$ that are colored $c_i$. Our coloring scheme implies that in the sequence $X|_\JJ$, for any $i>i'$, any generator belonging to $\JJ_i$ is in front of any generator belonging to $\JJ_{i'}$. In order to understand $(|\hat m_{ba}|)_{a,b\in \JJ}$ it remains to determine the relative orders of $\{ (Z_a,Z_b): a,b\in \JJ \text{ and $Z_a,Z_b$ are in different colors}\}$ in $X'|_\JJ$. 

Note that $X'|_\JJ$ has the distribution of a uniform permutation of $(Z_a)_{ a\in \JJ}$, which means we can construct this subsequence by inserting the generators to an existing sequence uniformly at random, one by one. Write $\JJ_i=(Z_{x_{i,1}},\dots,Z_{x_{i,d}})$ for $1\leq i\leq n$. In order to construct $X'|_\JJ$, we first sample a uniform permutation of  $(Z_a)_{ a\in \JJ_1}$ and denote it by $X'|_{\JJ_1}$. Without loss of generality, we can label them as $(Z_{x_{1,1}},\dots,Z_{x_{1,d}})$. Next we will insert the generators from $\JJ_2$ into this sequence, one by one. Once we are done with inserting the generators in $\JJ_2$, we proceed to insert the generators from $\JJ_3,\JJ_4$ and so on. Since all the generators are inserted at random and one by one, the resulting sequence will be a uniform permutation of the generators in $\JJ$.

Given the sequence $X'|_\JJ$, we can define a collection of good events $\{\CC_i: 2\leq i\leq n\}$. As an example, we will first define $\CC_2$. For any $l \in [d]$, if $Z_{x_{2,l}}$ is inserted between $Z_{x_{1,j}}$ and $Z_{x_{1,j+1}}$ for some $0\leq j\leq d$ (let $j=0$ if $Z_{x_{2,l}}$ is inserted in front of $Z_{x_{1,1}}$ whereas let $j=d$ if it is inserted behind $Z_{x_{1,d}}$), then 
$$|\hat m_{x_{1,j'},x_{2,l}}|=
\begin{cases}
1 &\text{ for }0\leq j'\leq j,\\
0 &\text{ for }j<j'\leq d.
\end{cases}
$$ 
We will let $\CC_{2}$ be the event that in the sequence $X'|_\JJ$ there are at least $K$ distinct pairs from the set $\{(Z_{x_{1,j}},Z_{x_{1,j+1}}): 0\leq j\leq d\}$  such that there is at least one generator from $\JJ_{2}$ that is  between them. The reason for this definition is that if $\CC_{2}$ occurs we can collect the first $K$ elements in
$$\{x_{1,j}: 0\leq j\leq d, \text{ there is some $Z_{x_{2,l}}\in \JJ_2$ inserted between $Z_{x_{1,j}}$ and $Z_{x_{1,j+1}}$}\}$$ 
as the row indices of $M$ and the corresponding $x_{2,l}$'s as the column indices. This choice leads to an upper triangular $K\times K$ matrix $M$ where the top right entries are $\{\pm 1\}$. Consequently, the induced matrix $M_p$ is still an upper triangular matrix that has full rank $K$. This argument can be most easily explained through an example.

\mn
\textbf{Example.} For simplicity we assume $X|_\JJ=(Z_1,Z_2,\dots,Z_6)$. Let $d=3$ so that there are 2 colors. In particular, $Z_1,Z_2,Z_3$ are in color 2 while $Z_4,Z_5,Z_6$ are in color 1. Suppose when constructing $X'|_\JJ$ we first sample a random permutation of $Z_4,Z_5,Z_6$ and get $(Z_4,Z_5,Z_6)$, and then inserting $Z_1,Z_2,Z_3$ to the current sequence, obtaining as a result 
$$X'|_\JJ=(Z_4, Z_1, Z_5, Z_2, Z_6,Z_3).$$
It is easy to see that as a consequence of inserting $Z_1$ (of color 2) in between $Z_4$ and $Z_5$ (of color 1) in $X'|_\JJ$, the relative order of  $(Z_1,Z_4)$  in $X'|_\JJ$ is different from that in $X|_\JJ$ and hence $|\hat m_{41}|=1$. Moreover, we can obtain an upper triangular $2\times 2$ submatrix

$$
\begin{pmatrix}
\hat m_{41} & \hat m_{42}\\
\hat m_{51} & \hat m_{52}
\end{pmatrix}
=
\begin{pmatrix}
\pm 1 & \pm 1\\
0 &  \pm 1
\end{pmatrix}.
$$

\bigskip
In general, for $2\leq i\leq n$ we can define $\CC_i$ to be the event that in the sequence $X'|_\JJ$ there are at least $K$ distinct pairs from the set 
$$\{(Z_{x_{i_1},l_1},Z_{x_{i_2},l_2}): i_1,i_2\leq i-1,l_1,l_2\in[d] \text{ and }Z_{x_{i_1},l_1},Z_{x_{i_2},l_2} \text{ are consecutive in }X'|_{\JJ_1\cup\cdots\cup\JJ_{i-1}} \}$$  such that there is at least one generator from $\JJ_{i}$ that is between them. If any of the events $\{\CC_i: 2\leq i\leq n\}$ occurred we would be able to find a $K\times K$ submatrix $M$ satisfying our condition and the set $\tildeK$ (from the definition of $\AA$ in Definition \ref{eventA}) by collecting the corresponding $K$ row indices of $M$. Hence, the corresponding column indices of $M$ are in $\tildeK^c$, a fact we shall now use to verify condition \eqref{def_K} from the definition of $\AA$. One can easily check that for any $b\in\tildeK$,
\begin{align*}
G_2\psi_b&=G_2\sum_{a\in[k]:a<b} m_{ba}Z_{a,1}+\sum_{a\in \tildeK^c:a>b} \hat m_{ba} Z_{a,1}\\
&=G_2 \sum_{a\in \tildeK^c} \hat m_{ba} Z_{a,1}+ G_2\sum_{a\in \tildeK: a<b} \hat m_{ba}Z_{a,1}
\end{align*}
where the first term is uniform in $G_{\mathrm{ab}}$ and independent from $(G_2Z_{b,1})_{ b\in\tildeK}$. Therefore, condition \eqref{def_K} is satisfied for this choice of $\tildeK$ and $\AA_p(\tildeK)$ occurs for all $p\in\Gamma$, i.e., $\AA$ occurs as long as $\cup_{i=2}^n \CC_i$ occurs. Therefore, it remains to upper bound
$$\P(\AA^c|V=0,\typ)\leq \P(\cap_{i=2}^n \CC^c_i).$$

For $1\leq i\leq n-1$, let $\GG_i$ be the $\sigma$-field that encodes relative orders of generators in $\JJ_1\cup\cdots\cup \JJ_i$ in $X'|_\JJ$. Then for $2\leq j\leq n$,
$$ \P(\cap_{i=2}^j \CC^c_i|\GG_{j-1})=\1\{ \cap_{i=2}^{j-1} \CC_i\}\cdot  \P(\CC_j|\GG_{j-1}).$$
The key is to observe that $\P(\CC^c_j|\GG_{j-1})=\P(\CC^c_j)$ is independent from $\GG_{j-1}$. Applying this iteratively gives that 
$$ \P(\cap_{i=2}^n \CC^c_i)=\prod_{i=2}^{n-1} \P(\CC_i^c).$$

We will calculate $\P(\CC_{i+1}^c)$ for $1\leq i\leq n-1$. Looking at the distribution of the subsequence $X'|_{\JJ_1\cup\cdots\cup\JJ_{i+1}}$ is equivalent to looking at the subsequence $X'|_{\JJ_1\cup\cdots\cup \JJ_i}$ and inserting the generators in $\JJ_{i+1}$ randomly and one by one to this subsequence. This perspective allows us to calculate $\P(\CC^c_i)$ via a multi-type urn scheme. 

The subsequence $X'|_{\JJ_1\cup\cdots\cup \JJ_i}$ has length $di$ and we are inserting the generators from $\JJ_{i+1}$. If we view the gap between two generators in the existing sequence $X'|_{\JJ_1\cup\cdots\cup \JJ_i}$ as a distinct type (also taking into account the gap before and behind all generators), then we would have $di+1$ types and there is one ball of each type in the urn. We will be conducting $d$ steps. At each step, we choose a ball randomly from the urn and place the ball together with a new ball of the same type back to the urn. It is not difficult to see this urn scheme is equivalent to inserting elements randomly to a sequence.

Our goal is to understand the probability of the event 
$$\{ \text{there are at most $K-1$ types with at least 2 balls after $d$ balls have been inserted}\},$$
which is the event equivalent to $\CC^c_{i+1}$ in the urn model and hence has the same probability.
This calculation can be simplified by first fixing the $di+1-(K-1)$ types each with at most one ball from all $di+1$ types and then group the rest $K-1$ types together to  form a new type, called type 0. Then we have a new urn model with $di+1-K$ types (coming from the  $di+1-(K-1)$ types and the new type 0) starting with $K-1$ balls of type 0 and one ball of each remaining type. The probability that we are only going to choose balls of type 0 (thus in the original urn model there are at most $K-1$ types that can potentially have at least 2 balls) is 
\begin{align*}
\prod_{j=0}^{d-1}  \frac{K-1+j}{di+1+j}=\frac{ (di)! (K+d-2)!}{(d(i+1))! (K-2)!}.
\end{align*}
Therefore,
$$\P(\CC_{i+1}^c)\leq { di+1 \choose K-1} \frac{ (di)! (K+d-2)!}{(d(i+1))! (K-2)!}\leq (di+1)^{K-1}\cdot \frac{ (di)! (K+d-2)!}{(d(i+1))! (K-2)!}.$$
Finally, we have
$$
\nonumber \P(\cap_{i=2}^{n}\CC_{i}^c)=\prod_{i=1}^{n-1} \P(\CC_{i+1}^c)\leq \left(\prod_{i=1}^{n-1}  (di+1)^{K-1} \right) \cdot \prod_{i=1}^{n-1} \frac{ (di)! (K+d-2)!}{(d(i+1))! (K-2)!},
$$
where the second product is a telescoping product and hence can be simplified to yield
\begin{align}\label{C_ub}
\nonumber \P(\cap_{i=2}^{n}\CC_{i}^c)&\leq \exp( K n\log(nd))\cdot \frac{d!}{(nd)!}\cdot \frac{ ((K+d-2)!)^{n-1}}{((K-2)!)^{n-1}}\\
\nonumber&=  \exp( K n\log(nd))\cdot  \frac{ (d!)^n}{(nd)!}\cdot \left( \frac{ (K+d-2)!}{d! (K-2)!}\right)^{n-1}\\
&\leq  \exp(K n\log(nd))\cdot (K+d-2)^{(K-2)(n-1)}\cdot  \frac{ (d!)^n}{(nd)!}
\end{align}

The key is to understand $\varphi(d):= \frac{ (d!)^n}{(nd)!}$. By stirling's formula we have 
$$\varphi(d) \lesssim  \frac{ (2\pi d)^{n/2-1} }{ n^{1/2}} \cdot \exp(-nd\log n).$$
We collect the first terms as
$$ \exp( K n\log(nd))\cdot (K+d-2)^{(K-2)(n-1)}\cdot \frac{ (2\pi d)^{n/2-1} }{ n^{1/2}}\leq \exp(C_K n\log(nd))$$
for some constant $C_K>0$. Recall that $nd\approx |\JJ|$. We will choose $d\gg1$ which implies that 
\beq
n\log(nd)\ll nd\log n, 
\eeq
and consequently, for arbitrarily small $\delta>0$, when $n$ is sufficiently large we have 
\beq\label{Ci_failure}
\P(\cap_{i=2}^{n}\CC_{i}^c)\leq \exp( -(1-\delta)nd\log n)
\eeq

Recall from \eqref{once_typ} and Definition \ref{typical_event} that when $k\gtrsim\log|G_{\mathrm{ab}}|$,  conditioning on $\typ$ ensures $|\JJ|\geq (1-\ep/2)te^{-t/k}\geq (1-\ep)t$. 

\begin{itemize}
\item In the regime $k\eqsim \lambda \log|G_{\mathrm{ab}}|\asymp \log|G|$, the above implies that $|\JJ|\asymp k$.  Since $nd\approx |\JJ|$ we have
$$\P(\cap_{i=2}^{n}\CC_{i}^c)\leq \exp( -(1-\delta)nd\log n)=o(|G|^{-1})$$ 
as long as $n$ is a sufficiently large constant.

\item For the regime $k\gg \log|G_{\mathrm{ab}}|$, we have $t/k\ll1 $. By typicality, $|\JJ|\geq (1-\ep/2)te^{-t/k}\geq (1-\ep)t$.
Recall that  we write $\rho=\frac{\log k}{\log\log|G_{\mathrm{ab}}|}$. 
\begin{itemize}
\item When $t_*=t_0$ and $t\geq (1+3\ep)t_0$, we can choose $n$ so that $nd\geq (1-\ep/2)|\JJ|\geq (1+\ep)t_0$.  Recall that in this regime $\frac{\rho}{\rho-1}\geq \frac{\log |G|}{\log |G_{\mathrm{ab}}|}$, i.e., $\frac{1}{\rho-1}\geq \frac{\log |G_2|}{\log |G_{\mathrm{ab}}|}$, and $t_0 \eqsim k\cdot \frac{1}{\kappa \log \kappa}$ where $\kappa=k/\log|G_{\mathrm{ab}}|$. 

Letting $1\ll d\ll t_0^{\ep/8}$ and $\delta=\ep/4$, the failure probability given by \eqref{Ci_failure} is at most
\begin{align*}
 \exp( -(1-\delta)nd\log n)&\leq \exp(-(1-\delta)(1+\ep) t_0\log (t_0/d))\leq \exp(-(1+\ep/2)t_0\log t_0)\\
&\leq  \exp\left(-(1+\ep/4)\frac{\log|G_{\mathrm{ab}}|}{\rho-1}\right)\leq |G_2|^{-(1+\ep/4)},
\end{align*}
where in the second to the last inequality we use the fact that 
$$t_0\log t_0 \eqsim \log|G_{\mathrm{ab}}| \cdot \frac{ \log\log|G_{\mathrm{ab}}|-\log\log(k/\log|G_{\mathrm{ab}|})}{\log(k/\log|G_{\mathrm{ab}}|)}=\frac{\log|G_{\mathrm{ab}}|}{\rho-1}(1-o(1)).$$
\item When $t_*=t_1=\log_k|G|$ and $t\geq (1+3\ep)t_1$, we have $nd\geq (1+\ep)t_1$ and similarly the failure probability is at most
$$ \exp( -(1-\delta)nd\log n)\leq \exp(-(1+\ep/2) t_1\log t_1)\leq \exp\left(-(1+\ep/4)\frac{\log|G|}{\rho}\right)=|G|^{-(1+\ep/4)/\rho},$$
where the last inequality holds because $\rho=\frac{\log k}{\log\log|G_{\mathrm{ab}}|} \eqsim \frac{\log k}{ \log\log|G|}$ and $t_1\log t_1=(1-o(1))t_1\log\log|G|$.
\end{itemize}
\end{itemize}

Therefore, we have proved that $\P(\AA^c|V=0,\typ)\ll e^h /|G|$ (see Definition \ref{def_h} for the value of $h$ in each regime)  and thus completed the proof of Proposition \ref{A_complement}.

\bibliographystyle{plain}
\bibliography{RW_nilpotent_group}

\end{document}